\documentclass[10pt]{article}

\usepackage{etex}

\usepackage{mathtools}

\usepackage{enumitem}

\usepackage{amsmath,amsthm}

\usepackage{csquotes}

\usepackage{todonotes}

\makeatletter
\let\orgdescriptionlabel\descriptionlabel
\renewcommand*{\descriptionlabel}[1]{
	\let\orglabel\label
	\let\label\@gobble
	\phantomsection
	\edef\@currentlabel{#1} 	
	\let\label\orglabel
	\orgdescriptionlabel{#1}
}
 											
\def\th@plain{
	\thm@notefont{}
	\itshape
}
\def\th@definition{
	\thm@notefont{}
	\normalfont
}

\g@addto@macro\th@remark{\thm@headpunct{}}
\g@addto@macro\th@definition{\thm@headpunct{}}
\g@addto@macro\th@plain{\thm@headpunct{}}
\makeatother

\usepackage{amssymb}

\usepackage{graphicx}

\usepackage{subfig}

\usepackage[final]{showkeys}

\usepackage{etoolbox}

\usepackage{wasysym}

\usepackage{mathrsfs}

\usepackage{mathpazo}

\usepackage[titletoc]{appendix}
\usepackage[doc]{optional}

\usepackage{soul}

\usepackage{colortbl,booktabs,sectsty,multirow}
\usepackage{xcolor}

\usepackage{cancel}

\usepackage{empheq}

\definecolor{myblue}{rgb}{.8, .8, 1}
  \newcommand*\mybluebox[1]{
    \colorbox{myblue}{\hspace{1em}#1\hspace{1em}}}

\usepackage[obeyspaces,hyphens,spaces]{url}

\usepackage{hyperref}
\hypersetup{
    colorlinks=true,
    linkcolor=blue,
    citecolor=blue,
    filecolor=magenta,
    urlcolor=cyan
}

\usepackage[
  open,
  openlevel=2,
  atend,
  numbered
]{bookmark}

\usepackage[capitalize, nameinlink, noabbrev]{cleveref}
\crefname{equation}{}{}
\crefname{chapter}{Chapter}{Chapters}
\crefname{item}{item}{items}
\crefname{figure}{Figure}{Figures}
\crefname{theorem}{Theorem}{Theorems}
\crefname{lemma}{Lemma}{Lemmas}
\crefname{proposition}{Proposition}{Propositions}
\crefname{corollary}{Corollary}{Corollarys}
\crefname{definition}{Definition}{Definitions}
\crefname{fact}{Fact}{Facts}
\crefname{example}{Example}{Examples}
\crefname{algorithm}{Algorithm}{Algorithms}
\crefname{remark}{Remark}{Remarks}
\crefname{note}{Note}{Notes}
\crefname{notation}{Notation}{Notations}
\crefname{case}{Case}{Cases}
\crefname{exercise}{Exercise}{Exercises}
\crefname{question}{Question}{Questions}
\crefname{claim}{Claim}{Claims}
\crefname{enumi}{}{}

\usepackage[top= 2cm, bottom = 2 cm, left = 2.2 cm, right= 2.2 cm]{geometry}

\usepackage{float}

\usepackage{pgf}

\usepackage{array}
\usepackage{tabu}

\allowdisplaybreaks

\numberwithin{equation}{section}

\theoremstyle{plain}
\newtheorem{theorem}{Theorem}[section]

\newtheorem{corollary}[theorem]{Corollary}
\newtheorem{fact}[theorem]{Fact}
\newtheorem{lemma}[theorem]{Lemma}
\newtheorem{proposition}[theorem]{Proposition}

\theoremstyle{definition}
\newtheorem{definition}[theorem]{Definition}
\newtheorem{example}[theorem]{Example}

\newtheorem{remark}[theorem]{Remark}

\newcommand{\aff}{\ensuremath{\operatorname{aff} \,}}

\newcommand{\spn}{\ensuremath{{\operatorname{span} \,}}}
\newcommand{\weakly}{\ensuremath{{\;\operatorname{\rightharpoonup}\;}}}

\newcommand{\Fix}{\ensuremath{\operatorname{Fix}}}
\newcommand{\Id}{\ensuremath{\operatorname{Id}}}

\newcommand{\dist}{\ensuremath{\operatorname{d}}}

\newcommand{\Pro}{\ensuremath{\operatorname{P}}}
\newcommand{\R}{\ensuremath{\operatorname{R}}}
\newcommand{\I}{\ensuremath{\operatorname{I}}}
\newcommand{\J}{\ensuremath{\operatorname{J}}}
\newcommand{\pa}{\ensuremath{\operatorname{par}}}
\newcommand{\card}{\ensuremath{\operatorname{card}}}

\newcommand{\CCO}[1]{CC{#1}}

\newcommand{\CC}[1]{CC_{#1}}

\providecommand{\norm}[1]{\lVert#1\rVert}

\providecommand{\innp}[1]{\langle#1\rangle}
\providecommand{\Innp}[1]{\Big\langle#1\Big\rangle}

\begin{document}

\title{ \sffamily  Circumcentered methods induced by
isometries\footnote{Dedicated to Professor Marco L\'opez on the occasion of his 70th birthday}}

\author{
         Heinz H.\ Bauschke\thanks{
                 Mathematics, University of British Columbia, Kelowna, B.C.\ V1V~1V7, Canada.
                 E-mail: \href{mailto:heinz.bauschke@ubc.ca}{\texttt{heinz.bauschke@ubc.ca}}.},~
         Hui\ Ouyang\thanks{
                 Mathematics, University of British Columbia, Kelowna, B.C.\ V1V~1V7, Canada.
                 E-mail: \href{mailto:hui.ouyang@alumni.ubc.ca}{\texttt{hui.ouyang@alumni.ubc.ca}}.},~
         and Xianfu\ Wang\thanks{
                 Mathematics, University of British Columbia, Kelowna, B.C.\ V1V~1V7, Canada.
                 E-mail: \href{mailto:shawn.wang@ubc.ca}{\texttt{shawn.wang@ubc.ca}}.}
                 }

\date{February 21, 2020}

\maketitle

\begin{abstract}
\noindent
Motivated by the circumcentered Douglas--Rachford method recently introduced
by Behling, Bello Cruz and Santos to accelerate the Douglas--Rachford
method, we study the properness of
the circumcenter mapping and the circumcenter method induced by isometries.
Applying the demiclosedness principle for circumcenter mappings, we present
weak convergence results for circumcentered isometry methods, which include
the Douglas--Rachford method (DRM) and circumcentered reflection methods as special
instances. We provide sufficient conditions for the linear convergence
of circumcentered isometry/reflection methods. 
We explore the convergence rate of
circumcentered reflection methods by considering the required number of
iterations and as well as run time as our performance measures.
Performance
profiles on circumcentered reflection methods, DRM and method of alternating
projections for finding the best approximation to the intersection of linear
subspaces are presented.

\end{abstract}

{\small

}

\section{Introduction} \label{sec:Introduction}
Throughout this paper, we assume that
\begin{empheq}[box = \mybluebox]{equation*}
\text{$\mathcal{H}$ is a real Hilbert space}
\end{empheq}
with inner product $\innp{\cdot,\cdot}$ and induced norm $\|\cdot\|$.
Denote the set of all nonempty subsets of $\mathcal{H}$
containing \emph{finitely many} elements by $\mathcal{P}(\mathcal{H})$.
Given $K \in \mathcal{P}(\mathcal{H})$, the circumcenter of $K$ is defined as
either empty set or the unique point $\CCO{(K)}$ such that $\CCO{(K)} \in
\aff (K)$ and $\CCO{(K)}$ is equidistant from all points in $K$, see
\cite[Proposition~3.3]{BOyW2018}.

Let $m \in \mathbb{N} \smallsetminus \{0\}$, and let
$T_{1}, \ldots, T_{m-1}, T_{m}$ be operators
from $\mathcal{H}$ to $\mathcal{H}$. Assume
\begin{empheq}[box = \mybluebox]{equation*}
\mathcal{S} =\{ T_{1}, \ldots, T_{m-1}, T_{m} \} \quad \text{with} \quad \bigcap^{m}_{j=1} \Fix T_{j} \neq \varnothing.
\end{empheq}
The associated set-valued operator $\mathcal{S}: \mathcal{H} \rightarrow \mathcal{P}(\mathcal{H})$ is defined by
\begin{empheq}[box = \mybluebox]{equation*}
(\forall x \in \mathcal{H}) \quad \mathcal{S}(x) :=\{ T_{1}x, \ldots, T_{m-1}x, T_{m}x\}.
\end{empheq}
The circumcenter mapping $\CC{\mathcal{S}}$ induced by $\mathcal{S}$ is defined by the composition of  $\CCO{}$ and $\mathcal{S}$, that is $(\forall x \in \mathcal{H})$ $\CC{\mathcal{S}}(x) =\CCO{} \left( \mathcal{S}(x)  \right)$. If $\CC{\mathcal{S}}$ is proper, i.e., $ (\forall x \in \mathcal{H})$, $\CC{\mathcal{S}}x \in \mathcal{H}$, then we are able to define the circumcenter methods induced by  $\mathcal{S}$ as
\begin{align*}
x_{0}=x, ~\mbox{and}~x_{k}=\CC{\mathcal{S}}(x_{k-1})=\CC{\mathcal{S}}^{k}x, ~\mbox{where}~k=1,2,\ldots.
\end{align*}

Motivated by Behling, Bello Cruz and Santos \cite{BCS2017}, we worked on circumcenters of finite set in Hilbert space in \cite{BOyW2018} and on the properness of circumcenter mappings in \cite{BOyW2018Proper}.  %This paper includes some generalized results and applications of \cite{BOyW2018Proper}.
For other recent developments on circumcentered isometry
methods, see also \cite{BCS2019Blockwise}, \cite{DHL2019}, \cite{SBLL2020}
and \cite{BOyW2019Linear}.
In this paper,
%generalizing some results in our paper \cite[Section~4.1]{BOyW2018Proper},
we study the properness of the circumcenter mapping induced by isometries, and
the circumcenter methods induced by isometries.
Isometry includes reflector associated with closed affine subspaces.
We provide convergence or even linear convergence results of the circumcentered isometry methods. In particular, for circumcentered reflection methods, we also offer some applications and evaluate  their linear convergence rate by comparing them with two classical algorithms, namely,
the  Douglas-Rachford method (DRM) and the method of alternating projections (MAP).

More precisely, our main results are the following:
\begin{itemize}
	\item \cref{thm:CCS:proper:NormPres:T} provides the properness of the  circumcenter
	mapping induced by isometries.
	\item \cref{theore:CM:WeakConver} presents a sufficient condition for the weak convergence of circumcentered isometry methods.
	\item \Cref{thm:CWP:line:conv,theorem:CCSLinearConvTSFirmNone} present  sufficient conditions for the linear convergence of circumcentered isometry methods in Hilbert space and $\mathbb{R}^{n}$, respectively.
	\item \cref{prop:CW:CCS:General:linearconv} takes advantage of the linear convergence of DRM to build the linear convergence of other circumcentered reflection methods.
\end{itemize}
%In \cite{BOyW2018Proper}, we generalize the C--DRM operator in \cite{BCS2017} to the circumcenter mapping induced by a set consisting of finitely many operators. We show the properness of the circumcenter mapping induced by reflectors in \cite[Theorem~4.3]{BOyW2018Proper} which is a generalization of \cite[Lemma~2]{BCS2017}.
\cref{thm:CCS:proper:NormPres:T} extends \cite[Theorem~4.3]{BOyW2018Proper}
from reflectors to isometries. Based on the demiclosedness principle for
circumcenter mappings built in \cite[Theorem~3.20]{BOyW2018Proper}, we obtain
the \cref{theore:CM:WeakConver}, which implies the weak convergence of the
DRM and the circumcentered reflection method, the main actor in
\cite{BCS2018}. Motivated by the role played by the Douglas--Rachford
operator in the proof of \cite[Theorem~1]{BCS2017}, we establish
\Cref{thm:CWP:line:conv} and \cref{prop:CW:CCS:General:linearconv}. As a
corollary of \cref{prop:CW:CCS:General:linearconv}, we observe that \cref{prop:CW:CCS:linear}
yields \cite[Theorem~1]{BCS2017}. Motivated by the role that the
firmly nonexpansive operator $A$ played in \cite[Theorem~3.3]{BCS2018} to
deduce the linear convergence of circumcentered reflection method in
$\mathbb{R}^{n}$, we obtain \cref{prop:LinFirQuaNonOperNorm} and
\cref{theorem:CCSLinearConvTSFirmNone}\cref{theorem:CCSLinearConvTSFirmNone:LineaConv}.
\cref{theorem:CCSLinearConvTSFirmNone}\cref{theorem:CCSLinearConvTSFirmNone:LineaConv}
says that some $\alpha$-averaged operators can be applied to construct linear
convergent methods, which imply the linear convergence of the circumcentered
isometry methods. As applications of
\cref{theorem:CCSLinearConvTSFirmNone}\cref{theorem:CCSLinearConvTSFirmNone:LineaConv},
\Cref{prop:GeneLinConBCS,prop:CwAverProLinRate,prop:W1reflection:conver}
display particular classes of circumcentered reflection methods being linearly
convergent.

The rest of the paper is organized as follows. In \cref{sec:Auxiliary}, we present various basic results for subsequent use.
Our main theory results start at \cref{sec:CircumMappinIsometries}.
Some results in \cite[Section~4.1]{BOyW2018Proper} are generalized in \cref{subsec:Genera5Sec} to deduce the properness of the circumcenter mapping induced by isometries. Thanks to the properness, we are able to generate the circumcentered isometry methods in \cref{sec:CircumcenterMethodIsome}.
In \cref{subsec:Convergence}, we focus on exploring  sufficient conditions for the
(weak, strong and linear) convergence of the circumcentered isometry methods.
In \Cref{sec:CircumMethodReflectors,sec:NumericalExperiment}, we consider the circumcentered reflection methods.
In \cref{sec:CircumMethodReflectors}, first, we display some particular linearly convergent circumcentered reflection methods.
Then  the circumcentered reflection methods
 are used to accelerate the DRM, which is then used  to find best approximation onto the intersection of finitely many linear subspaces.
Finally, in \cref{sec:NumericalExperiment}, in
order to evaluate the rate of linear convergence of the circumcentered
reflection methods, we use performance profile with performance measures on
both required number of iterations and run time to compare four
circumcentered reflection methods with DRM and MAP for solving the best
approximation problems associated with two linear subspaces with Friedrichs
angle taken in certain ranges.

We now turn to notation. Let $C$ be a nonempty subset
of $\mathcal{H}$. Denote the cardinality of $C$ by $\card (C)$. The
intersection of all the linear subspaces of $\mathcal{H}$ containing $C$ is
called the \emph{span} of $C$, and is denoted by $\spn C$; its closure is the
smallest closed linear subspace of $\mathcal{H}$ containing $C$ and it is
denoted by $\overline{\spn}~C$. $C$ is an \emph{affine subspace} of
$\mathcal{H}$ if $C \neq \varnothing$ and $(\forall \rho\in\mathbb{R})$ $\rho
C + (1-\rho)C = C$; moreover, the smallest affine subspace containing $C$ is
the \emph{affine hull} of $C$, denoted $\aff C$. An affine subspace $U$ is said to be \emph{parallel} to an affine subspace $
M $ if $U = M +a $ for some $ a \in \mathcal{H}$. Every affine subspace $U$ is parallel to a unique linear
subspace $L$, which is given by
$(\forall y \in U)$  $L := U - y = U - U$.
 For every affine subspace
$U$, we denote the linear subspace parallel to $U$ by $\pa U$. The orthogonal
complement of $C$ is the set $ C^{\perp} =\{x \in \mathcal{H}~|~ \innp{x,y}=0
~\text{for all}~y \in C\}. $
The best approximation operator (or projector) onto $C$ is denoted by
$\Pro_{C}$ while $\R_{C}:=2 \Pro_{C} -\Id$ is the reflector associated with $C$.
For two subsets $A$, $B$ of $\mathcal{H}$, the distance $\dist (A,
B)$ is $\inf\norm{A-B}$. A
sequence $(x_{k})_{k \in \mathbb{N}}$ in $\mathcal{H}$ \emph{converges
weakly} to a point $x \in \mathcal{H}$ if, for every $u \in \mathcal{H}$,
$\innp{x_{k},u} \rightarrow \innp{x,u}$; in symbols, $x_{k} \weakly x$. Let
$T: \mathcal{H} \rightarrow \mathcal{H}$ be an operator. The set of fixed
points of the operator $T$ is denoted by $\Fix T$, i.e., $\Fix T := \{x \in
\mathcal{H} ~|~ Tx=x\}$. $T$ is asymptotically regular if for each $x \in
\mathcal{H}$, $T^{k}x-T^{k+1}x\to 0$.
For other notation not explicitly defined here, we refer the reader to \cite{BC2017}.

\section{Auxiliary results} \label{sec:Auxiliary}

This section contains several results that will be useful later.

\subsection{Projections} \label{subsec:AffineSubspace }
%\begin{definition} {\rm \cite[page~4]{R1970}}
%	An affine subspace $U$ is said to be \emph{parallel} to an affine subspace $
%	M $ if $U = M +a $ for some $ a \in \mathcal{H}$.
%\end{definition}
%
%\begin{fact} {\rm \cite[Theorem~1.2]{R1970}}
%	\label{fac:AffinePointLinearSpace}
%	Every affine subspace $U$ is parallel to a unique linear
%	subspace $L$, which is given by
%	$(\forall y \in U)$  $L = U - y = U - U$.
%\end{fact}

%The following fact follows easily from the definitions.
%\begin{fact} \label{fact:SpanU1U2Plus}
%	Let $U_{1}, \ldots, U_{m}$ be linear subspaces in $\mathcal{H}$. Then
%	\begin{align*}
%	\spn (\cup^{m}_{i=1} U_{i}) = \sum^{m}_{i=1} U_{i}.
%	\end{align*}
%\end{fact}

%\begin{fact} {\rm \cite[Theorems~4.6(5) \& 4.5(8)]{D2012} }	\label{fact:UIntVIntUV}
%	Let $U,V$ be two closed linear subspace in $\mathcal{H}$. Then
%	\begin{align*}
%	U^{\perp} \cap V^{\perp} = (\overline{U+V})^{\perp}=(\overline{\spn} (U \cup V))^{\perp}.
%	\end{align*}
%\end{fact}
%\begin{proof}
%	The first equality is from \cite[Theorems~4.6(5) \& 4.5(8)]{D2012}. The second equality follows from \cref{fact:SpanU1U2Plus}.
%\end{proof}
\begin{fact} {\rm \cite[Proposition~29.1]{BC2017}} \label{fac:SetChangeProje}
Let $C$ be a nonempty closed convex subset of 
$\mathcal{H}$, and let $x \in \mathcal{H}$. Set $D :=z+C$, where $z \in
\mathcal{H}$. Then $\Pro_{D}x=z+\Pro_{C}(x-z)$.
\end{fact}

%\begin{fact} {\rm \cite[Proposition~29.15]{BC2017}} \label{fac:Proje:OthonormaBasis}
%	Suppose that $\{e_{i}\}_{i \in \I}$ is a countable orthonormal subset of the Hilbert space $\mathcal{H}$ and set $C=\overline{\spn}\{e_{i}\}_{i \in \I}$. Then
%	\begin{align*}
%	(\forall x \in \mathcal{H}) \quad \Pro_{C}x=\sum_{i\in I}\innp{x,e_{i}}e_{i}.
%	\end{align*}
%\end{fact}

%\begin{fact}{\rm \cite[Theorem~5.5]{D2012}} \label{fac:ProjecMonoNonexpan}
%	Let $K$ be a nonempty closed convex subset of $\mathcal{H}$. Then:
%	\begin{enumerate}
%		\item  \label{fac:ProjecMonoNonexpan:i} $\Pro_{K}$ is idempotent:
%		$
%		(\forall x \in \mathcal{H}) \quad \Pro_{K}(\Pro_{K}x) = \Pro_{K}(x).
%		$
%		Briefly, $\Pro_{K}^{2}=\Pro_{K}$.
%		\item \label{fac:ProjecMonoNonexpan:ii} $\Pro_{K}$ is firmly nonexpansive.
%	\end{enumerate}
%\end{fact}

\begin{fact} {\rm \cite[Theorems~5.8 and 5.13]{D2012}} \label{MetrProSubs8}
	Let $M$ be a closed linear subspace of $\mathcal{H}$. Then:
	\begin{enumerate}
		\item \label{MetrProSubs8:ii}$x=\Pro_{M}x+\Pro_{M^{\perp}}x$ for each $x \in \mathcal{H}$. Briefly, $\Id =\Pro_{M}+\Pro_{M^{\perp}}$.
		\item \label{MetrProSubs8:iv}$M^{\perp}=\{x \in \mathcal{H}~|~ \Pro_{M}(x)=0\}$ and $M=\{x \in \mathcal{H}~|~ \Pro_{M^{\perp}}(x)=0\}=\{x \in \mathcal{H}~|~ \Pro_{M}(x)=x \}$.
%		\item \label{MetrProSubs8:vi}$M^{\perp \perp}=M$.
%		\item \label{MetrProSubsLine13:i} $\Pro_{M}$ is a bounded linear operator and $\norm{\Pro_{M}}=1$ (unless $M=\{0\}$, in which case $\norm{\Pro_{M}}=0$).
	\end{enumerate}
\end{fact}

%\begin{fact} {\rm \cite[Lemma~9.2]{D2012}} \label{fac:proj:commu}
%	Let $M$ and $N$ be closed linear subspaces of $\mathcal{H}$. Assume  $M \subseteq N$ or $N \subseteq M$. Then $\Pro_{M}\Pro_{N}=\Pro_{N}\Pro_{M} =\Pro_{M \cap N} $.
%\end{fact}

\begin{fact}  {\rm \cite[Proposition~2.10]{BOyW2018Proper}}  \label{fact:PR}
	Let $C$ be a closed affine subspace of $\mathcal{H}$. Then the following hold:
	\begin{enumerate}
		\item \label{fact:PR:Affine} The projector $\Pro_{C}$ and the reflector $\R_{C}$ are affine operators.
		\item \label{fact:PR:Pythagoras} $(\forall x \in \mathcal{H})$  $(\forall v \in C)$ $\norm{x-\Pro_{C} x }^{2} + \norm{\Pro_{C} x -v}^{2}=\norm{x-v }^{2}$.
		\item \label{fact:PR:isometry}  $(\forall x \in \mathcal{H})$ $(\forall y \in \mathcal{H})$ $\norm{x-y}=\norm{\R_{C}x-\R_{C}y}$.
	\end{enumerate}
\end{fact}

\begin{lemma} \label{lem:othorg}
	Let $M := \aff \{x, x_{1}, \ldots, x_{n}\} \subseteq \mathcal{H}$, where $x_{1}-x, \ldots, x_{n}-x$ are linearly independent. Then for every $y \in \mathcal{H}$,
	\begin{align*}
	\Pro_{M}(y)=x+ \sum^{n}_{i=1} \innp{y-x, e_{i}} e_{i},
	\end{align*}
	where $(\forall i \in \{1, \ldots, n\}) \quad e_{i}=\frac{x_{i}-x - \sum^{i-1}_{j=1} \innp{x_{i}-x, e_{j}}e_{j}  }{\norm{x_{i}-x - \sum^{i-1}_{j=1} \innp{x_{i}-x, e_{j}}e_{j}  }}$.
\end{lemma}
\begin{proof}
	Since $x_{1}-x, \ldots, x_{n}-x$ are linearly independent, by the Gram-Schmidt orthogonalization process {\rm \cite[page~309]{MC2000}}, let $\big(\forall i \in \{1, \ldots, n\} \big)$ $e_{i} :=\frac{x_{i}-x - \sum^{i-1}_{j=1} \innp{x_{i}-x,e_{j}}e_{j}  }{\norm{x_{i}-x - \sum^{i-1}_{j=1} \innp{x_{i}-x,e_{j}}e_{j}  }}$,
	then $e_{1}, \ldots, e_{n}$ are orthonormal. Moreover
	\begin{align*}
	\spn \{e_{1}, \ldots, e_{n}\} =\spn \{ x_{1}-x, \ldots, x_{n}-x\}:=L.
	\end{align*}
	Since $M=x+L$, thus by \cref{fac:SetChangeProje}, we know
	$
	\Pro_{M}(y)=x+\Pro_{L}(y-x).
	$
	By \cite[Proposition~29.15]{BC2017}, we obtain that for every $z \in \mathcal{H}$,
	$
	\Pro_{L}(z)= \sum^{n}_{i=1} \innp{z, e_{i}} e_{i},
	$
	where $(\forall i \in \{1, \ldots, n\})$ $e_{i}=\frac{x_{i}-x - \sum^{i-1}_{j=1} \innp{x_{i}-x,e_{j}}e_{j}  }{\norm{x_{i}-x - \sum^{i-1}_{j=1} \innp{x_{i}-x,e_{j}}e_{j}  }}
	$.
\end{proof}

\subsection{Firmly nonexpansive mappings} \label{subsec:FirmlyNonexpansive}
\begin{definition} \label{defn:Nonexpansive} {\rm \cite[Definition~4.1]{BC2017}}
	Let $D$ be a nonempty subset of $\mathcal{H}$ and let $T:D \rightarrow \mathcal{H}$. Then $T$ is
	\begin{enumerate}
		\item \label{FirmNonex} \emph{firmly nonexpansive} if
		\begin{align*}
		(\forall x, y \in D) \quad \norm{Tx -Ty}^{2} + \norm{(\Id -T)x -(\Id -T)y}^{2} \leq \norm{x -y}^{2};
		\end{align*}
		\item \label{Nonex} \emph{nonexpansive} if it is Lipschitz continuous with constant 1, i.e.,
		\begin{align*}
		(\forall x, y \in D) \quad \norm{Tx -Ty} \leq \norm{x-y};
		\end{align*}
		\item \label{FirmQuasiNonex} \emph{firmly quasinonexpansive} if
		\begin{align*}
		(\forall x \in D)\quad (\forall y \in \Fix T) \quad \norm{Tx -y}^{2}+\norm{Tx-x}^{2} \leq \norm{x -y}^{2};
		\end{align*}
		\item \label{QuasiNonex} \emph{quasinonexpansive} if
		\begin{align*}
		(\forall x \in D)\quad (\forall y \in \Fix T) \quad \norm{Tx-y} \leq \norm{x-y}.
		\end{align*}
	\end{enumerate}
\end{definition}

%\begin{remark}{\rm \cite[page~70]{BC2017}} \label{rem:NonexpImplication}
%	Concerning \cref{defn:Nonexpansive}, by definition we have the implications:
%	\begin{align*}
%	\text{\cref{FirmNonex}} \Rightarrow  \text{\cref{Nonex}} \quad \text{and} \quad
%	\text{\cref{FirmNonex}} \Rightarrow  \text{\cref{FirmQuasiNonex}} \Rightarrow \text{\cref{QuasiNonex}}.
%	\end{align*}
%\end{remark}

\begin{fact} {\rm \cite[Corollary~4.24]{BC2017}} \label{fact:Nonexpan:Fix}
	Let $D$ be a nonempty closed convex subset of $\mathcal{H}$ and let $T : D \to \mathcal{H}$ be nonexpansive. Then $\Fix T$ is closed and convex.
\end{fact}

\begin{definition} {\rm \cite[Definition~4.33]{BC2017}} \label{defn:AlphaAverage}
	Let $D$ be a nonempty subset of $\mathcal{H}$, let $T: D \rightarrow \mathcal{H}$ be nonexpansive, and let $\alpha \in \, ]0,1[\,$. Then $T$ is \emph{averaged with constant $\alpha$}, or \emph{$\alpha$-averaged} for short, if there exists a nonexpansive operator $R: D \rightarrow \mathcal{H}$ such that $T=(1- \alpha) \Id +\alpha R$.
\end{definition}

\begin{fact} {\rm \cite[Proposition~4.35]{BC2017}}  \label{fact:averaged:character}
	Let $D$ be a nonempty subset of $\mathcal{H}$, let $T: D \rightarrow \mathcal{H}$ be nonexpansive, and let $\alpha \in \, ]0,1[\,$. Then the following are equivalent:
	\begin{enumerate}
		\item $T$ is $\alpha$-averaged.
		\item $(\forall x \in D)$ $(\forall y \in D)$ $ \norm{Tx -Ty}^{2} +\frac{1-\alpha}{\alpha} \norm{(\Id -T)x -(\Id -T)y}^{2} \leq \norm{x -y}^{2}$.
	\end{enumerate}	
\end{fact}	

%\begin{fact}  {\rm \cite[Remark~4.34(i)$\&$(iii)]{BC2017}} \label{fact:AverFirmNone}
%	Let $D$ be a nonempty subset of $\mathcal{H}$, let $T: D \rightarrow \mathcal{H}$.
%	\begin{enumerate}
%		\item \label{fact:AverFirmNone:nonexp} If $T$ is averaged, then it is nonexpansive.
%		\item \label{fact:AverFirmNone:firm} $T$ is firmly nonexpansive if and only if it is $\frac{1}{2}$-averaged.
%	\end{enumerate}
%\end{fact}

\begin{fact} {\rm \cite[Proposition~4.42]{BC2017}} \label{fact:AlphAvera}
	Let $D$ be a nonempty subset of $\mathcal{H}$, let $(T_{i})_{i \in \I}$ be a finite family of nonexpansive operators from $D$ to $\mathcal{H}$, let $(\omega_{i})_{i \in \I}$ be real numbers in $]0,1 ]$ such that $\sum_{i \in \I} \omega_{i}=1$, and let $(\alpha_{i})_{i \in \I}$ be real numbers in $]0,1[$ such that, for every $i \in \I$, $T_{i}$ is $\alpha_{i}$-averaged, and set $\alpha :=\sum_{i \in \I} \omega_{i} \alpha_{i}$. Then $\sum_{i \in \I} \omega_{i} T_{i}$ is $\alpha$-averaged.
\end{fact}

%\begin{fact} {\rm \cite[Proposition~4.47]{BC2017}} \label{fact:FixSumInters}
%	Let $D$ be a nonempty subset of $\mathcal{H}$, let $(T_{i})_{i \in \I}$ be a finite family of quasinonexpansive operators from $D$ to $\mathcal{H}$ such that $\cap_{i \in \I} \Fix T_{i} \neq \varnothing$, and let $(\omega_{i})_{i \in \I}$ be strictly positive real numbers such that $\sum_{i \in \I} \omega_{i}=1$. Then $\Fix \sum_{i \in \I} \omega_{i} T_{i}= \cap_{i \in \I} \Fix T_{i}$.
%\end{fact}

The following result is motivated by \cite[Lemma~2.1(iv)]{BCS2018}.
\begin{proposition} \label{prop:LinFirQuaNonOperNorm}
	Assume $\mathcal{H}=\mathbb{R}^{n}$. Let $T: \mathbb{R}^{n} \rightarrow \mathbb{R}^{n}$ be linear and $\alpha$-averaged with $\alpha \in \, ]0,1[\,$. Then $\norm{T\Pro_{(\Fix T)^{\perp}}} <1$.
\end{proposition}

\begin{proof} If $(\Fix T)^{\perp} = \{0\}$, then $\Pro_{(\Fix T)^{\perp}} =0$ and so $T\Pro_{(\Fix T)^{\perp}} =0$. Hence, the required result is trivial.
	
	Now assume $(\Fix T)^{\perp} \neq \{0\}$. By definition, $(\Fix T)^{\perp}$ is a closed linear subspace of $\mathbb{R}^{n}$. Since $T$ is $\alpha$-averaged, thus by \cref{fact:averaged:character},
	\begin{align} \label{eq:FirmNonexpansiveDef}
	(\forall x \in \mathbb{R}^{n}) \quad (\forall y \in \mathbb{R}^{n}) \quad \norm{Tx -Ty}^{2} +\frac{1-\alpha}{\alpha} \norm{(\Id -T)x -(\Id -T)y}^{2} \leq \norm{x -y}^{2}.
	\end{align}
	
	Since $(\Fix T)^{\perp} \neq \{0\}$, it is easy to see that
	\begin{align} \label{eq:TPperpOperNorm}
	\norm{T\Pro_{(\Fix T)^{\perp}}} = \sup_{\substack{x \in  \mathcal{H} \\ \norm{x} \leq 1} } \norm{T\Pro_{(\Fix T)^{\perp}} x} \stackrel{y= \Pro_{(\Fix T)^{\perp}} x}{=}  \sup_{\substack{y\in (\Fix T)^{\perp} \\  \norm{y} \leq 1} } \norm{Ty} = \sup_{\substack{y\in (\Fix T)^{\perp} \\  \norm{y} = 1} } \norm{Ty}.
	\end{align}
	Suppose to the contrary that $\norm{T\Pro_{(\Fix T)^{\perp}}} = 1$. 
	Then by \cref{eq:TPperpOperNorm} and by the Bolzano-Weierstrass Theorem, there exists $\overline{y} \in (\Fix T)^{\perp}$ with $\norm{\overline{y}}=1$ and $\norm{T \overline{y}}=1$.
	
	For every $x \in \mathbb{R}^{n}$, substituting  $y=\Pro_{\Fix T}x$ in \cref{eq:FirmNonexpansiveDef}, we get,
	\begin{align*}
	\norm{Tx-\Pro_{\Fix T}x}^{2}+\frac{1-\alpha}{\alpha}\norm{x -Tx}^{2} \leq \norm{x-\Pro_{\Fix T}x}^{2},
	\end{align*}
	which implies that
	\begin{align} \label{eq:prop:LinFirQuaNonOperNorm:AvoidFixT}
	(\forall x\not \in \Fix T) \quad \norm{Tx-\Pro_{\Fix T}x} < \norm{x-\Pro_{\Fix T}x}.
	\end{align}
	Since $\Fix T \cap (\Fix T)^{\perp} =\{0\}$ and since $\overline{y} \in (\Fix T)^{\perp}$ and $\norm{\overline{y}}=1$, so  $ \overline{y} \not \in \Fix T$. By \cref{MetrProSubs8}\cref{MetrProSubs8:iv}, $\overline{y} \in (\Fix T)^{\perp}$ implies that $\Pro_{\Fix T}( \overline{y} )=0$, thus substituting $x=\overline{y} $ in \cref{eq:prop:LinFirQuaNonOperNorm:AvoidFixT}, we obtain
	\begin{align*}
	1 = \norm{ T\overline{y}}= \norm{T\overline{y}-\Pro_{\Fix T}\overline{y}} < \norm{\overline{y}-\Pro_{\Fix T}\overline{y}}=\norm{\overline{y}}=1,
	\end{align*}
	which is a contradiction.
\end{proof}

%\subsection{Fej\'er monotone sequences}  \label{subsec:FejerMonotone}

\begin{definition} \label{defn:Fejer} {\rm \cite[Definition~5.1]{BC2017}}
	Let $C$ be a nonempty subset of $\mathcal{H}$ and let $(x_{k})_{k \in \mathbb{N}}$ be a sequence in $\mathcal{H}$. Then $(x_{k})_{k \in \mathbb{N}}$ is \emph{Fej\'er monotone} with respect to $C$ if
	\begin{align*}
	(\forall x \in C) \quad (\forall k \in \mathbb{N}) \quad \norm{x_{k+1}-x} \leq \norm{x_{k}-x}.
	\end{align*}
\end{definition}

%\begin{fact} {\rm  \cite[Lemma~2.45]{BC2017}} \label{fact:weakclusterExist}
%	Let $(x_{k})_{k \in \mathbb{N}}$ be a bounded sequence in $\mathcal{H}$. Then $(x_{k})_{k \in \mathbb{N}}$  possesses a weakly convergent subsequence.
%\end{fact}

\begin{fact} {\rm  \cite[Proposition~5.4]{BC2017}  }\label{fact:Fejer:Bounded}
	Let $C$ be a nonempty subset of $\mathcal{H}$ and let $(x_{k})_{k \in \mathbb{N}}$ be Fej\'er monotone with respect to $C$. Then $(x_{k})_{k \in \mathbb{N}}$ is bounded.
\end{fact}	

\begin{fact} {\rm \cite[Proposition~5.9]{BC2017}} \label{fact:FejerWeakCluster}
	Let $C$ be a closed affine subspace of $\mathcal{H}$ and let $(x_{k})_{k \in \mathbb{N}}$ be a sequence in $\mathcal{H}$. Suppose that $(x_{k})_{k \in \mathbb{N}}$ is Fej\'er monotone with respect to $C$. Then the following hold:
	\begin{enumerate}
		\item \label{fact:FejerWeakCluster:i} $(\forall k \in \mathbb{N})$ $\Pro_{C}x_{k}=\Pro_{C}x_{0}$.
		\item  \label{fact:FejerWeakCluster:ii}  Suppose that every weak sequential cluster point of $(x_{k})_{k \in \mathbb{N}}$ belongs to $C$. Then $x_{k} \weakly \Pro_{C}x_{0}$.
	\end{enumerate}
\end{fact}

%\begin{fact} {\rm \cite[Theorem~2.16(v)]{BB1996}} \label{fac:BasicPropertiesFejer}
%	Suppose $C$ is a closed, convex, nonempty set in $\mathcal{H}$ and the sequence $(x_{k})_{k\in \mathbb{N}}$ is Fej\'er monotone with respect to $C$. Then the following are equivalent:
%	\begin{enumerate}
%		\item $(x_{k})_{k\in \mathbb{N}}$ converges in norm to some point in $C$.
%		\item $(x_{k})_{k\in \mathbb{N}}$ has norm cluster points, all lying in $C$.
%		\item $(x_{k})_{k\in \mathbb{N}}$ has norm cluster points, one lying in $C$.
%	\end{enumerate}
%\end{fact}

\subsection{The Douglas--Rachford method} \label{subsec:DRM}

\begin{definition} {\rm \cite[page~2]{BCNPW2014}} \label{defn:DRO}
	Let $U$ and $V$ be closed convex subsets of $\mathcal{H}$ such that $U \cap V \neq \varnothing $.  The \emph{Douglas--Rachford splitting operator} is
	$	T_{V,U}:=\Pro_{V}(2\Pro_{U}-\Id)+\Id -\Pro_{U}$.
\end{definition}
It is well known that
\begin{empheq}[box = \mybluebox]{equation*}
T_{V,U} = \Pro_{V}(2\Pro_{U}-\Id)+\Id -\Pro_{U} = \frac{\Id +\R_{V}\R_{U}}{2}.
\end{empheq}

%\begin{fact} {\rm \cite[Proposition~4.13(i)]{BOyW2018Proper}} \label{fact:TT:CWhat}
% Let $U_{1}$ and $U_{2}$ be closed affine subspaces in $\mathcal{H}$ with $U_{1}
% \cap U_{2} \neq \varnothing$, 
%and let
% $T:=$ $T_{U_{2},U_{1}}$ be the Douglas--Rachford operator defined in
% \cref{defn:DRO}. 
% Set $\mathcal{S}= \{ \Id ,
% \R_{U_{2}}\R_{U_{1}}, \R_{U_{2}}\R_{U_{1}}\R_{U_{2}}\R_{U_{1}} \}$. 
% Then $\aff \{\Id, T, T^{2}\} =\aff \mathcal{S}$.
%\end{fact}

%\begin{fact}{\rm \cite[Proposition~3.6]{BCNPW2014}}  \label{fac:FixDRPFix}
% Let $U$ and $V$ be closed linear subspaces of $\mathcal{H}$ and set $T
% :=T_{V,U}$. Then $\Pro_{\Fix T} = \Pro_{U\cap V} + \Pro_{U^{\perp} \cap
% V^{\perp}}$.
%\end{fact}

\begin{definition} {\rm \cite[Definition~9.4]{D2012}} \label{defn:FredrichAngleClassical}
	The  \emph{Friedrichs angle} between two linear subspaces $U$ and $V$ is the angle $\alpha(U,V)$ between $0$ and $\frac{\pi}{2}$ whose cosine, $c(U,V) :=\cos \alpha(U,V)$, is defined by the expression
	\begin{align*}
	c(U,V)
	=  \sup \{ |\innp{u,v}| ~|~ u \in U \cap (U \cap V)^{\perp}, v \in V \cap (U \cap V)^{\perp}, \norm{u} \leq 1, \norm{v} \leq 1 \}.
	\end{align*}
\end{definition}

\begin{fact} {\rm \cite[Theorem~9.35]{D2012}} \label{fac:cFLess1}
Let $U$ and $V$ be closed linear subspaces of $\mathcal{H}$. 
Then the following  are equivalent:
	\begin{enumerate}
		\item $c(U,V) <1$;
		\item $U +V$ is closed.
	\end{enumerate}
\end{fact}

\begin{fact}{\rm \cite[Theorem~4.1]{BCNPW2014}}  \label{fac:DRCFn}
	Let $U$ and $V$ be closed linear subspaces of $\mathcal{H}$ and  $T :=T_{V,U}$ defined in \cref{defn:DRO}.  Let $n \in \mathbb{N} \smallsetminus \{0\}$ and let $x \in \mathcal{H}$. Denote the $c(U,V)$ defined in \cref{defn:FredrichAngleClassical} by $c_{F}$. Then
	\begin{align*}
	\norm{T^{n}x - \Pro_{\Fix T}x} \leq c^{n}_{F} \norm{x - \Pro_{\Fix T}x} \leq c^{n}_{F} \norm{x}.
	\end{align*}
\end{fact}

\begin{lemma} \label{lem:PUVFixSpnUV}
	Let $U$ and $V$ be closed linear subspaces of $\mathcal{H}$ and  $T :=T_{V,U}$.  Let $x \in \mathcal{H}$. Then
	\begin{align*}
	\Pro_{U \cap V} (x) = \Pro_{\Fix T}(x) \Leftrightarrow x \in \overline{\spn} (U \cup V) \Leftrightarrow x \in \overline{U+V}.
	\end{align*}
\end{lemma}

\begin{proof}
By \cite[Proposition~3.6]{BCNPW2014}, $\Pro_{\Fix T} = \Pro_{U\cap V} +
\Pro_{U^{\perp} \cap V^{\perp}}$. Moreover, by \cite[Theorems~4.6(5) \&
4.5(8)]{D2012}, we have $U^{\perp} \cap V^{\perp} =
(\overline{U+V})^{\perp}=(\overline{\spn} (U \cup V))^{\perp}$. Hence, by
\cref{MetrProSubs8}\cref{MetrProSubs8:iv}, we obtain that $\Pro_{U \cap V}
(x) = \Pro_{\Fix T}(x) \Leftrightarrow \Pro_{U^{\perp} \cap V^{\perp}}x=0
\Leftrightarrow \Pro_{(\overline{\spn} (U \cup V))^{\perp}}x=0
\Leftrightarrow x \in ( ( \overline{\spn} (U \cup V))^{\perp} )^{\perp}
=\overline{\spn} (U \cup V) = \overline{U+V} $.
%	Clearly,
%	\begin{align*}
%	\Pro_{U \cap V} (x) = \Pro_{\Fix T}(x) 
%& \Longleftrightarrow \Pro_{U^{\perp} \cap V^{\perp}}x=0 \quad (\text{by \cref{fac:FixDRPFix}})\\
%	& \Longleftrightarrow \Pro_{(\overline{\spn} (U \cup V))^{\perp}}x=0 \quad (\text{by \cref{fact:UIntVIntUV}})\\
%	& \Longleftrightarrow x \in ( ( \overline{\spn}  (U \cup V))^{\perp} )^{\perp} \quad (\text{by \cref{MetrProSubs8}\cref{MetrProSubs8:iv}}) \\
%	& \Longleftrightarrow x \in  \overline{\spn}  (U \cup V) \quad (\text{by \cref{MetrProSubs8}\cref{MetrProSubs8:vi}})\\
%	& \Longleftrightarrow x \in \overline{U+V} \quad (\text{by \cref{fact:SpanU1U2Plus}}).
%	\end{align*}
%	Therefore the required result is true.
\end{proof}

\begin{lemma} \label{lem:PFixTPWPUCapV}
	Let $U$ and $V$ be closed linear subspaces of $\mathcal{H}$ and  $T :=T_{V,U}$. Let $x \in \mathcal{H}$. Let $K$ be a closed linear subspace of $\mathcal{H}$ such that
	$
	U\cap V \subseteq K \subseteq  \overline{U+V}.
	$
	Then
	\begin{align*}
	\Pro_{\Fix T}\Pro_{K}x=\Pro_{U \cap V}\Pro_{K}x =\Pro_{U\cap V}x.
	\end{align*}
\end{lemma}

\begin{proof}
	Since $\Pro_{K}x \in K \subseteq   \overline{U+V}$, by \cref{lem:PUVFixSpnUV},
	\begin{align*}
	\Pro_{\Fix T}\Pro_{K}x=\Pro_{U \cap V}\Pro_{K}x.
	\end{align*}
	On the other hand, by assumption, $U\cap V \subseteq K$. 
	Hence, by \cite[Lemma~9.2]{D2012}, we get
	$\Pro_{U \cap V}\Pro_{K}x =\Pro_{K} \Pro_{U \cap V}x = \Pro_{U \cap V}x$.
\end{proof}

\subsection{Isometries} \label{subsec:Isometries}

\begin{definition} \label{defn:isometry} {\rm \cite[Definition~1.6-1]{Kreyszig1989}}
	A mapping $T: \mathcal{H} \rightarrow \mathcal{H}$ is said to be \emph{isometric} or an \emph{isometry} if
	\begin{align} \label{eq:T:normpreserving}
(\forall x \in \mathcal{H}) \quad (\forall y \in \mathcal{H}) \quad \norm{Tx -Ty} =\norm{x-y}.
	\end{align}
\end{definition}

Note that in some references, the definition of isometry is the linear operator satisfying \cref{eq:T:normpreserving}. In this paper, the definition of isometry follows  from \cite[Definition~1.6-1]{Kreyszig1989} where the linearity is not required.

%\begin{fact} {\rm \cite[Page~321]{MC2000}} \label{fact:isometry:orthogn}
%	The linear isometries on $\mathbb{R}^{n}$ are precisely the orthogonal matrices.
%\end{fact}	

\begin{corollary} \label{coro:averaNOTisometry}
Let $\alpha \in \,]0,1[\,$, and let $T: \mathcal{H} \to
\mathcal{H}$ be $\alpha$-averaged with $\Fix T \neq \varnothing$. Assume
that $T \neq \Id $. Then $T$ is not an isometry.
\end{corollary}

\begin{proof}
	Because $T \neq \Id $, $\Fix T \neq \mathcal{H}$. Take $x \in \mathcal{H} \smallsetminus \Fix T$. Then 
	\begin{align} \label{eq:coro:averaNOTisometry:norm}
	\norm{x -Tx} >0.
	\end{align}
	By assumption, $\Fix T \neq \varnothing$, take $y \in \Fix T$, that is, $y-Ty =0$. Because $T: \mathcal{H} \to \mathcal{H}$ is $\alpha$-averaged, by \cref{fact:averaged:character}, 
	\begin{align*}
	\norm{Tx -Ty}^{2} +\frac{1-\alpha}{\alpha} \norm{(\Id -T)x -(\Id -T)y}^{2} \leq \norm{x -y}^{2} & \Leftrightarrow \norm{Tx -Ty}^{2} +\frac{1-\alpha}{\alpha} \norm{x -Tx}^{2} \leq \norm{x -y}^{2} \\
	& \stackrel{\cref{eq:coro:averaNOTisometry:norm}}{\Rightarrow}  \norm{Tx -Ty} < \norm{x -y},
	\end{align*}
	which, by \cref{defn:isometry}, imply that $T$ is not isometric.
\end{proof}

\begin{definition} {\rm \cite[Page~32]{BC2017}}
	If $\mathcal{K}$ is a real Hilbert space and $T \in \mathcal{B}(\mathcal{H},\mathcal{K})$, then the \emph{adjoint}  of $T$ is the unique operator $T^{*}  \in \mathcal{B}(\mathcal{K}, \mathcal{H})$ that satisfies
	\begin{align*}
	(\forall x \in \mathcal{H}) \quad (\forall y \in \mathcal{K}) \quad \innp{Tx, y} =\innp{x, T^{*}y}.
	\end{align*}
\end{definition}

\begin{lemma} \label{lem:examples:normpreserving}
	\begin{enumerate}
		\item \label{lem:examples:normpreserving:R} Let $C$ be a closed affine subspace of $\mathcal{H}$. Then the reflector $\R_{C}:=2\Pro_{C}-\Id$ is isometric.
		\item \label{lem:examples:normpreserving:Trans} Let $a \in \mathcal{H}$. The translation operator $(\forall x \in \mathcal{H})$ $T_{a}x:=x+a$ is isometric.
		\item \label{lem:examples:normpreserving:TTStar} Let $T \in \mathcal{B}(\mathcal{H},\mathcal{H})$ and let $T^{*}$ be the adjoint of $T$. Then $T$ is isometric  if and only if $T^{*}T =\Id$.
		\item \label{lem:examples:normpreserving:Id}The identity operator is isometric.
	\end{enumerate}
\end{lemma}

\begin{proof}
	\cref{lem:examples:normpreserving:R}: The result follows from \cref{fact:PR}\cref{fact:PR:isometry}.
	
	\cref{lem:examples:normpreserving:Trans}: It is clear from the definitions.
	
	\cref{lem:examples:normpreserving:TTStar}: Assume that $T^{*}T =\Id$. Let $x \in \mathcal{H}$ and $y \in \mathcal{H}$. Now $	\norm{Tx-Ty}^{2}  = \innp{Tx-Ty,Tx -Ty}=  \innp{T(x-y), T(x-y)} = \innp{x-y,T^{*}T(x-y)} =  \innp{x-y,x-y} = \norm{x-y}^{2}$.	For the proof of the opposite direction, refer to \cite[Exercise~8 in Page~207]{Kreyszig1989}.
	
	\cref{lem:examples:normpreserving:Id}: The required result follows easily from  \cref{lem:examples:normpreserving:TTStar}.
\end{proof}
Clearly, the reflector associated with an affine subspace is affine and not necessarily linear. The translation operator $T_{a}$ defined in \cref{lem:examples:normpreserving}\cref{lem:examples:normpreserving:Trans} is not linear whenever $a \neq 0$.
\begin{lemma} \label{lem:Composi:NormPreserOpera}
	Assume $F:\mathcal{H} \rightarrow \mathcal{H}$ and $T: \mathcal{H} \rightarrow \mathcal{H}$ are isometric. Then the composition $F \circ T$ of $T$ and $F$ is isometric. In particular, the composition of finitely many isometries is an isometry.
\end{lemma}
\begin{proof}
	The first statement comes directly from the definition of isometry. Then by induction, we obtain the last assertion.
\end{proof}

\begin{lemma} \label{lem:normpreser:FixSet}
	Let $T: \mathcal{H} \rightarrow  \mathcal{H}$  be an isometry. Then the following  hold:
	\begin{enumerate}
		\item \label{lem:normpreser:FixSet:Nonexpansive} $T$ is nonexpansive.
		\item \label{lem:normpreser:FixSet:Fix} $\Fix T$ is closed and convex.
	\end{enumerate}
\end{lemma}

\begin{proof}
	\cref{lem:normpreser:FixSet:Nonexpansive}: This 
	is trivial from  \cref{defn:isometry} and  \cref{defn:Nonexpansive}\cref{Nonex}.
	\cref{lem:normpreser:FixSet:Fix}: Combine \cref{lem:normpreser:FixSet:Nonexpansive} and \cref{fact:Nonexpan:Fix}.
\end{proof}

\subsection{Circumcenter operators and circumcenter mappings} \label{subsec:CircumcenterOM}
In order to study circumcentered isometry methods, we require facts
on circumcenter operators and circumcenter mappings. 
Recall that  $\mathcal{P}(\mathcal{H})$ is the set of all nonempty subsets of $\mathcal{H}$ containing \emph{finitely many} elements.
By  \cite[Proposition~3.3]{BOyW2018}, we know that the following definition is well defined.
\begin{definition}[circumcenter operator]  {\rm \cite[Definition~3.4]{BOyW2018}} \label{defn:Circumcenter}
	The \emph{circumcenter operator} is
	\begin{empheq}[box=\mybluebox]{equation*}
	\CCO{} \colon \mathcal{P}(\mathcal{H}) \to \mathcal{H} \cup \{ \varnothing \} \colon K \mapsto \begin{cases} p, \quad ~\text{if}~p \in \aff (K)~\text{and}~\{\norm{p-y} ~|~y \in K \}~\text{is a singleton};\\
	\varnothing, \quad~ \text{otherwise}.
	\end{cases}
	\end{empheq}
	In particular, when $\CCO(K) \in \mathcal{H}$, that is, $\CCO(K) \neq \varnothing$, we say that the circumcenter of $K$ exists and we call $\CCO(K)$ the \emph{circumcenter} of $K$.
\end{definition}

%\begin{fact}[scalar multiples]  {\rm \cite[Proposition~6.1]{BOyW2018}} \label{fact:CircumHomoge}
%	Let $K \in \mathcal{P}(\mathcal{H})$ and $\lambda \in \mathbb{R} \smallsetminus \{0\}$.
%	Then
%	$\CCO(\lambda K)=\lambda \CCO(K)$.
%\end{fact}
%
%
%
%\begin{fact}[translations]  {\rm \cite[Proposition~6.3]{BOyW2018}} \label{fact:CircumSubaddi}
%	Let $K \in \mathcal{P}(\mathcal{H})$ and $y \in \mathcal{H}$. Then
%	$\CCO(K+y)=\CCO(K)+y$.
%\end{fact}

Recall that  $T_{1}, \ldots, T_{m-1}, T_{m}$ are operators from $\mathcal{H}$ to $\mathcal{H}$ with $\cap^{m}_{j=1} \Fix T_{j} \neq \varnothing$ and that
\begin{empheq}[box=\mybluebox]{equation*}
\mathcal{S}=\{ T_{1}, \ldots, T_{m-1}, T_{m} \} \quad \text{and} \quad (\forall x \in \mathcal{H}) \quad \mathcal{S}(x)=\{ T_{1}x, \ldots, T_{m-1}x, T_{m}x\} .
\end{empheq}

\begin{definition}[circumcenter mapping] {\rm \cite[Definition~3.1]{BOyW2018Proper}} \label{def:cir:map}
	The \emph{circumcenter mapping} induced by $\mathcal{S}$ is
	\begin{empheq}[box=\mybluebox]{equation*}
	\CC{\mathcal{S}} \colon \mathcal{H} \to \mathcal{H} \cup \{ \varnothing \} \colon x \mapsto \CCO(\mathcal{S}(x)),
	\end{empheq}
	that is, for every $x \in \mathcal{H}$, if the circumcenter of the set $\mathcal{S}(x)$ defined in \cref{defn:Circumcenter} does not exist, then  $\CC{\mathcal{S}}x= \varnothing $. Otherwise, $\CC{\mathcal{S}}x$ is the unique point satisfying the two conditions below:
	\begin{enumerate}
		\item $\CC{\mathcal{S}}x \in \aff(\mathcal{S}(x))=\aff\{T_{1}(x), \ldots,  T_{m-1}(x),  T_{m}(x)\}$, and
		\item $\left\{ \norm{\CC{\mathcal{S}}x - T_{i}(x)}~\big|~ i \in \{1, \ldots, m-1,m\} \right\}$ is a singleton, that is,
		\begin{align*}
		\norm{\CC{\mathcal{S}}x -T_{1}(x)}=\cdots =\norm{\CC{\mathcal{S}}x -T_{m-1}(x)}=\norm{\CC{\mathcal{S}}x -T_{m}(x)}.
		\end{align*}
	\end{enumerate}
	
	In particular, if for every $x \in \mathcal{H}$, $\CC{\mathcal{S}}x \in \mathcal{H}$, then we say the circumcenter mapping $\CC{\mathcal{S}}$ induced by $\mathcal{S}$ is \emph{proper}. Otherwise, we call the $\CC{\mathcal{S}}$ \emph{improper}.
\end{definition}

%\begin{fact}\label{fact:CW:ExistWellDefined} { \rm \cite[Proposition~3.6]{BOyW2018Proper} }
%	Suppose that for every $x \in \mathcal{H}$, there exists a point $p(x) \in
%	\mathcal{H}$ such that
%	\begin{enumerate}
%		\item  $p(x) \in \aff \{T_{1}x, \ldots, T_{m-1}x, T_{m}x\}$, and
%		\item  $\norm{p(x)-T_{1}x} =\cdots =\norm{p(x)-T_{m-1}x}=\norm{p(x)-T_{m}x}$.
%	\end{enumerate}
%	Then $\CC{\mathcal{S}}$ is proper and
%	$
%	(\forall x \in \mathcal{H})$ $\CC{\mathcal{S}}x=p(x).
%	$
%\end{fact}

\begin{fact} \label{fact:CW:FixPointSet} {\rm \cite[Proposition~3.10(i)\&(iii)]{BOyW2018Proper} }
	Assume $\CC{\mathcal{S}}$ is proper. Then the following  hold:
	\begin{enumerate}
		\item \label{fact:CW:FixPointSet:Basic} $\cap^{m}_{j=1} \Fix T_{j} \subseteq \Fix \CC{\mathcal{S}}$.
		\item \label{fact:CW:FixPointSet:Id} If $T_{1} =\Id$,
		then $\cap^{m}_{i=1} \Fix T_{i} =\Fix \CC{\mathcal{S}}$.
	\end{enumerate}
\end{fact}

%\textcolor{green}{
	To facilitate the notations, from now on, for any nonempty and finite family of operators $F_{1}, \ldots, F_{t}$,
\begin{empheq}[box=\mybluebox]{equation} \label{eq:Omega}
\Omega ( F_{1}, \ldots, F_{t}) :=  \Big\{ F_{i_{r}}\cdots F_{i_{2}}F_{i_{1}}  ~\Big|~ r \in \mathbb{N}, ~\mbox{and}~ i_{1}, \ldots,  i_{r} \in \{1, \ldots,t\}    \Big\}
\end{empheq}
which is the set consisting of all finite composition of operators from $\{F_{1}, \ldots, F_{t}\}$. We use the empty product convention, so for $r=0$, $F_{i_{0}}\cdots F_{i_{1}} = \Id$.
%}

\begin{proposition} \label{prop:Fix:CCS:EQ2}
	Let $t$ be a positive integer. Let $F_{1}, \ldots, F_{t}$ be $t$ operators  from $\mathcal{H}$ to $\mathcal{H}$. Assume that $\CC{\mathcal{S}}$ is proper. Assume that $\mathcal{S}$ is a finite subset of $\Omega ( F_{1}, \ldots, F_{t}) $ defined in \cref{eq:Omega} such that $\{ \Id, F_{1}, F_{2}F_{1}, \ldots, F_{t} F_{t-1} \cdots F_{2}F_{1} \} \subseteq \mathcal{S}$ or $\{ \Id, F_{1}, F_{2}, \ldots, F_{t}  \} \subseteq \mathcal{S}$. Then $\Fix \CC{\mathcal{S}} = \cap^{t}_{j=1} \Fix F_{j}$.
\end{proposition}

\begin{proof}
 Because each element of $\mathcal{S}$ is composition of operators from $\{
 F_{1}, \ldots, F_{t} \}$, and because $(\forall i \in \{1, \ldots, t \})$ $
 \cap^{t}_{j=1} \Fix F_{j} \subseteq \Fix F_{i}$, we obtain that 
 \begin{align} \label{eq:prop:Fix:CCS:EQ2}
 \cap^{t}_{j=1} \Fix
 F_{j} \subseteq \cap_{T \in \mathcal{S}} \Fix T = \Fix \CC{\mathcal{S}}, 
 \end{align}
 where the equality is from
 \cref{fact:CW:FixPointSet}\cref{fact:CW:FixPointSet:Id}.
	
	On the other hand, if $\{ \Id, F_{1}, F_{2}, \ldots, F_{t}  \} \subseteq \mathcal{S}$, then clearly $\cap_{T \in \mathcal{S}} \Fix T \subseteq  \cap^{t}_{j=1} \Fix F_{j}$.  Hence, by \cref{eq:prop:Fix:CCS:EQ2}, $\Fix \CC{\mathcal{S}} = \cap^{t}_{j=1} \Fix F_{j}$.
	
	Suppose that $\{ \Id, F_{1}, F_{2}F_{1}, \ldots, F_{t} F_{t-1} \cdots F_{2}F_{1} \} \subseteq \mathcal{S}$. Then for every $x \in \mathcal{H}$, by \cref{def:cir:map},
	\begin{align*}
	x \in  \Fix \CC{\mathcal{S}}  & \Rightarrow \norm{x -x} = \norm{x -F_{1}x} =\norm{x - F_{2}F_{1}x}= \ldots = \norm{x - F_{t} F_{t-1} \cdots F_{2}F_{1}x}\\
	& \Leftrightarrow x= F_{1}x= F_{2}F_{1}x=\ldots= F_{t} F_{t-1} \cdots F_{2}F_{1}x\\
	& \Leftrightarrow x= F_{1}x= F_{2}x=\ldots= F_{t-1}x= F_{t}x\\
	& \Leftrightarrow x \in \cap^{t}_{j=1} \Fix F_{j} ,
	\end{align*}
	which imply that $\Fix \CC{\mathcal{S}} \subseteq \cap^{t}_{j=1} \Fix F_{j}$.  Again, by \cref{eq:prop:Fix:CCS:EQ2}, $\Fix \CC{\mathcal{S}} = \cap^{t}_{j=1} \Fix F_{j}$.
	Therefore, the proof is complete.
\end{proof}
The following example says that the condition \enquote{$\{ \Id, F_{1},
F_{2}F_{1}, \ldots, F_{t} F_{t-1} \cdots F_{2}F_{1} \} \subseteq
\mathcal{S}$} in \cref{prop:Fix:CCS:EQ2} above is indeed critical. Clearly,
for each reflector $\R_{U}$, $\Fix \R_{U} = U$.
\begin{example} \label{exam:FixPint:IdRU321}
	Assume $\mathcal{H} = \mathbb{R}^{2}$. Set $U_{1}:= \mathbb{R}\cdot(1,0)$, $U_{2} :=\mathbb{R}\cdot(1,1)$ and $U_{3} :=\mathbb{R}\cdot(0,1)$. Assume $\mathcal{S} = \{\Id, \R_{U_{3}}\R_{U_{2}}\R_{U_{1}}\}$. Since $(\forall x \in U_{2})$ $\R_{U_{3}}\R_{U_{2}}\R_{U_{1}}x=x$, $\CC{\mathcal{S}} = \frac{1}{2} (\Id + \R_{U_{3}}\R_{U_{2}}\R_{U_{1}})$ and since the set of fixed points of linear and continuous operator is a linear space, thus $\cap^{3}_{i=1}U_{i} = \{(0,0)\} \varsubsetneqq U_{2} = \Fix \CC{\mathcal{S}}$.
\end{example}

\begin{fact}[demiclosedness principle for circumcenter mappings]{\rm \cite[Theorem~3.20]{BOyW2018Proper}} \label{fact:demi:Nonexpan}
	Suppose that $T_{1} =\Id$, that each operator in
	$\mathcal{S} = \{T_{1}, T_{2}, \ldots, T_{m}\}$ is nonexpansive,
	and that
	$\CC{\mathcal{S}}$ is proper. Then $\Fix \CC{\mathcal{S}} = \cap^{m}_{i=1}\Fix T_{i}$ and
	the demiclosedness principle holds for $\CC{\mathcal{S}}$, that is,
	\begin{align}
	\label{eq:fact:sadresult:assum}
	\left.
	\begin{array}{c}
	x_{k} \weakly  \overline{x}\\
	x_{k} - \CC{\mathcal{S}}x_{k} \to 0
	\end{array}
	\right\}
	\;\; \Rightarrow \;\;
	\overline{x} \in  \Fix \CC{\mathcal{S}}.
	\end{align}
\end{fact}

\begin{fact} {\rm \cite[Proposition~3.3]{BOyW2018Proper}} \label{fact:form:m2:Oper}
	Assume $m=2$ and $\mathcal{S}=\{T_{1}, T_{2}\}$. Then $\CC{\mathcal{S}}$ is proper. Moreover,
	$
	(\forall x \in \mathcal{H})$ $ \CC{\mathcal{S}}x = \frac{T_{1}x +T_{2}x}{2}.
	$
\end{fact}

The following result plays a critical role in our
calculations of circumcentered reflection methods in our numerical
experiments in \cref{sec:NumericalExperiment} below.

\begin{proposition}\label{prop:CW:Welldefined:Formula}
	Assume $\CC{\mathcal{S}}$ is proper. Let $x \in \mathcal{H}$. Set $d_{x} :=\dim \big( \spn\{ T_{2}x-T_{1}x, \ldots, T_{m}x-T_{1}x\} \big) $. Let $\widetilde{\mathcal{S}} := \{ T_{1}, T_{i_{1}}, \ldots, T_{i_{d_{x}}}\} \subseteq \mathcal{S}$ be such that \footnotemark
	\begin{align*}
	T_{i_{1}}x -T_{1}x, \ldots,T_{i_{d_{x}}}x-T_{1}x ~\text{is a basis of } \spn\{ T_{2}x-T_{1}x, \ldots, T_{m}x-T_{1}x\}.
	\end{align*}
	Then
	\begin{align}
	\CC{\mathcal{S}}x= \CC{\widetilde{\mathcal{S}}}x=
	T_{1}x+ \sum^{d_{x}}_{j=1} \alpha_{i_{j}}(x) (T_{i_{j}}x-T_{1}x)
	\end{align}
	where
	\begin{align*}
	\begin{pmatrix}
	\alpha_{i_{1}}(x)\\
	\vdots\\
	\alpha_{i_{d_{x}}}(x)\\
	\end{pmatrix}
	=\frac{1}{2} G( T_{i_{1}}x -T_{1}x, \ldots,  T_{i_{d_{x}}}x-T_{1}x)^{-1}
	\begin{pmatrix}
	\norm{T_{i_{1}}x -T_{1}x}^{2} \\
	\vdots\\
	\norm{T_{i_{d_{x}}}x-T_{1}x}^{2} \\
	\end{pmatrix},
	\end{align*}
	and  $G( T_{i_{1}}x -T_{1}x, \ldots,  T_{i_{d_{x}}}x-T_{1}x)$ is the Gram matrix of $ T_{i_{1}}x -T_{1}x, \ldots,  T_{i_{d_{x}}}x-T_{1}x$.
	\footnotetext{Note that if $\card \left(\mathcal{S}(x)\right)=1$, then $d_{x}=0$ and so $\CC{\mathcal{S}}x=T_{1}x$. }
\end{proposition}
\begin{proof}
	The desired result follows from \cite[Corollary~4.3]{BOyW2018}.
\end{proof}

%The following two lemmas can be easily obtained from \Cref{fact:CircumHomoge,fact:CircumSubaddi} respectively.

%\begin{lemma}[scalar multiples] \label{prop:CCS:Homo}
%	Assume that $T_{1}, \ldots, T_{m}$ are homogeneous, that is
%	\begin{align*}
%	(\forall x \in \mathcal{H}) \quad (\forall \lambda \in \mathbb{R}) \quad (\forall i \in \{1, \ldots,m\}) \quad T_{i}(\lambda x)=\lambda T_{i}(x).
%	\end{align*}
%	Then
%	\begin{align*}
%	(\forall x \in \mathcal{H}) \quad (\forall \lambda \in \mathbb{R} \smallsetminus \{0\}) \quad \CC{\mathcal{S}}(\lambda x)=\lambda \CC{\mathcal{S}}x.
%	\end{align*}
%\end{lemma}
%\begin{proof}
%	The required result comes from \cite[Proposition~6.1]{BOyW2018}.
%\end{proof}
%
%
%\begin{lemma}[quasitranslations] \label{prop:CCS:ModuAddi}
%	Assume that  $T_{1}, T_{2}, \ldots,T_{m}$ are quasitranslational, that is
%	\begin{align*}
%	(\forall i \in \{1, \ldots, m\}) \quad	(\forall x\in \mathcal{H}) \quad (\forall y \in \Fix T_{i})\quad T_{i}(x+y)=T_{i}x+y.
%	\end{align*} Then
%	\begin{align*}
%	(\forall x\in \mathcal{H}) \quad (\forall z \in \cap^{m}_{j=1} \Fix T_{j}) \quad \CC{\mathcal{S}}(x+z)=\CC{\mathcal{S}}x+z.
%	\end{align*}
%\end{lemma}
%
%\begin{proof}
%	The desired result is from \cite[Proposition~6.3]{BOyW2018}.
%\end{proof}

\section{Circumcenter mappings  induced by isometries} \label{sec:CircumMappinIsometries}
Denote $\I :=\{1, \ldots, m\}$. Recall that $(\forall i \in \I)$ $T_{i} : \mathcal{H} \rightarrow \mathcal{H}$ and that
\begin{empheq}[box=\mybluebox]{equation*}
\mathcal{S}=\{ T_{1}, \ldots, T_{m-1}, T_{m} \} \quad \text{with} \quad \cap^{m}_{j=1} \Fix T_{j} \neq \varnothing.
\end{empheq}
In the remaining part of the paper, we assume additionally that
\begin{empheq}[box=\mybluebox]{equation*}
(\forall i \in \I) \quad T_{i} : \mathcal{H} \rightarrow \mathcal{H} ~\text{is isometry}.
\end{empheq}
\subsection{Properness of circumcenter mapping induced by isometries} \label{subsec:Genera5Sec}

The following three results generalize Lemma 4.1,
Proposition 4.2 and Theorem 4.3 respectively in
\cite[Section~4]{BOyW2018Proper} from reflectors associated with affine
subspaces to isometries. 
In view of 
\cite[Theorem~3.14(ii)]{BOyW2019Linear}, we know that isometries are indeed
more general than reflectors associated with affine subspaces.
The proofs are similar to those given in \cite[Section~4]{BOyW2018Proper}.

\begin{lemma} \label{lem:EquDistS:T}
	Let $x \in \mathcal{H}$. Then
	\begin{align*}
	(\forall z \in \cap^{m}_{j=1} \Fix T_{j}) \quad (\forall i \in \{1,2,\ldots, m\}) \quad \norm{T_{i}x -z} = \norm{x-z}.
	\end{align*}
\end{lemma}

\begin{proof}
	Let $z \in \cap^{m}_{j=1}  \Fix T_{j}$ and  $ i \in \{1,2,\ldots, m\} $. Since $T_{i}$ is  isometric, and since $z \in \cap^{m}_{j=1}  \Fix T_{j} \subseteq \Fix T_{i}$,   thus $\norm{T_{i}x -z} = \norm{T_{i}x - T_{i} z} = \norm{x-z}$.
\end{proof}

\begin{proposition} \label{prop:CCS:proper:Normpreserv}
	For every $z \in \cap^{m}_{j=1} \Fix T_{j}$, and for every $x \in \mathcal{H}$, we have
	\begin{enumerate}
		\item \label{thm:CCS:proper:belong:normPres} $\Pro_{\aff (\mathcal{S}(x))}(z) \in \aff (\mathcal{S}(x))$, and
		\item  \label{thm:CCS:proper:EquaDistance:Normprese} $
		\big\{  \norm{\Pro_{\aff (\mathcal{S}(x))}(z) -Tx } ~\big|~ T \in \mathcal{S} \big\} $ is a singleton.
	\end{enumerate}
\end{proposition}

\begin{proof}
	Let $z \in \cap^{m}_{j=1}  \Fix T_{j}$, and let  $x \in \mathcal{H}$.
	
	\cref{thm:CCS:proper:belong:normPres}: Because $\aff (\mathcal{S}(x))$ is a nonempty  finite-dimensional affine subspace, we know $\Pro_{\aff (\mathcal{S}(x))}(z)$ is well-defined.
	Clearly, $\Pro_{\aff (\mathcal{S}(x))}(z) \in \aff (\mathcal{S}(x))$.
	
	\cref{thm:CCS:proper:EquaDistance:Normprese}: Take an arbitrary but fixed element $T \in \mathcal{S}$. Then $T x \in \mathcal{S}(x) \subseteq  \aff (\mathcal{S}(x))$. Denote $p := \Pro_{\aff (\mathcal{S}(x))}(z)$. 	By \cref{fact:PR}\cref{fact:PR:Pythagoras},
	\begin{align} \label{thm:eq:2}
	\norm{z-p}^{2}+\norm{p- Tx }^{2} &= \norm{z-Tx}^{2}.
	\end{align}
	By \cref{lem:EquDistS:T}, $\norm{z-Tx} = \norm{z -x}$. Thus, \cref{thm:eq:2} yields that
	\begin{align*}
	(\forall T \in \mathcal{S} ) \quad \norm{p- Tx} = \big(  \norm{z -x}^{2} - \norm{z-p}^{2} \big)^{\frac{1}{2}},
	\end{align*}
	which implies that $\left\{  \norm{p - Tx } ~|~ T \in \mathcal{S} \right\}$ is a singleton.
\end{proof}

The following
\cref{thm:CCS:proper:NormPres:T}\cref{thm:CCS:proper:NormPres:T:prop} states
that the circumcenter mapping induced by isometries is proper, which makes
the circumcentered isometry method well-defined and is therefore fundamental for 
our study on circumcentered isometry methods.

\begin{theorem} \label{thm:CCS:proper:NormPres:T}
	Let $x\in \mathcal{H}$. Then the  following  hold:
	\begin{enumerate}
  \item \label{thm:CCS:proper:NormPres:T:prop} The circumcenter mapping
  $\CC{\mathcal{S}} : \mathcal{H} \rightarrow \mathcal{H}$ induced by
  $\mathcal{S}$ is proper; moreover, 
  $\CC{\mathcal{S}}x$ is the unique point satisfying the two conditions
  below:
		\begin{enumerate}
			\item  \label{thm:CCS:proper:NormPres:T:i} $\CC{\mathcal{S}}x\in  \aff (\mathcal{S}(x))$, and
			\item  \label{thm:CCS:proper:NormPres:T:ii} $
			\left\{  \norm{\CC{\mathcal{S}}x-Tx } ~|~ T \in \mathcal{S} \right\} $ is a singleton.
		\end{enumerate}
		\item \label{thm:CCS:proper:NormPres:T:AllU} $(\forall z \in \cap^{m}_{j=1}  \Fix T_{j})$  $\CC{\mathcal{S}}x= \Pro_{\aff (\mathcal{S}(x))}(z)$.
		
  \item \label{thm:CCS:proper:NormPres:T:PaffU} Assume that $\varnothing \neq
  W \subseteq \cap^{m}_{j=1} \Fix T_{j}$ and that $W$ is closed and convex.
  Then $\CC{\mathcal{S}}x= \Pro_{\aff
  (\mathcal{S}(x))}(\Pro_{\cap^{m}_{j=1} \Fix T_{j}} x) = \Pro_{\aff
  (\mathcal{S}(x))}(\Pro_{W} x)$.
	\end{enumerate}
\end{theorem}

\begin{proof}
	\cref{thm:CCS:proper:NormPres:T:prop} and \cref{thm:CCS:proper:NormPres:T:AllU}  come from  \cref{prop:CCS:proper:Normpreserv} and \cite[Proposition~3.6]{BOyW2018Proper}. 
%	\cref{fact:CW:ExistWellDefined}.
	
	Using \cref{lem:normpreser:FixSet} and the underlying assumptions, we know $\cap^{m}_{j=1}  \Fix T_{j} $ is nonempty, closed and convex, so $\Pro_{\cap^{m}_{j=1}  \Fix T_{j}} x \in \cap^{m}_{j=1}  \Fix T_{j}$ is well-defined. Hence \cref{thm:CCS:proper:NormPres:T:PaffU} comes from \cref{thm:CCS:proper:NormPres:T:AllU}.
\end{proof}

\subsection{Further properties of circumcenter mappings induced by isometries} \label{subsec:PropoertiesCircmMapping}

Similarly to \cref{prop:CW:Welldefined:Formula}, 
we provide a formula of the circumcenter
mapping in the following result. 
Because $\Pro_{\cap^{m}_{i=1} \Fix T_{i} } x$ or $\Pro_{W} x$ is
unknown in general, 
\cref{prop:CW:Welldefined:Formula} is more practical.

\begin{proposition}  \label{prop:CCSformula:GSOthorgonal}
	Let  $\varnothing \neq W \subseteq \cap^{m}_{j=1}  \Fix T_{j}$ and let $W$ be closed and convex.	   Let $x \in \mathcal{H}$. Set $d_{x}:= \dim \big( \spn\{ T_{2}x-T_{1}x, \ldots, T_{m}x-T_{1}x\} \big) $.  Let $\widetilde{\mathcal{S}} := \{ T_{1}, T_{i_{1}}, \ldots, T_{i_{d_{x}}}\} \subseteq \mathcal{S}$ be such that \footnotemark
	
	\footnotetext{Note that if $\card \left(\mathcal{S}(x)\right)=1$, then $d_{x}=0$ and so $\CC{\mathcal{S}}x=T_{1}x$. }
	\begin{align} \label{eq:Prop:CCSformula:assum}
	T_{i_{1}}x -T_{1}x, \ldots,T_{i_{d_{x}}}x-T_{1}x ~\text{is a basis of } \spn\{ T_{2}x-T_{1}x, \ldots, T_{m}x-T_{1}x\}.
	\end{align}
	Then
	\begin{align*}
	\CC{\mathcal{S}}x=T_{1}x+ \sum^{d_{x}}_{j=1} \innp{\Pro_{\cap^{m}_{i=1}  \Fix T_{i} } x-T_{1}x, e_{j}} e_{j} = T_{1}x+ \sum^{d_{x}}_{j=1} \innp{\Pro_{W} x-T_{1}x, e_{j}} e_{j}.
	\end{align*}
	where $(j \in \{1, \ldots, d_{x}\})$ $e_{j} =\frac{T_{i_{j}}x-T_{1}x - \sum^{j-1}_{k=1} \innp{T_{i_{j}}x-T_{1}x,e_{k}}e_{k} }{\norm{T_{i_{j}}x-T_{1}x - \sum^{j-1}_{k=1} \innp{T_{i_{j}}x-T_{1}x,e_{k}}e_{k} }}$.
\end{proposition}

\begin{proof}
	By \cref{thm:CCS:proper:NormPres:T}\cref{thm:CCS:proper:NormPres:T:PaffU} ,
	\begin{align*}
	\CC{\mathcal{S}}x= \Pro_{\aff (\mathcal{S}(x))}(\Pro_{\cap^{m}_{j=1}  \Fix T_{j} } x)  = \Pro_{\aff (\mathcal{S}(x))}(\Pro_{W} x).
	\end{align*}
	By \cref{eq:Prop:CCSformula:assum}, we know that
	\begin{align*}
	\aff (\mathcal{S}(x))=\aff \{T_{1}x, T_{i_{1}}x,  \ldots, T_{i_{d_{x}}}x \}=T_{1}x +\spn\{T_{i_{1}}x-T_{1}x, \ldots, T_{i_{d_{x}}}x-T_{1}x\}.
	\end{align*}
	Substituting $(x,x_{1},\ldots,x_{n},M)$ by 
	$(T_{1}x,T_{i_{1}}x, \ldots,T_{i_{d_{x}}}x,\aff (\mathcal{S}(x))$ in
	\cref{lem:othorg}, we obtain the desired result.
\end{proof}

The following result plays a important role for the proofs
of the linear convergence of circumcentered isometry methods.

\begin{lemma} \label{lem:CCS:Equations}
	Let $x \in \mathcal{H}$, and  $z \in \cap^{m}_{j=1} \Fix T_{j}$. Then the following  hold:
	\begin{enumerate}
		\item \label{lem:CCS:Equations:F} Let $F: \mathcal{H} \to \mathcal{H}$ satisfy  $(\forall y \in \mathcal{H})$ $F(y) \in \aff (\mathcal{S}(y))$. Then $\norm{z- \CC{\mathcal{S}}x }^{2} + \norm{\CC{\mathcal{S}}x -Fx}^{2} =\norm{z- Fx}^{2} $;
		\item  \label{lem:CCS:Equations:TS} If $T_{\mathcal{S}} \in \aff \mathcal{S} $, then $\norm{z - \CC{\mathcal{S}}x }^{2}  + \norm{ \CC{\mathcal{S}}x - T_{\mathcal{S}}x }^{2}  =\norm{z - T_{\mathcal{S}}x}^{2}$;
		
		\item \label{lem:CCS:Equations:Id} If $\Id \in \aff \mathcal{S}$, then$\norm{z- \CC{\mathcal{S}}x }^{2} + \norm{\CC{\mathcal{S}}x -x}^{2} =\norm{z- x}^{2}$;
		
		\item \label{lem:CCS:Equations:TinS} $( \forall T \in \mathcal{S})$  $\norm{z- \CC{\mathcal{S}}x }^{2} + \norm{\CC{\mathcal{S}}x -Tx}^{2} =\norm{z- x}^{2}$.
	\end{enumerate}
\end{lemma}

\begin{proof}
	Using \cref{thm:CCS:proper:NormPres:T}\cref{thm:CCS:proper:NormPres:T:AllU}, we obtain
	\begin{align} \label{eq:lem:CCS:FormP}
	\CC{\mathcal{S}}x=\Pro_{\aff(\mathcal{S}(x))}(z).
	\end{align}
	
	\cref{lem:CCS:Equations:F}: Since $F(x) \in \aff (\mathcal{S}(x))$,
	%	$(\forall y \in \mathcal{H})$ $F(y) \in \aff (\mathcal{S}(y))$,
	\cref{fact:PR}\cref{fact:PR:Pythagoras} implies
	\begin{align*}
	\norm{z-\CC{\mathcal{S}}x}^{2} + \norm{ \CC{\mathcal{S}}x-Fx}^{2}= \norm{z-Fx}^{2}.
	\end{align*}
	
	\cref{lem:CCS:Equations:TS} and \cref{lem:CCS:Equations:Id}  come directly from \cref{lem:CCS:Equations:F}.
	
	Note that $( \forall T \in \mathcal{S})$ $T$ is isometric and $z \in
	\cap^{m}_{j=1} \Fix T_{j} \subseteq \Fix T$. Hence,
	\cref{lem:CCS:Equations:TinS} follows easily from \cref{lem:CCS:Equations:TS}.
\end{proof}

We now present some calculus rules for circumcenter mappings.

\begin{corollary} \label{cor:HomogeAdditiveModulo}
	Assume  $(\forall T \in \mathcal{S})$ $T$ is linear. Then
	\begin{enumerate}
		\item \label{cor:HomogeAdditiveModulo:Homoge} $\CC{\mathcal{S}}$ is homogeneous, that is
		$(\forall x \in \mathcal{H})$ $(\forall \lambda \in \mathbb{R})$ $ \CC{\mathcal{S}} (\lambda x)=\lambda \CC{\mathcal{S}} x$;
		\item \label{cor:HomogeAdditiveModulo:AdditiveModulo} $\CC{\mathcal{S}}$ is quasitranslation, that is, $(\forall x \in \mathcal{H})$  $(\forall z \in \cap^{m}_{j=1} \Fix T_{j}) \quad \CC{\mathcal{S}} (x+ z ) = \CC{\mathcal{S}} (x) + z.$
	\end{enumerate}
\end{corollary}

\begin{proof}
	By assumption, $(\forall T \in \mathcal{S})$ $T$ is linear, so for every $\alpha, \beta \in \mathbb{R}$, and for every $x ,y \in \mathcal{H}$,
	\begin{align*}
	(\forall T \in \mathcal{S}) \quad T(\alpha x + \beta y) = \alpha T x+ \beta Ty.
	\end{align*}
	
	Note that by \cref{thm:CCS:proper:NormPres:T}\cref{thm:CCS:proper:NormPres:T:prop}, $\CC{\mathcal{S}}$ is proper. By \cref{fact:CW:FixPointSet}\cref{fact:CW:FixPointSet:Basic}, $0 \in \cap^{m}_{j=1} \Fix T_{j} \subseteq \Fix \CC{\mathcal{S}}$. Hence,
	\begin{align*}
	(\forall x \in \mathcal{H}) \quad \CC{\mathcal{S}} (0 x)=0=0 \CC{\mathcal{S}} x.
	\end{align*}
	
	Therefore, \cref{cor:HomogeAdditiveModulo:Homoge} is from  \cite[Proposition~6.1]{BOyW2018} and \cref{cor:HomogeAdditiveModulo:AdditiveModulo}  comes from \cite[Proposition~6.3]{BOyW2018}.
%	Therefore, \cref{cor:HomogeAdditiveModulo:Homoge} and \cref{cor:HomogeAdditiveModulo:AdditiveModulo} are directly from \cref{prop:CCS:Homo} and \cref{prop:CCS:ModuAddi} respectively.
\end{proof}

The following result characterizes the fixed point set of
circumcenter mappings induced by isometries under some conditions.

\begin{proposition} \label{prop:FixCCS:F1Ft}
	Recall that $\mathcal{S}=\{ T_{1}, \ldots, T_{m-1}, T_{m} \} $. Then the following  hold:
	\begin{enumerate}
		\item  \label{prop:FixCCS:F1Ft:T} Assume $T_{1} =\Id$. Then $\Fix \CC{\mathcal{S}} = \cap^{m}_{j=1} \Fix T_{j}$.
		\item  \label{prop:FixCCS:F1Ft:F} Let $F_{1}, \ldots, F_{t}$ be isometries from $\mathcal{H}$ to $\mathcal{H}$.  Assume that $\CC{\mathcal{S}}$ is proper, and that $\mathcal{S}$ is a finite subset of $\Omega ( F_{1}, \ldots, F_{t}) $ defined in \cref{eq:Omega} such that $\{ \Id, F_{1}, F_{2}F_{1}, \ldots, F_{t} F_{t-1} \cdots F_{2}F_{1} \} \subseteq \mathcal{S}$ or $\{ \Id, F_{1}, F_{2}, \ldots, F_{t}  \} \subseteq \mathcal{S}$.  Then $\Fix \CC{\mathcal{S}} = \cap^{t}_{j=1} \Fix F_{j}  = \cap^{m}_{j=1} \Fix T_{j} $.
	\end{enumerate}
\end{proposition}

\begin{proof}
	\cref{prop:FixCCS:F1Ft:T} is clear from \cref{thm:CCS:proper:NormPres:T}\cref{thm:CCS:proper:NormPres:T:prop} and \cref{fact:CW:FixPointSet}\cref{fact:CW:FixPointSet:Id}.
	
	\cref{prop:FixCCS:F1Ft:F}:
	Combining \cref{thm:CCS:proper:NormPres:T}\cref{thm:CCS:proper:NormPres:T:prop} with \cref{prop:Fix:CCS:EQ2}, we obtain $\Fix \CC{\mathcal{S}} = \cap^{t}_{j=1} \Fix F_{j} $.  In addition,  the \cref{prop:FixCCS:F1Ft:T} proved above implies that $  \Fix \CC{\mathcal{S}}  = \cap^{m}_{j=1} \Fix T_{j} $. Hence, the proof is complete.
\end{proof}	

\begin{proposition} \label{prop:FirmQuasiNon}
	Let $F_{1}, \ldots, F_{t}$ be isometries from $\mathcal{H}$ to $\mathcal{H}$.
	Assume that $\CC{\mathcal{S}}$ is proper, and that $\mathcal{S}$ is a finite subset of $\Omega ( F_{1}, \ldots, F_{t}) $ defined in \cref{eq:Omega} such that $\{ \Id, F_{1}, F_{2}F_{1}, \ldots, F_{t} F_{t-1} \cdots F_{2}F_{1} \} \subseteq \mathcal{S}$ or $\{ \Id, F_{1}, F_{2}, \ldots, F_{t}  \} \subseteq \mathcal{S}$.  Then
	\begin{align}  \label{eq:prop:FirmQuasiNon}
	(\forall x \in \mathcal{H})\quad (\forall y \in \Fix \CC{\mathcal{S}}) \quad \norm{\CC{\mathcal{S}}x -y}^{2}+\norm{\CC{\mathcal{S}}x-x}^{2} = \norm{x -y}^{2}.
	\end{align}
	In particular,
	$\CC{\mathcal{S}}$ is firmly quasinonexpansive.
\end{proposition}

\begin{proof}
	\cref{prop:FixCCS:F1Ft}\cref{prop:FixCCS:F1Ft:F}  says that in both cases stated in the assumptions, $\Fix \CC{\mathcal{S}} = \cap^{t}_{j=1} \Fix F_{j}  = \cap_{T \in \mathcal{S}} \Fix T $.  Combining this result with \cref{lem:CCS:Equations}\cref{lem:CCS:Equations:Id}, we obtain \cref{eq:prop:FirmQuasiNon}.
	
	Hence, by \cref{defn:Nonexpansive}\cref{FirmQuasiNonex}, $\CC{\mathcal{S}}$ is firmly quasinonexpansive.
\end{proof}	

\begin{corollary} \label{coro:RUtRU1}
 Let $U_{1}, \ldots, U_{t}$ be closed affine subspaces in $\mathcal{H}$.
 Assume that $\mathcal{S}_{1} =\{\Id, \R_{U_{1}}, \ldots, \R_{U_{t}}\}$ and
that  $\mathcal{S}_{2} =\{\Id, \R_{U_{1}}, \R_{U_{2}}\R_{U_{1}}, \ldots,
 \R_{U_{t}}\cdots\R_{U_{2}}\R_{U_{1}} \}$. Then
	\begin{enumerate}
		\item \label{coro:RUtRU1:Fix} $(\forall i \in \{ 1,2\} )$  $\Fix \CC{ \mathcal{S}_{i}} = \bigcap_{T \in \mathcal{S}_{i} } \Fix T = \cap^{t}_{j=1} \Fix \R_{U_{j}} = \cap^{t}_{j=1} U_{j}$.
		\item \label{coro:RUtRU1:FirmQuasi} $\CC{\mathcal{S}_{1}}$ and  $\CC{\mathcal{S}_{2}}$  are firmly quasinonexpansive.
	\end{enumerate}
\end{corollary}	
\begin{proof}
	We obtain \cref{coro:RUtRU1:Fix} and  \cref{coro:RUtRU1:FirmQuasi} by substituting $F_{1}=\R_{U_{1}}, \ldots, F_{t}=\R_{U_{t}}$ in \cref{prop:FixCCS:F1Ft,prop:FirmQuasiNon} respectively.
\end{proof}
In fact,  the $\CC{\mathcal{S}_{2}}$ in \cref{coro:RUtRU1} is the main actor in \cite{BCS2018}.

\section{Circumcenter methods induced by isometries} \label{sec:CircumcenterMethodIsome}

Recall that $\mathcal{S}=\{ T_{1}, \ldots, T_{m-1}, T_{m} \}$ with
$\cap^{m}_{j=1} \Fix T_{j} \neq \varnothing$ and that every element of
$\mathcal{S}$ is isometric and affine.

Let $x \in \mathcal{H}$. The \emph{circumcenter method} induced by $\mathcal{S}$ is
\begin{align*}
x_{0} :=x, ~\mbox{and}~x_{k} :=\CC{\mathcal{S}}(x_{k-1})=\CC{\mathcal{S}}^{k}x, ~\mbox{where}~k=1,2,\ldots.
\end{align*}
\cref{thm:CCS:proper:NormPres:T}\cref{thm:CCS:proper:NormPres:T:prop} says that $\CC{\mathcal{S}}$ is proper, which ensures that the circumcenter method induced by $\mathcal{S}$ is well defined. Since every element of  $\mathcal{S}$ is isometric, we say that the circumcenter method is the \emph{circumcenter method induced by isometries}.

\subsection{Properties of circumcentered isometry methods} \label{subsec:ProperCMIsometries}
In this subsection, we provide some properties of
circumcentered isometry methods. All of the properties are interesting in
their own right. Moreover, the following
\Cref{prop:CW:NormPreserv:Feje,prop:CCS:Equations} play an important role in
the convergence proofs later.

\begin{proposition} \label{prop:CW:NormPreserv:Feje}
	Let $x \in \mathcal{H}$.
	Then the following  hold:
	\begin{enumerate}
		\item \label{prop:CW:NormPreserv:Feje:i} $(\CC{\mathcal{S}}^{k}x)_{k \in \mathbb{N}}$ is a Fej\'er monotone sequence with respect to $\cap^{m}_{j=1}  \Fix T_{j}$.
		\item \label{prop:CW:NormPreserv:Feje:limitexits} $(\forall z \in  \cap^{m}_{j=1}  \Fix T_{j} )$ the limit $\lim_{k \rightarrow + \infty} \norm{\CC{\mathcal{S}}^{k}x-z}$ exists.
		\item \label{prop:CW:NormPreserv:Feje:bounded}  $(\CC{\mathcal{S}}^{k}x)_{k \in \mathbb{N}}$ is bounded sequence.
		\item \label{prop:CW:NormPreserv:Feje:subset} Assume $\varnothing \neq W \subseteq \cap^{m}_{j=1}  \Fix T_{j}$. Then $(\CC{\mathcal{S}}^{k}x)_{k \in \mathbb{N}}$ is a Fej\'er monotone sequence with respect to $W$.
		\item \label{prop:CW:NormPreserv:Feje:Asym} Assume $\Id \in \aff \mathcal{S}$. Then $\CC{\mathcal{S}}$ is asymptotically regular, that is for every $y \in \mathcal{H}$,
		\begin{align*}
		\lim_{k \rightarrow \infty}\CC{\mathcal{S}}^{k} y-\CC{\mathcal{S}}^{k+1} y =0.
		\end{align*}
	\end{enumerate}
\end{proposition}
\begin{proof}
	For every $k \in \mathbb{N}$, substitute $x$ by $\CC{\mathcal{S}}^{k}x$ in \cref{lem:CCS:Equations}\cref{lem:CCS:Equations:TinS} to obtain
	\begin{align} \label{eq:prop:CW:NormPreserv:Feje:star}
	(\forall T \in \mathcal{S})	\quad 	(\forall z \in  \cap^{m}_{j=1}  \Fix T_{j}) \quad  \norm{z -  \CC{\mathcal{S}}^{k+1}x}^{2}   + \norm{\CC{\mathcal{S}}^{k+1}x-T \CC{\mathcal{S}}^{k}x}^{2} = \norm{z - \CC{\mathcal{S}}^{k}x}^{2}.
	\end{align}
	
	\cref{prop:CW:NormPreserv:Feje:i}: By \cref{eq:prop:CW:NormPreserv:Feje:star}, it is clear that
	\begin{align} \label{eq:prop:CW:Feje}
	(\forall z \in  \cap^{m}_{j=1}  \Fix T_{j}) \quad (\forall k \in \mathbb{N}) \quad \norm{\CC{\mathcal{S}}^{k+1}x -z} \leq \norm{\CC{\mathcal{S}}^{k}x-z}.
	\end{align}
	By \cref{defn:Fejer}, $(\CC{\mathcal{S}}^{k}x)_{k \in \mathbb{N}}$ is a Fej\'er monotone sequence with respect to $\cap^{m}_{j=1}  \Fix T_{j}$.
	
	\cref{prop:CW:NormPreserv:Feje:limitexits}: By \cref{eq:prop:CW:Feje}, clearly $(\forall z \in  \cap^{m}_{j=1}  \Fix T_{j} )$  $ \lim_{k \rightarrow + \infty} \norm{\CC{\mathcal{S}}^{k}x-z}$	exists.

	\cref{prop:CW:NormPreserv:Feje:bounded}:   It directly comes from \cref{prop:CW:NormPreserv:Feje:i} and \cref{fact:Fejer:Bounded}.
	
	\cref{prop:CW:NormPreserv:Feje:subset}: The desired result is directly from \cref{prop:CW:NormPreserv:Feje:i} and \cref{defn:Fejer}.

	\cref{prop:CW:NormPreserv:Feje:Asym}: Let $z \in \cap^{m}_{j=1} \Fix
	T_{j}$. By \cref{prop:CW:NormPreserv:Feje:limitexits} above, we know
	$L_{z} := \lim_{k \rightarrow + \infty} \norm{\CC{\mathcal{S}}^{k}x-z}$
	exists. Since $\Id \in \aff \mathcal{S}$, for every $k \in
	\mathbb{N}$, substituting $x$ by $\CC{\mathcal{S}}^{k}x$ in
	\cref{lem:CCS:Equations}\cref{lem:CCS:Equations:Id}, we have
	\begin{align} \label{prop:CW:essential}
	\norm{ \CC{\mathcal{S}}^{k}x-\CC{\mathcal{S}}^{k+1}x}^{2}  = \norm{\CC{\mathcal{S}}^{k}x-z}^{2} - \norm{\CC{\mathcal{S}}^{k+1}x -z}^{2}.
	\end{align}
	
	Summing over $k$ from $0$ to infinity in both sides of $\cref{prop:CW:essential}$, we obtain
	\begin{align*}
	\sum^{\infty}_{k=0} \norm{ \CC{\mathcal{S}}^{k}x-\CC{\mathcal{S}}^{k+1}x}^{2}  = \norm{x-z}^{2} - L^{2}_{z} < +\infty,
	\end{align*}
	which yields $ \lim_{k \rightarrow + \infty} \CC{\mathcal{S}}^{k}x-\CC{\mathcal{S}}^{k+1}x=0$, i.e., $\CC{\mathcal{S}}$ is asymptotically regular.

\end{proof}

The following results are motivated by \cite[Lemmas~1 and 3]{BCS2017}.
Note that by \cref{lem:normpreser:FixSet}\cref{lem:normpreser:FixSet:Fix}, $\cap^{m}_{j=1}  \Fix T_{j}$  is always closed and convex.
\begin{proposition} \label{prop:CCS:Equations}
Let $\varnothing \neq W \subseteq \cap^{m}_{j=1} \Fix T_{j}$ such that
$W$ is convex and closed. Let $x \in \mathcal{H}$. Then 
the following hold:
	\begin{enumerate}
		\item  \label{prop:CCS:Equations:T}  $(\forall T \in \mathcal{S})$ $\Pro_{W} Tx=T\Pro_{W}x=\Pro_{W}x$ and $d(x, W)= d(Tx, W)$.
		\item \label{prop:CCS:Equations:CCSxk:RIGHT}  $(\forall k \in \mathbb{N})$  $\CC{\mathcal{S}}^{k}\Pro_{W} x =\Pro_{W} x$.
		\item \label{prop:CCS:Equations:CCSxk} Assume $W$ is closed and affine. Then $(\forall k \in \mathbb{N})$ $\Pro_{W} (\CC{\mathcal{S}}^{k}x)=\Pro_{W} x$.
		\item \label{prop:CCS:Equations:CCSTSP} Let $T_{\mathcal{S}} \in \aff (\mathcal{S})$. Then $\norm{ \Pro_{W} x - \CC{\mathcal{S}}x}^{2} + \norm{\CC{\mathcal{S}}x - T_{\mathcal{S}}x}^{2} = \norm{ \Pro_{W} x - T_{\mathcal{S}}x}^{2}$.
	\end{enumerate}
\end{proposition}

\begin{proof}
	\cref{prop:CCS:Equations:T}: Let $T \in \mathcal{S}$.  Since $W \subseteq \cap^{m}_{j=1} \Fix T_{j} \subseteq \Fix T$, thus it is clear that $T\Pro_{W}x=\Pro_{W}x$. Moreover, since $\Pro_{W} x \in W \subseteq \cap^{m}_{j=1} \Fix T_{j} \subseteq \Fix T$,   $\Pro_{W} Tx \in W \subseteq \cap^{m}_{i=1} \Fix T_{i} \subseteq \Fix T$
	and since $T$ is isometric, thus
	\begin{align*}
	\norm{x - \Pro_{W}x} & \leq \norm{x - \Pro_{W}Tx} \quad (\text{by definition of best approximation and $ \Pro_{W}Tx \in W$}) \\
	&	= \norm{Tx - \Pro_{W}Tx}  \quad (\text{$T$ is isometric})\\
	&	\leq \norm{Tx - \Pro_{W}x}  \quad (\text{by definition of best approximation and $ \Pro_{W}x \in W$}) \\
	&	= \norm{x - \Pro_{W}x}, \quad (\text{$T$ is isometric})
	\end{align*}
	which imply that
	\begin{align} \label{eq:xProWxTx}
	\norm{x - \Pro_{W}x} =\norm{Tx - \Pro_{W}Tx}  = \norm{x - \Pro_{W}Tx} .
	\end{align}
	Since $W$ is nonempty, closed and convex, the best approximation of $x$ onto $W$ uniquely exists. So \cref{eq:xProWxTx} implies that $\Pro_{W}Tx =\Pro_{W}x $ and $d(x, W)= d(Tx, W)$.
	
	\cref{prop:CCS:Equations:CCSxk:RIGHT}:  By  assumption and by \cref{fact:CW:FixPointSet}\cref{fact:CW:FixPointSet:Basic}, $\Pro_{W}x \in W \subseteq \cap^{m}_{j=1} \Fix T_{j} \subseteq \Fix \CC{\mathcal{S}}$, thus it is clear that $(\forall k \in \mathbb{N})$  $\CC{\mathcal{S}}^{k}\Pro_{W} x =\Pro_{W} x$.
	
	\cref{prop:CCS:Equations:CCSxk}: The required result comes from \cref{prop:CW:NormPreserv:Feje}\cref{prop:CW:NormPreserv:Feje:subset} and \cref{fact:FejerWeakCluster}\cref{fact:FejerWeakCluster:i}.

	\cref{prop:CCS:Equations:CCSTSP}: 	By \cref{thm:CCS:proper:NormPres:T}\cref{thm:CCS:proper:NormPres:T:PaffU}, $\CC{\mathcal{S}}x = \Pro_{\aff(\mathcal{S}(x))}\Pro_{W}x$. Since $T_{\mathcal{S}} \in \aff (\mathcal{S})$, which implies that $T_{\mathcal{S}}x \in \aff (\mathcal{S}(x))$, thus by \cref{fact:PR}\cref{fact:PR:Pythagoras}, $\norm{ \Pro_{W} x - \CC{\mathcal{S}}x}^{2} + \norm{\CC{\mathcal{S}}x - T_{\mathcal{S}}x}^{2} = \norm{ \Pro_{W} x - T_{\mathcal{S}}x}^{2}$.
\end{proof}

With $W=\cap^{m}_{j=1} \Fix T_{j}$ in the following
result, we know that $(\forall x \in \mathcal{H})$ the distance between
$\CC{\mathcal{S}}x \in \aff (\mathcal{S}(x))$ and $\Pro_{\cap^{m}_{j=1} \Fix
T_{j}}x \in \cap^{m}_{j=1} \Fix T_{j}$ is exactly the distance between the
two affine subspaces $\aff (\mathcal{S}(x))$ and $\cap^{m}_{j=1} \Fix
T_{j}$.

\begin{corollary} \label{cor:CapUiAffSxDist}
	Let $\varnothing \neq W \subseteq \cap^{m}_{j=1}  \Fix T_{j}$ such that $W$ is closed and  affine. Let $x \in \mathcal{H}$. Then
	\begin{align*}
	\norm{\CC{\mathcal{S}}x - \Pro_{W}x}  =\dist (\aff (\mathcal{S}(x)),  W ).
	\end{align*}
\end{corollary}

\begin{proof}
	By \cref{thm:CCS:proper:NormPres:T}\cref{thm:CCS:proper:NormPres:T:AllU}, $(\forall z \in \cap^{m}_{j=1} \Fix T_{j})$ $\CC{\mathcal{S}}x= \Pro_{\aff (\mathcal{S}(x))}(z)$, which implies that
	\begin{align} \label{eq:cor:CapUiAffSxDist}
	(\forall z \in W \subseteq  \cap^{m}_{j=1} \Fix T_{j} ) \quad \norm{\CC{\mathcal{S}}x - z} = \dist (\aff (\mathcal{S}(x)), z).
	\end{align}
	Now taking infimum over all $z$ in $W$ in \cref{eq:cor:CapUiAffSxDist}, we obtain
	\begin{align*}
	\dist (\CC{\mathcal{S}}x, W ) = 	\inf_{z \in  W} \norm{\CC{\mathcal{S}}x - z} = \inf_{z \in W} \dist (\aff (\mathcal{S}(x)), z) =\dist (\aff (\mathcal{S}(x)),  W).
	\end{align*}
	Hence, using \cref{prop:CCS:Equations}\cref{prop:CCS:Equations:CCSxk}, we
	 deduce that $\norm{\CC{\mathcal{S}}x - \Pro_{  W }x }  = \norm{\CC{\mathcal{S}}x -\Pro_{ W}(\CC{\mathcal{S}}x ) }
	= \dist (\CC{\mathcal{S}}x, W )
	=   \dist (\aff (\mathcal{S}(x)),  W)$.
\end{proof}

% \textcolor{purple}{By the following result with $W =\cap^{m}_{j=1} \Fix
% T_{j}$, for every $x \in \mathcal{H}$, if only $ \CC{\mathcal{S}}x \in
% \cap^{m}_{j=1} \Fix T_{j}$, then $\CC{\mathcal{S}}x = \Pro_{\cap^{m}_{j=1}
% \Fix T_{j}}x$.}

\begin{proposition} \label{prop:CCSkxInPCapUi}
	Let $\varnothing \neq W \subseteq \cap^{m}_{j=1}  \Fix T_{j}$ such that $W$ is closed and  affine. Let $x \in \mathcal{H}$. Then the following  are equivalent:
	\begin{enumerate}
		\item \label{prop:CCSkxInPCapUi:i} $\CC{\mathcal{S}}x \in W$;
		\item \label{prop:CCSkxInPCapUi:ii} $\CC{\mathcal{S}}x = \Pro_{W}x$;
		\item \label{prop:CCSkxInPCapUi:iii} $(\forall k \geq 1)$ $\CC{\mathcal{S}}^{k}x = \Pro_{W}x$.
	\end{enumerate}
\end{proposition}

\begin{proof}
	\enquote{\cref{prop:CCSkxInPCapUi:i} $\Rightarrow$ \cref{prop:CCSkxInPCapUi:ii}}: If $\CC{\mathcal{S}}x \in W$, then $\CC{\mathcal{S}}x = \Pro_{W} \CC{\mathcal{S}}x  = \Pro_{W}x$ using \cref{prop:CCS:Equations}\cref{prop:CCS:Equations:CCSxk}.
	
	\enquote{\cref{prop:CCSkxInPCapUi:ii} $\Rightarrow$ \cref{prop:CCSkxInPCapUi:iii}}: Assume $\CC{\mathcal{S}}x = \Pro_{W}x$. By \cref{fact:CW:FixPointSet}\cref{fact:CW:FixPointSet:Basic}, $\Pro_{W}x \in W \subseteq \cap^{m}_{j=1} \Fix T_{j} \subseteq \Fix \CC{\mathcal{S}}$. Hence,
	\begin{align*}
	(\forall k \geq 2) \quad  \CC{\mathcal{S}}^{k}x = \CC{\mathcal{S}}^{k-1}(\CC{\mathcal{S}}x) =  \CC{\mathcal{S}}^{k-1}(\Pro_{W}x) = \Pro_{W}x.
	\end{align*}
	
	\enquote{\cref{prop:CCSkxInPCapUi:iii} $\Rightarrow$ \cref{prop:CCSkxInPCapUi:i}}: Take $k=1$.
\end{proof}

\begin{corollary} \label{cor:CCSkPCapUiNotinCapUi}
	Let $\varnothing \neq W \subseteq \cap^{m}_{j=1}  \Fix T_{j}$ such that $W$ is closed and affine. Let $x \in \mathcal{H}$. Assume that $\lim_{k \rightarrow \infty} \CC{\mathcal{S}}^{k}x \neq \Pro_{W}x$. Then
	\begin{align}
	(\forall k \in \mathbb{N}) \quad \CC{\mathcal{S}}^{k}x \not \in W.
	\end{align}
\end{corollary}

\begin{proof}
	We argue by contradiction and thus assume there exists $n \in \mathbb{N}$ such that $\CC{\mathcal{S}}^{n}x  \in W$. If $n=0$, then, by \cref{fact:CW:FixPointSet}\cref{fact:CW:FixPointSet:Basic}, $(\forall k \in \mathbb{N})$ $\CC{\mathcal{S}}^{k}x =x = \Pro_{W}x$, which contradicts the assumption, $\lim_{k \rightarrow \infty} \CC{\mathcal{S}}^{k}x \neq \Pro_{W}x$. Assume $n \geq 1$.
	Then \cref{prop:CCSkxInPCapUi} implies 	$(\forall k \geq n)$  $\CC{\mathcal{S}}^{k}x = \Pro_{W} \CC{\mathcal{S}}^{n-1}x$, which is absurd.
\end{proof}

\begin{proposition} \label{prop:CCSK}
	Assume  $(\forall T \in \mathcal{S})$ $T$ is linear. Then
	\begin{enumerate}
		\item  \label{prop:CCSK:homo} $(\forall x \in \mathcal{H})$  $(\forall \lambda \in \mathbb{R})$ $\CC{\mathcal{S}}^{k} (\lambda x)=\lambda \CC{\mathcal{S}}^{k} x$.
		\item  \label{prop:CCSK:Addit} $(\forall x \in \mathcal{H})$  $(\forall z \in \cap^{m}_{j=1} \Fix T_{j})$ $ \CC{\mathcal{S}}^{k} (x+ z ) = \CC{\mathcal{S}}^{k} (x) + z$.
	\end{enumerate}
\end{proposition}	

\begin{proof}
	The required results follow easily from \cref{cor:HomogeAdditiveModulo} and some easy induction.
\end{proof}	

\subsection{Convergence} \label{subsec:Convergence}

In this subsection, we consider the weak, strong and linear convergence 
of circumcentered isometry methods.

\begin{theorem} \label{theore:CM:WeakConver}
	Assume $T_{1} =\Id$ and $ \cap^{m}_{j=1}  \Fix T_{j}$ is an affine subspace of $\mathcal{H}$. Let $x \in \mathcal{H}$. Then $(\CC{\mathcal{S}}^{k}x)$ weakly converges to $\Pro_{ \cap^{m}_{j=1}  \Fix T_{j}} x$ and $(\forall k \in \mathbb{N})$ $\Pro_{ \cap^{m}_{j=1}  \Fix T_{j}} \left( \CC{\mathcal{S}}^{k}x \right) =\Pro_{ \cap^{m}_{j=1}  \Fix T_{j}} x$.  In particular, if $\mathcal{H} $ is finite-dimensional space, then $(\CC{\mathcal{S}}^{k}x)_{k \in \mathbb{N}}$ converges to $\Pro_{ \cap^{m}_{j=1}  \Fix T_{j}}x$.
\end{theorem}

\begin{proof}
	By \cref{prop:CCS:Equations}\cref{prop:CCS:Equations:CCSxk}, we have $(\forall k \in \mathbb{N} \smallsetminus \{0\})$ $\Pro_{ \cap^{m}_{j=1}  \Fix T_{j}} \left( \CC{\mathcal{S}}^{k}x \right) =\Pro_{ \cap^{m}_{j=1}  \Fix T_{j}} x$.
	
	In \cref{prop:CW:NormPreserv:Feje}\cref{prop:CW:NormPreserv:Feje:i}, we proved that $(\CC{\mathcal{S}}^{k}x)_{k \in \mathbb{N}}$ is a Fej\'er monotone sequence with respect to $ \cap^{m}_{j=1}  \Fix T_{j}$.  	
	
	By assumptions above and \cref{fact:FejerWeakCluster}\cref{fact:FejerWeakCluster:ii}, in order to prove the weak convergence, it suffices to show that every weak sequential cluster point of $(\CC{\mathcal{S}}^{k}x)_{k \in \mathbb{N}}$ belongs to $\cap^{m}_{j=1}  \Fix T_{j}$.  	
	
	Because every bounded sequence in a Hilbert space possesses weakly convergent subsequence, by \cref{fact:Fejer:Bounded}, there exist weak sequential cluster points of $(\CC{\mathcal{S}}^{k}x)_{k \in \mathbb{N}}$. Assume $\bar{x}$ is a weak sequential cluster point of $(\CC{\mathcal{S}}^{k}x)_{k \in \mathbb{N}}$, that is, there exists a subsequence $(\CC{\mathcal{S}}^{k_{j}}x)_{j \in \mathbb{N}}$ of $ (\CC{\mathcal{S}}^{k}x)_{k \in \mathbb{N}}$ such that $\CC{\mathcal{S}}^{k_{j}}x \weakly \bar{x}$. Applying \cref{prop:CW:NormPreserv:Feje}\cref{prop:CW:NormPreserv:Feje:Asym}, we know that $\CC{\mathcal{S}}^{k}x -\CC{\mathcal{S}}\left(\CC{\mathcal{S}}^{k}x \right) \rightarrow 0$. So $\CC{\mathcal{S}}^{k_{j}}x -\CC{\mathcal{S}}\left(\CC{\mathcal{S}}^{k_{j}}x \right) \rightarrow 0$.
	Combining the results above with
	\cref{lem:normpreser:FixSet}\cref{lem:normpreser:FixSet:Nonexpansive},
	\cref{thm:CCS:proper:NormPres:T}\cref{thm:CCS:proper:NormPres:T:prop} and
	\cref{fact:demi:Nonexpan}, we conclude that $\bar{x} \in \Fix
	\CC{\mathcal{S}} = \cap^{m}_{j=1} \Fix T_{j}$. 
\end{proof}	

From \cref{theore:CM:WeakConver}, 
we obtain the well-known weak convergence of the Douglas-Rachford method next.

\begin{corollary} \label{cor:DRM:Conver}
	Let $U_{1}, U_{2}$ be two closed affine subspaces in $\mathcal{H}$. Denote $T_{U_{2},U_{1}} :=  \frac{\Id +\R_{U_{2}}\R_{U_{1}}}{2}$ the Douglas-Rachford operator. Let $x \in \mathcal{H}$. Then the Douglas-Rachford method $(T^{k}_{U_{2},U_{1}}x )_{k \in \mathbb{N}}$ weakly converges to $\Pro_{\Fix T_{U_{2},U_{1}}  }x$. In particular, if $\mathcal{H} $ is finite-dimensional space, then $(T^{k}_{U_{2},U_{1}}x )_{k \in \mathbb{N}}$ converges to $\Pro_{\Fix T_{U_{2},U_{1}}  }x$.
\end{corollary}	

\begin{proof}
	Set $\mathcal{S} := \{\Id, \R_{U_{2}}\R_{U_{1}} \}$. By \cref{fact:form:m2:Oper}, we know that $\CC{\mathcal{S}} = T_{U_{2},U_{1}}$.  Since $U_{1}, U_{2}$ are closed affine, thus, by	\cref{lem:examples:normpreserving}\cref{lem:examples:normpreserving:R} and \cref{lem:Composi:NormPreserOpera},  $\R_{U_{2}}\R_{U_{1}}$ is isometric and,  by \cref{lem:normpreser:FixSet}\cref{lem:normpreser:FixSet:Nonexpansive} and \cref{fact:PR}\cref{fact:PR:Affine}, $\R_{U_{2}}\R_{U_{1}}$ is nonexpansive and affine. So $\Fix \Id \cap \Fix \R_{U_{2}}\R_{U_{1}} = \Fix \R_{U_{2}}\R_{U_{1}} $ is closed and affine.  In addition, by definition of $T_{U_{2},U_{1}} $, it is clear that $\Fix T_{U_{2},U_{1}}  = \Fix \R_{U_{2}}\R_{U_{1}}$.
	
	Hence, the result comes from \cref{theore:CM:WeakConver}.
\end{proof}	

We now 
provide examples of weakly convergent circumcentered reflection methods. 

\begin{corollary} \label{coro:RUtRU1Demi}
 Let $U_{1}, \ldots, U_{t}$ be closed affine subspaces in $\mathcal{H}$.
 Assume that $\mathcal{S}_{1} =\{\Id, \R_{U_{1}}, \ldots, \R_{U_{t}}\}$ and
that  $\mathcal{S}_{2} =\{\Id, \R_{U_{1}}, \R_{U_{2}}\R_{U_{1}}, \ldots,
 \R_{U_{t}}\cdots\R_{U_{2}}\R_{U_{1}} \}$. Let $x \in \mathcal{H}$. Then both
 $(\CC{\mathcal{S}_{1}}^{k}x)$ and $(\CC{\mathcal{S}_{2}}^{k}x)$ weakly
 converge to $\Pro_{ \cap^{t}_{j=1} U_{j}} x$. In particular, if $\mathcal{H}
 $ is finite-dimensional space, then both $(\CC{\mathcal{S}_{1}}^{k}x)$ and
 $(\CC{\mathcal{S}_{2}}^{k}x)$ converges to $\Pro_{ \cap^{t}_{j=1} U_{j}} x$.
\end{corollary}	

\begin{proof}
	Since $U_{1}, \ldots, U_{t}$ are closed affine subspaces in $\mathcal{H}$, thus $\cap^{t}_{j=1} U_{j}$ is closed and affine subspace in $\mathcal{H}$. Moreover, by	\cref{lem:examples:normpreserving}\cref{lem:examples:normpreserving:R} and \cref{lem:Composi:NormPreserOpera}, every element of $\mathcal{S}$ is isometric. In addition, by \cref{coro:RUtRU1}\cref{coro:RUtRU1:Fix}, $(\forall i \in \{ 1,2\} )$  $ \bigcap_{T \in \mathcal{S}_{i} } \Fix T = \cap^{t}_{j=1} U_{j}$.
	Therefore, the required results follow from \cref{theore:CM:WeakConver}.
\end{proof}	

In fact, in \cref{sec:LinearConve:Reflector} below, we will show that  if $\mathcal{H} $ is finite-dimensional space, then both $(\CC{\mathcal{S}_{1}}^{k}x)$ and $(\CC{\mathcal{S}_{2}}^{k}x)$ defined in \cref{coro:RUtRU1Demi} above linearly  converge to $\Pro_{ \cap^{t}_{j=1} U_{j}} x$.
\begin{corollary}
	Assume that $A_{1}, \ldots, A_{d}$ are orthogonal matrices in $\mathbb{R}^{n \times n}$ and that $\mathcal{S}=\{\Id, A_{1}, \ldots, A_{d}\}$. Let $x \in \mathbb{R}^{n}$. Then  $(\CC{\mathcal{S}}^{k}x)_{k \in \mathbb{N}}$ converges to $\Pro_{ \cap^{d}_{j=1}  \Fix A_{j}}x$.
\end{corollary}	

\begin{proof}
	Since $\Fix \Id = \mathbb{R}^{n}$, we have $ \Fix \Id \bigcap (\cap^{d}_{j=1} \Fix A_{j} ) = \cap^{d}_{j=1} \Fix A_{j}$ is a closed linear subspace in $\mathbb{R}^{n}$. Moreover, by \cite[Page~321]{MC2000}, the linear isometries on $\mathbb{R}^{n}$ are precisely the orthogonal matrices. Hence, the result comes from 
%	\cref{fact:isometry:orthogn},
 \cref{lem:examples:normpreserving}\cref{lem:examples:normpreserving:Id} and  \cref{theore:CM:WeakConver}.
\end{proof}	
\begin{remark}
	If we replace $\Pro_{ \cap^{m}_{j=1}  \Fix T_{j}} x$ by $\Pro_{ W} x$ for any $\varnothing \neq W \subseteq \cap^{m}_{j=1}  \Fix T_{j}$, the result showing in \cref{theore:CM:WeakConver} may not  hold. For instance, consider  $\mathcal{H} =\mathbb{R}^{n}$, $\mathcal{S}=\{\Id\}$ and  $W \varsubsetneqq \mathbb{R}^{n}$ being closed and affine and $x \in \mathbb{R}^{n} \setminus W$. Then $\CC{\mathcal{S}}^{k}x \equiv x \not\to \Pro_{W}x$.
\end{remark}	

Let us now present sufficient conditions for the strong convergence 
of circumcentered isometry methods. 

\begin{theorem} \label{thm:CWP}
	Let $W$ be a nonempty closed affine subset of $\cap^{m}_{j=1}  \Fix T_{j}$, and let $x \in \mathcal{H}$.
	Then the following hold:
	\begin{enumerate}
		\item \label{prop:CWP:clust} If $(\CC{\mathcal{S}}^{k}x)_{k \in \mathbb{N}}$ has a norm cluster point in $W$, then $(\CC{\mathcal{S}}^{k}x)_{k \in \mathbb{N}}$ converges in norm to $\Pro_{W} (x)$.
		\item  \label{prop:CWP:norm:conv} The following  are equivalent:
		\begin{enumerate}
			\item \label{conv:it:it:0} $(\CC{\mathcal{S}}^{k}x)_{k \in \mathbb{N}}$  converges in norm to $\Pro_{W} (x)$.
			\item \label{conv:it:it:1} $(\CC{\mathcal{S}}^{k}x)_{k \in \mathbb{N}}$  converges in norm to some point in $W$.
			\item \label{conv:it:it:2} $(\CC{\mathcal{S}}^{k}x)_{k \in \mathbb{N}}$ has norm cluster points, all lying in $W$.
			\item \label{conv:it:it:3} $(\CC{\mathcal{S}}^{k}x)_{k \in \mathbb{N}}$  has norm cluster points, one lying in $W$.
		\end{enumerate}
	\end{enumerate}
\end{theorem}
\begin{proof}
	\cref{prop:CWP:clust}: Assume $\overline{x} \in W$ is a norm cluster point of $(\CC{\mathcal{S}}^{k}x)_{k \in \mathbb{N}}$, that is, there exists a subsequence $(\CC{\mathcal{S}}^{k_{j}}x)_{j \in \mathbb{N}}$ of $(\CC{\mathcal{S}}^{k}x)_{k \in \mathbb{N}}$ such that $\lim_{j \rightarrow \infty}\CC{\mathcal{S}}^{k_{j}}x$ $=$ $\overline{x}$. Now for every $j \in \mathbb{N}$,
	\begin{align*}
	\norm{\CC{\mathcal{S}}^{k_{j}}x-\Pro_{W} x} &= \norm{\CC{\mathcal{S}}^{k_{j}}x-\Pro_{W} (\CC{\mathcal{S}}^{k_{j}}x)} \quad \text{(by \cref{prop:CCS:Equations}\cref{prop:CCS:Equations:CCSxk})}\\
	& \leq \norm{\CC{\mathcal{S}}^{k_{j}}x-\overline{x}}. \quad  (\text{since}~\overline{x} \in W)
	\end{align*}
	So
	\begin{align*}
	0 \leq \lim_{j \rightarrow \infty} \norm{\CC{\mathcal{S}}^{k_{j}}x-\Pro_{W} (x)} \leq \lim_{j \rightarrow \infty} \norm{\CC{\mathcal{S}}^{k_{j}}x-\overline{x}}=0.
	\end{align*}
	Hence, $ \lim_{j \rightarrow + \infty} \CC{\mathcal{S}}^{k_{j}}x = \Pro_{W} (x)$.
	
	Substitute $z$ in \cref{prop:CW:NormPreserv:Feje}\cref{prop:CW:NormPreserv:Feje:limitexits} by $\Pro_{W} x$, then we know that $
	\lim_{k \rightarrow + \infty} \norm{\CC{\mathcal{S}}^{k}x- \Pro_{W} x }
	$  exists.
	Hence,
	\begin{align*}
	\lim_{k \rightarrow + \infty} \norm{\CC{\mathcal{S}}^{k}x-\Pro_{W} x}  = \lim_{j \rightarrow + \infty} \norm{\CC{\mathcal{S}}^{k_{j}}x-\Pro_{W} x} =0,
	\end{align*}
	from which follows that $(\CC{\mathcal{S}}^{k}x)_{k \in \mathbb{N}}$ converges strongly to $\Pro_{W}x$.
	
	\cref{prop:CWP:norm:conv}: By \cref{prop:CW:NormPreserv:Feje} \cref{prop:CW:NormPreserv:Feje:subset}, $(\CC{\mathcal{S}}^{k}x)_{k \in \mathbb{N}}$ is a Fej\'er monotone sequence with respect to $W$. Then the equivalences follow from \cite[Theorem~2.16(v)]{BB1996} and \cref{prop:CWP:clust} above.
\end{proof}

To facilitate a later proof, we provide the following lemma.
\begin{lemma} \label{lem:CCS:CCSK}
	Let $\varnothing \neq W \subseteq \cap^{m}_{j=1}  \Fix T_{j}$ such that $W$ is closed and affine. Assume there exists $\gamma \in [0,1[$ such that
	\begin{align}  \label{CWP:eq:ind:BaseCase}
	(\forall x \in \mathcal{H}) \quad \norm{\CC{\mathcal{S}}x-\Pro_{W} x} \leq \gamma \norm{x -\Pro_{W} x}.
	\end{align}
	Then
	\begin{align*}
	(\forall x \in \mathcal{H}) \quad (\forall k \in \mathbb{N}) \quad \norm{\CC{\mathcal{S}}^{k}x -\Pro_{W} x} \leq \gamma^{k} \norm{x -  \Pro_{W} x}.
	\end{align*}
\end{lemma}

\begin{proof}
	Let $x \in \mathcal{H}$.  For $k=0$, the result is trivial.
	
	Assume for some $k \in \mathbb{N}$ we have
	\begin{align} \label{CWP:eq:ind:L}
	(\forall y \in \mathcal{H}) \quad \norm{\CC{\mathcal{S}}^{k}y-\Pro_{W} y} \leq \gamma^{k} \norm{y -\Pro_{W} y}.
	\end{align}
	Now
	\begin{align*}
	\norm{\CC{\mathcal{S}}^{k+1}x-\Pro_{W} x}
	&= \norm{\CC{\mathcal{S}}(\CC{\mathcal{S}}^{k}x)-\Pro_{W} (\CC{\mathcal{S}}^{k}x)} \quad (\text{by \cref{prop:CCS:Equations}\cref{prop:CCS:Equations:CCSxk}})\\
	&\stackrel{ \cref{CWP:eq:ind:BaseCase} }{\leq} \gamma \norm{\CC{\mathcal{S}}^{k}x-\Pro_{W} (\CC{\mathcal{S}}^{k}x)} \\
	&= \gamma \norm{\CC{\mathcal{S}}^{k}x-\Pro_{W} x} \quad (\text{by \cref{prop:CCS:Equations}\cref{prop:CCS:Equations:CCSxk}})\\
	&\stackrel{\text{\cref{CWP:eq:ind:L}}}{\leq}  \gamma^{k+1} \norm{x -\Pro_{W} x}.
	\end{align*}
	Hence, we obtain the desired result inductively.
\end{proof}

The following powerful result will play an essential role to prove the linear convergence of the circumcenter method induced by reflectors.
\begin{theorem}  \label{thm:CWP:line:conv}
	Let $W$ be a nonempty, closed and affine subspace of $\cap^{m}_{j=1}  \Fix T_{j}$. 	
	\begin{enumerate}
		\item \label{thm:CWP:line:conv:F} Assume that there exist $F: \mathcal{H} \to \mathcal{H}$ and $\gamma \in [0,1[$ such that $(\forall y \in \mathcal{H})$ $F(y) \in \aff(\mathcal{S}(y) )$ and
		\begin{equation} \label{CWP:eq:F}
		(\forall x \in \mathcal{H}) \quad \norm{Fx-\Pro_{W} x} \leq \gamma \norm{x -\Pro_{W} x}.
		\end{equation}
		Then
		\begin{align}  \label{ineq:CCSkx}
		(\forall x \in \mathcal{H})\quad (\forall k \in \mathbb{N}) \quad \norm{\CC{\mathcal{S}}^{k}x-\Pro_{W} x} \leq \gamma^{k} \norm{x -\Pro_{W} x}.
		\end{align}
		Consequently, $(\CC{\mathcal{S}}^{k}x)_{k \in \mathbb{N}}$ converges linearly to $\Pro_{W} x$ with a linear rate $\gamma$.
		
		\item \label{thm:CWP:line:conv:T} If there exist $T_{\mathcal{S}} \in \aff (\mathcal{S})$ and $\gamma \in [0,1[$, such that
		\begin{equation*}
		(\forall x \in \mathcal{H}) \quad \norm{T_{\mathcal{S}}x-\Pro_{W} x} \leq \gamma \norm{x -\Pro_{W} x},
		\end{equation*}
		then $(\CC{\mathcal{S}}^{k}x)_{k \in \mathbb{N}}$ converges linearly to $\Pro_{W} x$ with a linear rate $\gamma$.
	\end{enumerate}
\end{theorem}

\begin{proof}
	\cref{thm:CWP:line:conv:F}: Using the assumptions and applying \cref{lem:CCS:Equations}\cref{lem:CCS:Equations:F} with $(\forall x \in \mathcal{H} )$ $z=\Pro_{ W}x$, we obtain that
	\begin{align*}
	(\forall x \in \mathcal{H}) \quad \norm{\CC{\mathcal{S}}x-\Pro_{W} x} \leq \norm{Fx-\Pro_{W} x} \stackrel{\cref{CWP:eq:F}}{\leq} \gamma \norm{x -\Pro_{W} x}.
	\end{align*}
	Hence, \cref{ineq:CCSkx} follows directly from \cref{lem:CCS:CCSK}.
	
	\cref{thm:CWP:line:conv:T}:	Since $T_{\mathcal{S}} \in \aff (\mathcal{S})$ implies that $(\forall y \in \mathcal{H})$ $T_{\mathcal{S}}y \in \aff (\mathcal{S}(y))$, thus the required result follows from \cref{thm:CWP:line:conv:F} above by substituting $F=T_{\mathcal{S}}$.
\end{proof}

\begin{theorem} \label{theorem:CCSLinearConvTSFirmNone}
	Let $T_{\mathcal{S}} \in \aff (\mathcal{S})$  satisfy that $ \Fix T_{\mathcal{S}} \subseteq \cap_{T \in \mathcal{S}} \Fix T$.  Then the following  hold:
	\begin{enumerate}
		\item \label{theorem:CCSLinearConvTSFirmNone:FixT} $\Fix T_{\mathcal{S}} = \cap_{T \in \mathcal{S}} \Fix T$.
		\item \label{theorem:CCSLinearConvTSFirmNone:LineaConv} Let $\mathcal{H}= \mathbb{R}^{n}$.  Assume that $T_{\mathcal{S}} $ is linear and $\alpha$-averaged with $\alpha \in \, ]0,1[\,$.
		For every $x \in \mathcal{H}$, $ (\CC{\mathcal{S}}^{k}x)_{k \in \mathbb{N}}$ converges  to $ \Pro_{ \cap_{T \in \mathcal{S}} \Fix T }x$ with a linear rate $\norm{T_{\mathcal{S}} \Pro_{(\cap_{T \in \mathcal{S}} \Fix T)^{\perp} }} \in [0,1[\,$.
	\end{enumerate}
\end{theorem}

\begin{proof}
	\cref{theorem:CCSLinearConvTSFirmNone:FixT}: Clearly, $T_{\mathcal{S}} \in \aff (\mathcal{S})$  implies that $\cap_{T \in \mathcal{S}} \Fix T \subseteq \Fix T_{\mathcal{S}}$.
	Combining the result with the assumption, $ \Fix T_{\mathcal{S}} \subseteq \cap_{T \in \mathcal{S}} \Fix T$, we get \cref{theorem:CCSLinearConvTSFirmNone:FixT}.
	
	\cref{theorem:CCSLinearConvTSFirmNone:LineaConv}: Since $T_{\mathcal{S}}$ is linear and $\alpha$-averaged, thus by \cref{fact:Nonexpan:Fix}, $\Fix T_{\mathcal{S}}$ is a nonempty closed linear subspace.  It is clear that
	\begin{align} \label{cor:TFixPP}
	T_{\mathcal{S}} \Pro_{ \Fix T_{\mathcal{S}} } = \Pro_{ \Fix T_{\mathcal{S}} }.
	\end{align}
	Using \cref{prop:LinFirQuaNonOperNorm}, we know
	\begin{align} \label{cor:CCSLinearConvTSFirmNone:TSP}
	\gamma :=\norm{ T_{\mathcal{S}}\Pro_{(\Fix T_{\mathcal{S}})^{\perp}} }<1.
	\end{align}
	Now for every $x \in \mathbb{R}^{n}$,
	\begin{align*}
	\norm{T_{\mathcal{S}}x-\Pro_{ \Fix T_{\mathcal{S}} } x}
	& \stackrel{\cref{cor:TFixPP}}{=}  \norm{T_{\mathcal{S}}x-T_{\mathcal{S}}\Pro_{ \Fix T_{\mathcal{S}} } x} \\
	&= \norm{T_{\mathcal{S}} (x-\Pro_{ \Fix T_{\mathcal{S}} } x)} \quad (T_{\mathcal{S}}~\text{linear}) \\
	&= \norm{T_{\mathcal{S}} \Pro_{(\Fix T_{\mathcal{S}} )^{\perp}} (x)} \quad (\text{by \cref{MetrProSubs8}\cref{MetrProSubs8:ii}})\\
	&= \norm{T_{\mathcal{S}} \Pro_{(\Fix T_{\mathcal{S}} )^{\perp}} \Pro_{(\Fix T_{\mathcal{S}} )^{\perp}} (x)} \quad \\
%	(\text{by \cref{fac:ProjecMonoNonexpan}\cref{fac:ProjecMonoNonexpan:i}})\\
	& \leq \norm{ T_{\mathcal{S}} \Pro_{(\Fix T_{\mathcal{S}} )^{\perp}} } \norm{\Pro_{(\Fix T_{\mathcal{S}} )^{\perp}} (x)} \\
	&=\gamma \norm{x-\Pro_{\Fix T_{\mathcal{S}} } (x)}. \quad (\text{by \cref{MetrProSubs8}\cref{MetrProSubs8:ii}})
	\end{align*}
	Hence, the desired result follows from \cref{thm:CWP:line:conv}\cref{thm:CWP:line:conv:T} by substituting $W =\Fix T_{\mathcal{S}}$ and \cref{theorem:CCSLinearConvTSFirmNone:FixT} above.
\end{proof}

Useful properties of the $T_{\mathcal{S}}$ in
\cref{theorem:CCSLinearConvTSFirmNone} can be found in the following results.

\begin{proposition} \label{prop:CCS:Equations:Affine}
	Let $\varnothing \neq W \subseteq \cap^{m}_{j=1}  \Fix T_{j}$ such that $W$ is a closed and affine subspace of $\mathcal{H}$ and let $T_{\mathcal{S}} \in \aff (\mathcal{S})$. Let $x \in \mathcal{H}$. Then
	\begin{enumerate}
		\item \label{prop:CCS:Equations:Affine:TSk}  $(\forall k \in \mathbb{N})$ $\Pro_{W} (T^{k}_{\mathcal{S}}x)=T^{k}_{\mathcal{S}}\Pro_{W} x=\Pro_{W} x$.
		\item \label{prop:CCS:Equations:Affine:CCsTs} $\norm{\Pro_{W} (\CC{\mathcal{S}}x) -\CC{\mathcal{S}}x}^{2} = \norm{\Pro_{W} (T_{\mathcal{S}}x) -T_{\mathcal{S}}x}^{2} -\norm{\CC{\mathcal{S}}x-T_{\mathcal{S}}x}^{2}$.
		\item  \label{prop:CCS:Equations:Affine:Dist} $ d(\CC{\mathcal{S}}x, W) =  \norm{\CC{\mathcal{S}}x - \Pro_{W} (x) }
		\leq  \norm{T_{\mathcal{S}}x - \Pro_{W} x}
		=  d(T_{\mathcal{S}}x, W)$.
	\end{enumerate}
	
\end{proposition}

\begin{proof}
	\cref{prop:CCS:Equations:Affine:TSk} : Denote $\I :=\{1,\ldots, m\}$. By assumption, $T_{\mathcal{S}} \in \aff (\mathcal{S})$, that is, there exist $(\alpha_{i})_{i\in \I} \in \mathbb{R}^{m}$ such that $\sum^{m}_{i=1} \alpha_{i} =1$ and $T_{\mathcal{S}} =\sum^{m}_{i=1} \alpha_{i} T_{i}$. By assumption, $W$ is closed and affine, thus by \cref{fact:PR}\cref{fact:PR:Affine}, $\Pro_{W}$ is affine. Hence, using \cref{prop:CCS:Equations}\cref{prop:CCS:Equations:T}, we obtain that
	\begin{align*}
	\Pro_{W}T_{\mathcal{S}}x = \Pro_{W} \left( \sum^{m}_{i=1} \alpha_{i} T_{i} x \right) =
	\sum^{m}_{i=1}  \alpha_{i} \Pro_{W}T_{i} x  = \sum^{m}_{i=1}  \alpha_{i} \Pro_{W}x = \Pro_{W}x .
	\end{align*}
	Using $T_{\mathcal{S}} \in \aff (\mathcal{S})$ again, we know $\Pro_{W}x \in W \subseteq \cap^{m}_{j=1}  \Fix T_{j} \subseteq \Fix T_{\mathcal{S}}$. So it is clear that $T_{\mathcal{S}}\Pro_{W}x= \Pro_{W}x$. Then \cref{prop:CCS:Equations:Affine:TSk} follows easily by induction on $k$.
	
	\cref{prop:CCS:Equations:Affine:CCsTs}: The result comes from \cref{prop:CCS:Equations}\cref{prop:CCS:Equations:CCSxk}, \cref{prop:CCS:Equations}\cref{prop:CCS:Equations:CCSTSP} and the item \cref{prop:CCS:Equations:Affine:TSk} above.
	
	\cref{prop:CCS:Equations:Affine:Dist}: The desired result follows from \cref{prop:CCS:Equations}\cref{prop:CCS:Equations:CCSxk} and from the \cref{prop:CCS:Equations:Affine:CCsTs} $\&$ \cref{prop:CCS:Equations:Affine:TSk}  above.
\end{proof}

\begin{remark}
Recall our global assumptions that $\mathcal{S}=\{
T_{1}, \ldots, T_{m-1}, T_{m} \}$ with $\cap^{m}_{j=1} \Fix T_{j} \neq
\varnothing$ and that every element of $\mathcal{S}$ is isometric. So, by
\cref{coro:averaNOTisometry}, for every $i \in \{1, \ldots, m\}$, if
$T_{i} \neq \Id$, $T_{i}$ is not averaged. Hence, we cannot construct the
operator $T_{\mathcal{S}}$ used in
\cref{theorem:CCSLinearConvTSFirmNone}\cref{theorem:CCSLinearConvTSFirmNone:LineaConv}
as in \cref{fact:AlphAvera}. 
See also 
\cref{prop:GeneLinConBCS} and \Cref{lem:AmP:TS,lem:AIdmP:TS} below for 
further examples of $T_{\mathcal{S}}$.
\end{remark}

\begin{remark}[relationship to \cite{BOyW2019Linear}]
In this present paper, we study systematically on the circumcentered isometry method. We first show that the circumcenter
mapping induced by isometries is proper which makes the circumcentered
isometry method well-defined and gives probability for any study on
circumcentered isometry methods. Then we consider the weak, strong and
linear convergence of the circumcentered isometry method. In addition, we
provide examples of linear convergent circumcentered reflection methods
in $\mathbb{R}^{n}$ and some applications of circumcentered reflection
methods. We also display performance profiles showing the outstanding
performance of two of our new circumcentered reflection methods without
theoretical proofs. 
The paper plays a fundamental role for our study of \cite{BOyW2019Linear}. 
In particular, \cref{thm:CWP:line:conv}\cref{thm:CWP:line:conv:F} and
\cref{theorem:CCSLinearConvTSFirmNone}\cref{theorem:CCSLinearConvTSFirmNone:LineaConv}
are two principal facts used in some proofs of \cite{BOyW2019Linear} which 
is an in-depth study of the linear convergence of circumcentered isometry methods.
Indeed, in \cite{BOyW2019Linear}, we first show the
corresponding linear convergent circumcentered isometry methods for
all of the linear convergent circumcentered reflection methods in
$\mathbb{R}^{n}$ shown in this paper. We provide two sufficient
conditions for the linear convergence of circumcentered isometry
methods in Hilbert spaces with first applying another operator on the
initial point. In fact, one of the sufficient conditions is inspired
by \cref{prop:CW:CCS:General:linearconv} in this paper. Moreover, we
present sufficient conditions for the linear convergence of
circumcentered reflection methods in Hilbert space. In addition, we
find some circumcentered reflection methods attaining the known
linear convergence rate of the accelerated symmetric MAP in Hilbert
spaces, which explains the dominant performance of the CRMs in the
numerical experiments in this paper. 
\end{remark}

\section{Circumcenter methods induced by reflectors} \label{sec:CircumMethodReflectors}

As \cref{lem:examples:normpreserving}\cref{lem:examples:normpreserving:R} showed, the reflector associated with any closed and affine subspace is isometry.
This section is devoted to study particularly the circumcenter method induced by reflectors.
In the whole section, we assume that $t \in \mathbb{N} \smallsetminus \{0\}$ and that
\begin{empheq}{equation*}
U_{1}, \ldots, U_{t}~\text{are closed affine subspaces in}~\mathcal{H}~\text{with}~\cap^{t}_{i=1} U_{i} \neq \varnothing,
\end{empheq}
and set that
\begin{empheq}[box=\mybluebox]{equation*}
\Omega := \Big\{ \R_{U_{i_{r}}}\cdots \R_{U_{i_{2}}}\R_{U_{i_{1}}}  ~\Big|~ r \in \mathbb{N}, ~\mbox{and}~ i_{1}, \ldots,  i_{r} \in \{1, \ldots,t\}    \Big\}.
\end{empheq}
Suppose $\mathcal{S}$ is a finite set such that
\begin{empheq}[box=\mybluebox]{equation*}
\varnothing \neq \mathcal{S} \subseteq \Omega.
\end{empheq}
We assume that
\begin{empheq}{equation*}
\R_{U_{i_{r}}}\cdots \R_{U_{i_{1}}}~\text{is the representative element of the set}~\mathcal{S}.
\end{empheq}
%\textcolor{purple}{For every affine subspace $U$, by \cref{fac:AffinePointLinearSpace}, there exists an unique linear subspace $\pa U$.}
 In order to prove some convergence results on the circumcenter methods induced by reflectors later, we consider the linear subspace $\pa U$ paralleling to the associated affine subspace $U$. 
 We denote
\begin{align} \label{defn:L}
L_{1} :=\pa U_{1}, \ldots, L_{t} := \pa U_{t}.
\end{align}
We set
\begin{empheq}[box=\mybluebox]{equation*}
\mathcal{S}_{L} :=  \left\{\R_{L_{i_{r}}}\cdots \R_{L_{i_{2}}}\R_{L_{i_{1}}}  ~|~ \R_{U_{i_{r}}}\cdots \R_{U_{i_{2}}}\R_{U_{i_{1}}}   \in \mathcal{S} \right\} .
\end{empheq}
Note that if $\Id \in \mathcal{S}$, then the corresponding element in $\mathcal{S}_{L}$ is $\Id$.

For example, if $\mathcal{S} =\{\Id, \R_{U_{1}}, \R_{U_{2}}\R_{U_{1}}, \R_{U_{3}}\R_{U_{1}}\}$, then $\mathcal{S}_{L}= \{\Id, \R_{L_{1}}, \R_{L_{2}}\R_{L_{1}}, \R_{L_{3}}\R_{L_{1}}\}$.

\subsection{Properties of circumcentered reflection methods} \label{subsec:ProperCCMethod}
\begin{lemma} \label{lemma:capUi}
	$\cap^{t}_{i=1} U_{i} $ is closed and affine.   Moreover,  $\varnothing \neq \cap^{t}_{i=1} U_{i} \subseteq \cap_{T \in \mathcal{S}} \Fix T$.
\end{lemma}
\begin{proof}
	By the underlying assumptions, $\cap^{t}_{i=1} U_{i}$ is closed and affine.
	
	Take an arbitrary but fixed  $\R_{U_{i_{r}}}\cdots \R_{U_{i_{1}}} \in \mathcal{S}$. If $\R_{U_{i_{r}}}\cdots \R_{U_{i_{1}}} = \Id$, then $\cap^{t}_{i=1} U_{i} \subseteq \mathcal{H} = \Fix \Id$. Assume $\R_{U_{i_{r}}}\cdots \R_{U_{i_{1}}} \neq \Id$. Let $x \in \cap^{t}_{i=1} U_{i}$. Since $(\forall j \in \{1, \ldots, t\})$ $\cap^{t}_{i=1} U_{i} \subseteq U_{j} = \Fix \R_{U_{j}}$, thus clearly $\R_{U_{i_{r}}}\cdots \R_{U_{i_{1}}}x =x$. Hence,  $\cap^{t}_{i=1} U_{i}  \subseteq \cap_{T \in \mathcal{S}} \Fix T$ as required.
\end{proof}

\cref{lemma:capUi} tells us that we are able to substitute the $W$ in
all of the results in \cref{sec:CircumcenterMethodIsome} by the $\cap^{t}_{i=1} U_{i} $. Therefore, the circumcenter methods induced by reflectors can be used in the best approximation problem associated with the intersection $\cap^{t}_{i=1} U_{i}$ of finitely many affine subspaces.

\begin{lemma} \label{lemma:RU:RL}
	Let $x\in \mathcal{H}$ and let $z \in \cap^{t}_{i=1} U_{i}$. Then the following   hold:
	\begin{enumerate}
		\item \label{lemma:RU:RL:R} $(\forall \R_{U_{i_{r}}}\cdots \R_{U_{i_{1}}} \in \mathcal{S})$ $\R_{U_{i_{r}}}\cdots \R_{U_{i_{1}}}x = z + \R_{L_{i_{r}}}\cdots \R_{L_{i_{1}}} (x-z)$.
		\item \label{lemma:RU:RL:S} $\mathcal{S}(x)=z + \mathcal{S}_{L}(x-z)$.
		\item \label{lemma:RU:RL:CCSk} $(\forall k \in \mathbb{N})$ $\CC{\mathcal{S}}^{k}x = z+ \CC{\mathcal{S}_{L}}^{k}(x-z)$.
	\end{enumerate}	
\end{lemma}	

\begin{proof}
	\cref{lemma:RU:RL:R}: Let $ \R_{U_{i_{r}}}\cdots \R_{U_{i_{1}}} \in \mathcal{S}$.
	Since for every $y \in \mathcal{H}$ and for every $ i \in \{1, \ldots, t\}$, $\R_{U_{i}}y = \R_{z+L_{i}}y =2\Pro_{z+L_{i}}y  -y= 2 \left( z+ \Pro_{L_{i}}(y-z) \right)-y =z+ \left( 2 \Pro_{L_{i}}(y-z)  - (y-z) \right) =z + \R_{L_{i}}(y-z)$, where the third and the fifth equality is by using \cref{fac:SetChangeProje}, thus
	\begin{align}  \label{eq:lemma:RU:RL:i}
	(\forall y \in \mathcal{H}) \quad 	(\forall i \in \{1, \ldots, t\}) \quad \R_{U_{i}}y = z + \R_{L_{i}}(y-z).
	\end{align}
	Then assume for some $k \in \{1,\ldots,r-1\}$,
	\begin{align} \label{eq:lemma:RU:RL:Rk}
	\R_{U_{i_{k}}}\cdots \R_{U_{i_{1}}}x = z + \R_{L_{i_{k}}}\cdots \R_{L_{i_{1}}} (x-z).
	\end{align}
	Now
	\begin{align*}
	\R_{U_{i_{k+1}}}\R_{U_{i_{k}}}\cdots \R_{U_{i_{1}}}x & \stackrel{\cref{eq:lemma:RU:RL:Rk}}{=} \R_{U_{i_{k+1}}} \left(
	z + \R_{L_{i_{k}}}\cdots \R_{L_{i_{1}}} (x-z) \right) \\
	&\stackrel{\cref{eq:lemma:RU:RL:i}}{=} z + \R_{L_{i_{k+1}}}\R_{L_{i_{k}}}\cdots \R_{L_{i_{1}}} (x-z).
	\end{align*}
	Hence, by induction, we know \cref{lemma:RU:RL:R} is true.
	
	\cref{lemma:RU:RL:S}: Combining the result proved in \cref{lemma:RU:RL:R}
	above with the definitions of the set-valued operator $\mathcal{S}$ and
	$\mathcal{S}_{L}$, we obtain 
	\begin{align*}
	\mathcal{S}(x) & = \left\{\R_{U_{i_{r}}}\cdots \R_{U_{i_{2}}}\R_{U_{i_{1}}}x  ~|~ \R_{U_{i_{r}}}\cdots \R_{U_{i_{2}}}\R_{U_{i_{1}}}   \in \mathcal{S} \right\} \\
	&=  \left\{ z + \R_{L_{i_{r}}}\cdots \R_{L_{i_{2}}}\R_{L_{i_{1}}}(x-z)  ~|~ \R_{U_{i_{r}}}\cdots \R_{U_{i_{2}}}\R_{U_{i_{1}}}   \in \mathcal{S} \right\} \\
	& = z + \left\{ \R_{L_{i_{r}}}\cdots \R_{L_{i_{2}}}\R_{L_{i_{1}}}(x-z)  ~|~ \R_{U_{i_{r}}}\cdots \R_{U_{i_{2}}}\R_{U_{i_{1}}}   \in \mathcal{S} \right\} \\
	& = z + \mathcal{S}_{L}(x-z).
	\end{align*}
	
	\cref{lemma:RU:RL:CCSk}: By \cite[Proposition~6.3]{BOyW2018}, for every  $K \in \mathcal{P}(\mathcal{H})$ and $y \in \mathcal{H}$,
	$\CCO(K+y)=\CCO(K)+y$. Because $z \in \cap^{t}_{i=1} U_{i} \subseteq \cap_{T \in \mathcal{S}} \Fix T$, by \cref{def:cir:map},
%	\cref{fact:CircumSubaddi},
	\begin{align} \label{eq:lemma:RU:RL:CCSk:base}
	(\forall y \in \mathcal{H})\quad	\CC{\mathcal{S}}y = \CCO{\left(\mathcal{S}(y) \right) }  \stackrel{\text{\cref{lemma:RU:RL:S}}}{=} \CCO{\left(z + \mathcal{S}_{L}(y-z) \right)} = z + \CCO{\left(\mathcal{S}_{L}(y-z) \right) }  = z + \CC{\mathcal{S}_{L}}(y-z).
	\end{align}
	Assume for some $k \in \mathbb{N}$,
	\begin{align} \label{eq:lemma:RU:RL:CCSk:AssCCS}
	(\forall y \in \mathcal{H}) \quad \CC{\mathcal{S}}^{k}y = z+ \CC{\mathcal{S}_{L}}^{k}(y-z).
	\end{align}
	Now
	\begin{align*}
	\CC{\mathcal{S}}^{k+1}x & =\CC{\mathcal{S}} \left(\CC{\mathcal{S}}^{k}x  \right) \\
	& =  \CC{\mathcal{S}} \left( z+ \CC{\mathcal{S}_{L}}^{k}(x-z) \right) \quad (\text{by \cref{eq:lemma:RU:RL:CCSk:AssCCS}})\\
	& =z + \CC{\mathcal{S}_{L}} \left( z+ \CC{\mathcal{S}_{L}}^{k}(x-z) - z\right) \quad (\text{by \cref{eq:lemma:RU:RL:CCSk:base} }) \\
	& =  z+ \CC{\mathcal{S}_{L}}^{k+1}(x-z).
	\end{align*}
	Hence, by induction, we know \cref{lemma:RU:RL:CCSk}  is true.
\end{proof}
The following \cref{prop:LineConvCCS:CCSL} says that the convergence of the circumcenter methods induced by reflectors associated with linear subspaces is equivalent to the convergence of the corresponding circumcenter methods induced by reflectors associated with affine subspaces. In fact, \cref{prop:LineConvCCS:CCSL} is a generalization of  \cite[Corollary~3]{BCS2017}.

\begin{proposition} \label{prop:LineConvCCS:CCSL}
	Let $x\in \mathcal{H}$ and let $z \in \cap^{t}_{i=1} U_{i}$.  Then $\left(\CC{\mathcal{S}}^{k}x \right)_{k \in \mathbb{N}}$ converges to $\Pro_{ \cap^{t}_{i=1} U_{i} }x $ $($with a linear rate $\gamma \in \, [0,1[\,$$)$  if and only if  $\left( \CC{\mathcal{S}_{L}}^{k}(x-z) \right)_{k \in \mathbb{N}}$ converges to $\Pro_{ \cap^{t}_{i=1} L_{i} }(x-z)$ $($with a linear rate $\gamma \in \, [0,1[\,$$)$.
\end{proposition}	

\begin{proof}
	By \cref{lemma:RU:RL}\cref{lemma:RU:RL:CCSk}, we know that $(\forall k \in \mathbb{N})$ $\CC{\mathcal{S}}^{k}x = z+ \CC{\mathcal{S}_{L}}^{k}(x-z)$. Moreover, by \cref{fac:SetChangeProje}, $\Pro_{ \cap^{t}_{i=1} U_{i} }x = \Pro_{ z +\cap^{t}_{i=1} L_{i} }x  = z+ \Pro_{ \cap^{t}_{i=1} L_{i} }(x-z)$. Hence, the equivalence holds.
\end{proof}

The proof of 
\cref{prop:CCSxxIntUPerp} requires the following result. 

\begin{lemma} \label{lem:ElemSinIntUperp}
	Let $x \in \mathcal{H}$ and let $\R_{U_{i_{r}}}\cdots \R_{U_{i_{1}}} \in \mathcal{S}$. Let
	$L_{1}, L_{2}, \ldots, L_{t}$ be the closed linear subspaces defined in \cref{defn:L}.
	Then $\R_{U_{i_{r}}}\cdots \R_{U_{i_{1}}}x-x \in (\cap^{t}_{i=1} L_{i})^{\perp}$, that is,
	\begin{align*}
	(\forall z \in \cap^{t}_{i=1} L_{i}) \quad \innp{\R_{U_{i_{r}}}\cdots \R_{U_{i_{1}}}x-x, z}=0.
	\end{align*}
\end{lemma}

\begin{proof}
	By \cref{lemma:RU:RL}\cref{lemma:RU:RL:R}, for every $z \in \cap^{t}_{i=1} L_{i}$,
	\begin{align*}
	\innp{\R_{U_{i_{r}}}\cdots \R_{U_{i_{1}}}x-x, z}=  \innp{z + \R_{L_{i_{r}}}\cdots \R_{L_{i_{1}}}(x-z)-x, z}= \innp{ \R_{L_{i_{r}}}\cdots \R_{L_{i_{1}}}(x-z)-(x-z), z}.
	\end{align*}
	Hence, it suffices to prove
	\begin{align}
	(\forall y \in \mathcal{H}) \quad (\forall z \in \cap^{t}_{i=1} L_{i})   \quad \innp{\R_{L_{i_{r}}}\cdots \R_{L_{i_{1}}}y-y, z}=0.
	\end{align}
	Let $y \in \mathcal{H}$ and $z\in \cap^{t}_{i=1} L_{i}$.  Take an arbitrary $j \in \{1,2, \ldots,t\}$. By \cref{MetrProSubs8}\cref{MetrProSubs8:ii}
$
	\innp{\R_{L_{j}}(y)-y,z}=\innp{2(\Pro_{L_{j}}-\Id )y,z}= \innp{-2\Pro_{L_{j}^{\perp}}y,z} =0,
$
	which yields that
	\begin{align}\label{eq:prop:orth}
	(\forall w \in \mathcal{H}) \quad \left(\forall d \in \{1,2, \ldots,t\}\right) \quad \innp{\R_{L_{d}}(w)-w, z}=0.
	\end{align}
	Recall $\prod^{0}_{j=1}\R_{L_{i_{j}}} =\Id$. So we have
	\begin{align} \label{eq:lem:ElemSinIntUperp:RrrX}
	\R_{L_{i_{r}}}\R_{L_{i_{r-1}}}\cdots \R_{L_{i_{1}}}(y) -y
	= \sum^{r-1}_{j=0}  \Big(\R_{L_{i_{j+1}}}\R_{L_{i_{j}}}\cdots \R_{L_{i_{1}}}(y)- \R_{L_{i_{j}}}\cdots \R_{L_{i_{1}}}(y) \Big).
	\end{align}
	Hence,
	\begin{align*}
	\Innp{\R_{L_{i_{r}}}\R_{L_{i_{r-1}}}\cdots \R_{L_{i_{1}}}(y)-y,z }
	\stackrel{\cref{eq:lem:ElemSinIntUperp:RrrX}}{=} & \Innp{ \sum^{r-1}_{j=0}  \Big(\R_{L_{i_{j+1}}}\R_{L_{i_{j}}}\cdots \R_{L_{i_{1}}}(y)- \R_{L_{i_{j}}}\cdots \R_{L_{i_{1}}}(y) \Big),z } \\
	= & \sum^{r-1}_{j=0}  \Innp{ \R_{L_{i_{j+1}}}\big(\R_{L_{i_{j}}}\cdots \R_{L_{i_{1}}}(y)\big)- \R_{L_{i_{j}}}\cdots \R_{L_{i_{1}}}(y), z }\\
	\stackrel{\text{\cref{eq:prop:orth}}}{=} & 0.
	\end{align*}
	Hence, the proof is complete.
\end{proof}

\begin{proposition}  \label{prop:CCSxxIntUPerp}
Assume $\Id \in \mathcal{S}$. Let $L_{1}, L_{2},
\ldots, L_{t}$ be the closed linear subspaces defined in \cref{defn:L}.
Let $x \in \mathcal{H}$. Then the following hold:
	\begin{enumerate}
		\item \label{prop:CCSxxIntUPerp:1} $\CC{\mathcal{S}}x-x \in (\cap^{t}_{i=1} L_{i})^{\perp}$, that is,
		$ (\forall z \in\cap^{t}_{i=1} L_{i})$ $\innp{\CC{\mathcal{S}}x-x,z}=0$.
		\item  \label{prop:CCSxxIntUPerp:k} $(\forall k \in \mathbb{N}) $ $ \CC{\mathcal{S}}^{k}x-x \in (\cap^{t}_{i=1} L_{i})^{\perp},
		$
		that is,
		\begin{align} \label{eq:CCSK:prop:orth}
		(\forall k \in \mathbb{N}) \quad (\forall z \in\cap^{t}_{i=1} L_{i}) \quad \innp{\CC{\mathcal{S}}^{k}x-x,z}=0.
		\end{align}
	\end{enumerate}
\end{proposition}

\begin{proof}
	\cref{prop:CCSxxIntUPerp:1}: By \cref{thm:CCS:proper:NormPres:T}\cref{thm:CCS:proper:NormPres:T:prop}, we know that $\CC{\mathcal{S}}$ is proper. Hence, by \cref{prop:CW:Welldefined:Formula} and $\Id \in \mathcal{S}$, there exist $n \in \mathbb{N}$ and $\alpha_{1}, \ldots, \alpha_{n} \in \mathbb{R}$ and $T_{1}, \ldots, T_{n} \in \mathcal{S}$ such that
	\begin{align} \label{eq:prop:orth:CW:Welldefined:Formula}
	\CC{\mathcal{S}}x=x+
	\sum^{n}_{j=1} \alpha_{j} (T_{j}x-x).
	\end{align}
	Let $z \in \cap^{t}_{i=1} L_{i}$.
	Since $\{T_{1}, \ldots, T_{n}\} \subseteq \mathcal{S}$, by  \cref{lem:ElemSinIntUperp},
	$
	\sum^{n}_{j=1} \alpha_{j} \innp{ T_{j}x-x,z} =0.
	$
	Therefore,
	\begin{align*}
	\innp{\CC{\mathcal{S}}x-x,z} \stackrel{\cref{eq:prop:orth:CW:Welldefined:Formula}}{=} \sum^{n}_{j=1} \alpha_{j} \innp{T_{j}x-x,z}= 0.
	\end{align*}
	Hence, \cref{prop:CCSxxIntUPerp:1} is true.
	
	\cref{prop:CCSxxIntUPerp:k}: When $k=0$, \cref{eq:CCSK:prop:orth} is trivial. By \cref{prop:CCSxxIntUPerp:1},
	\begin{align} \label{eq:CW:innp:assum}
	(\forall y \in \mathcal{H}) \quad (\forall z \in \cap^{t}_{i=1} L_{i}) \quad \innp{\CC{\mathcal{S}}y-y,z}=0.
	\end{align}
	Then for every $k \in \mathbb{N} \smallsetminus \{0\}$, and for every $z \in \cap^{m}_{i=1} L_{i}$,
	\begin{align*}
	\innp{\CC{\mathcal{S}}^{k}x-x,z} & = \Innp{\sum^{k-1}_{i=0}\big(\CC{\mathcal{S}}^{i+1}(x)-\CC{\mathcal{S}}^{i}(x) \big),z}\\
	& = \Innp{\sum^{k-1}_{i=0}\Big(\CC{\mathcal{S}}(\CC{\mathcal{S}}^{i}(x))-\CC{\mathcal{S}}^{i}(x) \Big),z}\\
	& = \sum^{k-1}_{i=0} \Innp{\CC{\mathcal{S}}(\CC{\mathcal{S}}^{i}(x))-\CC{\mathcal{S}}^{i}(x),z}\\
	& \stackrel{\text{\cref{eq:CW:innp:assum}}}{=}0.
	\end{align*}
	Hence, \cref{prop:CCSxxIntUPerp:k}  holds.
\end{proof}

%\begin{proposition}\label{prop:orth}
%	Assume $\Id \in \mathcal{S}$.	Let
%	$L_{1}, L_{2}, \ldots, L_{t}$ be the closed linear subspaces defined in \cref{defn:L}. Let $x \in \mathcal{H}$. Then
%	$(\forall k \in \mathbb{N}) $ $ \CC{\mathcal{S}}^{k}x-x \in (\cap^{t}_{i=1} L_{i})^{\perp},
%	$
%	that is,
%	\begin{align} \label{eq:CCSK:prop:orth}
%	(\forall k \in \mathbb{N}) \quad (\forall z \in\cap^{t}_{i=1} L_{i}) \quad \innp{\CC{\mathcal{S}}^{k}x-x,z}=0.
%	\end{align}
%\end{proposition}
%
%\begin{proof}
%	When $k=0$, \cref{eq:CCSK:prop:orth} is trivial. By \cref{prop:CCSxxIntUPerp},
%	\begin{align} \label{eq:CW:innp:assum}
%	(\forall y \in \mathcal{H}) \quad (\forall z \in \cap^{t}_{i=1} L_{i}) \quad \innp{\CC{\mathcal{S}}y-y,z}=0.
%	\end{align}
%	Then for every $k \in \mathbb{N} \smallsetminus \{0\}$, and for every $z \in \cap^{m}_{i=1} L_{i}$,
%	\begin{align*}
%	\innp{\CC{\mathcal{S}}^{k}x-x,z} & = \Innp{\sum^{k-1}_{i=0}\big(\CC{\mathcal{S}}^{i+1}(x)-\CC{\mathcal{S}}^{i}(x) \big),z}\\
%	& = \Innp{\sum^{k-1}_{i=0}\Big(\CC{\mathcal{S}}(\CC{\mathcal{S}}^{i}(x))-\CC{\mathcal{S}}^{i}(x) \Big),z}\\
%	& = \sum^{k-1}_{i=0} \Innp{\CC{\mathcal{S}}(\CC{\mathcal{S}}^{i}(x))-\CC{\mathcal{S}}^{i}(x),z}\\
%	& \stackrel{\text{\cref{eq:CW:innp:assum}}}{=}0.
%	\end{align*}
%	Hence, we are done.
%\end{proof}

\begin{remark} Assume $\Id \in \mathcal{S}$. Let $x \in \mathcal{H}$, and let $k \in \mathbb{N}$. Then
	\begin{align*}
	\Pro_{ \cap^{t}_{i=1} U_{i}} x -\Pro_{ \cap^{t}_{i=1} U_{i}} \CC{\mathcal{S}}^{k}x & =
	z + \Pro_{ \cap^{t}_{i=1} L_{i}} (x-z) -z-\Pro_{ \cap^{t}_{i=1} L_{i}} (\CC{\mathcal{S}}^{k}(x)-z) \quad (\text{by  \cref{fac:SetChangeProje}})\\
	&= \Pro_{ \cap^{t}_{i=1} L_{i}} (x-z) -\Pro_{ \cap^{t}_{i=1} L_{i}} \CC{\mathcal{S}_{L}}^{k}(x-z)  \quad (\text{ by \cref{lemma:RU:RL}\cref{lemma:RU:RL:CCSk}})\\
	& = \Pro_{ \cap^{t}_{i=1} L_{i}}  \left( (x-z) - \CC{\mathcal{S}}^{k}(x-z)  \right)  =0. \quad (\text{by \cref{prop:CCSxxIntUPerp}\cref{prop:CCSxxIntUPerp:k}})
	\end{align*}
	In fact, we proved $(\forall x \in \mathcal{H} )$ $\Pro_{ \cap^{t}_{i=1} U_{i}} \CC{\mathcal{S}}^{k}x =\Pro_{ \cap^{t}_{i=1} U_{i}} x $ which is a special case of \cref{prop:CCS:Equations}\cref{prop:CCS:Equations:CCSxk}.
\end{remark}

In the remainder of this subsection,  we
consider cases when the initial points of circumcentered isometry methods are 
drawn from special sets.

\begin{lemma} \label{lemma:com:affspan:U1Um}
Let $x$ be in $\mathcal{H}$. Then the following hold:
\begin{enumerate}
\item \label{lemma:com:affspan:U1Um:lem:AffSx:AffU1Um} Suppose $x \in \aff
(\cup^{t}_{i=1} U_{i})$. Then $\aff \mathcal{S}(x) \subseteq \aff
(\cup^{t}_{i=1} U_{i})$ and $(\forall k \in \mathbb{N})$
$\CC{\mathcal{S}}^{k}x \in \aff (\cup^{t}_{i=1} U_{i})$.
\item \label{lemma:com:affspan:U1Um:lem:SpnSx:SpanU1Um} Suppose $x \in \spn
(\cup^{t}_{i=1} U_{i})$. Then $\aff \mathcal{S}(x) \subseteq \spn
\mathcal{S}(x) \subseteq \spn (\cup^{t}_{i=1} U_{i})$ and $(\forall k \in
\mathbb{N})$ $\CC{\mathcal{S}}^{k}x \in \spn (\cup^{t}_{i=1} U_{i})$.
\end{enumerate}
\end{lemma}

\begin{proof}
	\cref{lemma:com:affspan:U1Um:lem:AffSx:AffU1Um}: Let $\R_{U_{i_{r}}}\cdots \R_{U_{i_{1}}}$ be an arbitrary but fixed element in $\mathcal{S}$. If $r=0$, $\R_{U_{i_{r}}}\cdots \R_{U_{i_{1}}}x =x \in \aff (\cup^{t}_{i=1} U_{i})$. Assume $r \geq 1$. Since $i_{1} \in \{1, \ldots, t\}$, $ \Pro_{U_{i_{1}}}x \in \aff (\cup^{t}_{i=1} U_{i})$. So
	\begin{align*}
	\R_{U_{i_{1}}} x = 2 \Pro_{U_{i_{1}}}x - x \in \aff (\cup^{t}_{i=1} U_{i}).
	\end{align*}
	Assume for some $j \in \{1, \ldots, r-1\}$,
	\begin{align*}
	\R_{U_{i_{j}}}\cdots \R_{U_{i_{1}}}x \in  \aff (\cup^{t}_{i=1} U_{i}).
	\end{align*}
	Now since $i_{j+1} \in \{1, \ldots, t\}$, thus $\Pro_{U_{i_{j+1}}}(\R_{U_{i_{j}}}\cdots \R_{U_{i_{1}}}x) \in \aff (\cup^{t}_{i=1} U_{i})$. Hence,
	\begin{align*}
	\R_{U_{i_{j+1}}}\R_{U_{i_{j}}}\cdots \R_{U_{i_{1}}}x = 2 \Pro_{U_{i_{j+1}}}(\R_{U_{i_{j}}}\cdots \R_{U_{i_{1}}}x) - \R_{U_{i_{j}}}\cdots \R_{U_{i_{1}}}x \in  \aff (\cup^{t}_{i=1} U_{i}).
	\end{align*}
	Hence, we have inductively proved $\R_{U_{i_{r}}}\cdots \R_{U_{i_{1}}}x \in  \aff (\cup^{t}_{i=1} U_{i})$.
	
	Since $\R_{U_{i_{r}}}\cdots \R_{U_{i_{1}}}x \in \mathcal{S}(x)$ is chosen
	arbitrarily, we conclude that $\mathcal{S}(x) \subseteq \aff (\cup^{t}_{i=1} U_{i})$
	which in turn yields $\aff \mathcal{S}(x) \subseteq \aff (\cup^{t}_{i=1} U_{i})$.
	
	Moreover, by \cref{thm:CCS:proper:NormPres:T}\cref{thm:CCS:proper:NormPres:T:prop},
	$
	\CC{\mathcal{S}}x \in \aff \mathcal{S}(x) \subseteq \aff (\cup^{t}_{i=1} U_{i}).
	$
	Therefore, an easy inductive argument deduce $(\forall k \in \mathbb{N})$ $\CC{\mathcal{S}}^{k}x  \in \aff (\cup^{t}_{i=1} U_{i})$.
	
	\cref{lemma:com:affspan:U1Um:lem:SpnSx:SpanU1Um}:  Using the similar technique showed in the proof of \cref{lemma:com:affspan:U1Um:lem:AffSx:AffU1Um}, we know that $x \in \spn (\cup^{t}_{i=1} U_{i})$ implies that 	$\mathcal{S}(x)   \subseteq \spn (\cup^{t}_{i=1} U_{i})$.
	The remaining part of the proof is similar with the proof in \cref{lemma:com:affspan:U1Um:lem:AffSx:AffU1Um}, so we omit it.
\end{proof}

\begin{corollary}\label{cor:CCSCommutePcapUi:PUiPerp}
	Assume $U_{1}, \ldots, U_{t}$ are closed linear subspaces in $\mathcal{H}$. Then the following hold:
	\begin{enumerate}
		\item \label{cor:CCSCommutePcapUi:PUiPerp:equa} $\CC{\mathcal{S}} \Pro_{(\cap^{t}_{i=1}U_{i})^{\perp}}= \CC{\mathcal{S}} - \Pro_{\cap^{t}_{i=1}U_{i}}= \Pro_{(\cap^{t}_{i=1}U_{i})^{\perp}} \CC{\mathcal{S}}$.
		
		\item \label{cor:CCSCommutePcapUi:PUiPerp:perp:X}  Let $x \in(\cap^{t}_{i=1} U_{i})^{\perp}$. Then $(\forall k \in \mathbb{N})$ $\CC{\mathcal{S}}^{k}x  \in (\cap^{t}_{i=1} U_{i})^{\perp}$.
	\end{enumerate}
\end{corollary}

\begin{proof}
	\cref{cor:CCSCommutePcapUi:PUiPerp:equa}: Let $x \in \mathcal{H}$. By \cref{MetrProSubs8}\cref{MetrProSubs8:ii}, we get $\Pro_{(\cap^{t}_{i=1}U_{i})^{\perp}} = \Id - \Pro_{\cap^{t}_{i=1}U_{i}}$. By \cref{lemma:capUi}, $- \Pro_{\cap^{t}_{i=1}U_{i}}x \in \cap^{t}_{i=1}U_{i} \subseteq \cap^{t}_{j=1} \Fix T_{j}$. Applying \cref{cor:HomogeAdditiveModulo}\cref{cor:HomogeAdditiveModulo:AdditiveModulo} with $z =- \Pro_{\cap^{t}_{i=1}U_{i}}x$, we obtain
	$\CC{\mathcal{S}}(x - \Pro_{\cap^{t}_{i=1}U_{i}}x) =\CC{\mathcal{S}}x - \Pro_{\cap^{t}_{i=1}U_{i}}x$. Hence,
	\begin{align} \label{cor:CCSCommutePcapUi:PUiPerp:left}
	\CC{\mathcal{S}}(\Pro_{(\cap^{t}_{i=1}U_{i})^{\perp}}x)=  \CC{\mathcal{S}}(x - \Pro_{\cap^{t}_{i=1}U_{i}}x) = \CC{\mathcal{S}}x - \Pro_{\cap^{t}_{i=1}U_{i}}x.
	\end{align}
	On the other hand, substituting $W= \cap^{t}_{i=1}U_{i}$ in \cref{prop:CCS:Equations}\cref{prop:CCS:Equations:CCSxk}, we obtain that
	\begin{align} \label{cor:CCSCommutePcapUi:PUiPerp:right}
	\Pro_{(\cap^{t}_{i=1}U_{i})^{\perp}} (\CC{\mathcal{S}}x) = \CC{\mathcal{S}}x - \Pro_{\cap^{t}_{i=1}U_{i}} \CC{\mathcal{S}}x = \CC{\mathcal{S}}x - \Pro_{\cap^{t}_{i=1}U_{i}}x.
	\end{align}
	Thus, \cref{cor:CCSCommutePcapUi:PUiPerp:left} and \cref{cor:CCSCommutePcapUi:PUiPerp:right} yield
	\begin{align*}
	\CC{\mathcal{S}} \Pro_{(\cap^{t}_{i=1}U_{i})^{\perp}}= \CC{\mathcal{S}} -  \Pro_{\cap^{t}_{i=1}U_{i}}= \Pro_{(\cap^{t}_{i=1}U_{i})^{\perp}} \CC{\mathcal{S}}.
	\end{align*}
	
	\cref{cor:CCSCommutePcapUi:PUiPerp:perp:X}:  By \cref{cor:CCSCommutePcapUi:PUiPerp:equa}, $\CC{\mathcal{S}}x = \CC{\mathcal{S}} \Pro_{(\cap^{t}_{i=1} U_{i})^{\perp}} x = \Pro_{(\cap^{t}_{i=1} U_{i})^{\perp}} \CC{\mathcal{S}}x \in (\cap^{t}_{i=1} U_{i})^{\perp}$, which implies that
	\begin{align*}
	(\forall y \in (\cap^{t}_{i=1} U_{i})^{\perp}) \quad  \CC{\mathcal{S}} y \in (\cap^{t}_{i=1} U_{i})^{\perp}.
	\end{align*}
	Hence, we obtain \cref{cor:CCSCommutePcapUi:PUiPerp:perp:X} by induction.
\end{proof}

The following example tells us that in \cref{cor:CCSCommutePcapUi:PUiPerp}\cref{cor:CCSCommutePcapUi:PUiPerp:equa}, the condition \enquote{$U_{1}, \ldots, U_{t}$ are linear subspaces in $\mathcal{H}$} is indeed necessary.
\begin{example}  \label{exam:cor:CCSCommutePcapUi}
	Assume $\mathcal{H} = \mathbb{R}^{2}$ and $U_{1}:=\{(x_{1},x_{2}) \in \mathbb{R}^{2} ~|~ x_{2}=1\}$ and $U_{2} := \{ (x_{1},x_{2}) \in \mathbb{R}^{2} ~|~ x_{2}=x_{1}+1\}$.  Assume $\mathcal{S} =\{\Id, \R_{U_{1}}, \R_{U_{2}}\}$. Let $x :=(1,0) $. Since $U_{1} \cap U_{2} =\{(0,1)\}$ and since $\left(  U_{1} \cap U_{2} \right)^{\perp}=\{(x_{1},x_{2}) \in \mathbb{R}^{2} ~|~ x_{2}=0\}$, thus
	\begin{align*}
	\CC{\mathcal{S}} \Pro_{(U_{1} \cap U_{2})^{\perp}}x =(0,1)  \neq (0,0)= \CC{\mathcal{S}}x - \Pro_{U_{1} \cap U_{2}}x= \Pro_{(U_{1} \cap U_{2})^{\perp}} \CC{\mathcal{S}}x.
	\end{align*}
\end{example}

\subsection{Linear convergence of circumcentered reflection methods} \label{sec:LinearConve:Reflector}

This subsection is motivated by \cite[Theorem~3.3]{BCS2018}. In particular,
\cite[Theorem~3.3]{BCS2018} is \cref{prop:GeneLinConBCS} below for the
special case when $\{ \Id, \R_{U_{1}}, \R_{U_{2}}\R_{U_{1}}, \ldots,
\R_{U_{t}}\R_{U_{t-1}}\cdots \R_{U_{2}}\R_{U_{1}}\} = \mathcal{S}$ and
$U_{1}, \ldots, U_{t}$ are linear subspaces. The operator $T_{\mathcal{S}}$
defined in the \cref{prop:GeneLinConBCS} below is the operator $A$ defined in
{\rm\cite[Lemma~2.1]{BCS2018}}.

\begin{proposition} \label{prop:GeneLinConBCS}
	Assume that $\mathcal{H}=\mathbb{R}^{n}$ and that
	\begin{align*}
	\{ \Id, \R_{U_{1}}, \R_{U_{2}}\R_{U_{1}}, \ldots, \R_{U_{t}}\R_{U_{t-1}}\cdots \R_{U_{2}}\R_{U_{1}}\} \subseteq \mathcal{S}.
	\end{align*}
	Let $L_{1}, \ldots, L_{t}$ be the closed linear subspaces defined in \cref{defn:L}.
	Define $T_{\mathcal{S}} : \mathbb{R}^{n} \to \mathbb{R}^{n}$ by
	$
	T_{\mathcal{S}} := \frac{1}{t} \sum^{t}_{i=1} T_{i},
	$
	where $T_{1} := \frac{1}{2}(\Id + \Pro_{L_{1}})$ and $(\forall i \in \{2, \ldots, t\})$ $T_{i} := \frac{1}{2} (\Id + \Pro_{L_{i}} \R_{L_{i-1}} \cdots R_{L_{1}})$.
	Let $x \in \mathcal{H}$. Then $(\CC{\mathcal{S}}^{k}x)_{k \in \mathbb{N}}$ converges to $\Pro_{\cap^{t}_{i=1} U_{i} }x$ with a linear rate $\norm{T_{\mathcal{S}} \Pro_{(\cap^{t}_{i=1} L_{i})^{\perp} }}$ $\in$ $[0,1[\,$.
\end{proposition}

\begin{proof}
	Now
	\begin{align*}
	T_{1} & = \frac{1}{2}(\Id + \Pro_{L_{1}})= \frac{1}{2} \Big(\Id +\frac{\Id + \R_{L_{1}}}{2}\Big) =\frac{3}{4}\Id +\frac{1}{4} R_{L_{1}} \\
	& \in \aff \{ \Id, \R_{L_{1}}, \R_{L_{2}}\R_{L_{1}}, \ldots, \R_{L_{t}}\R_{L_{t-1}}\cdots \R_{L_{2}}\R_{L_{1}}\} ,
	\end{align*}
	and for every $i \in \{2, \ldots, t\}$,
	\begin{align*}
	T_{i} & = \frac{1}{2} (\Id + \Pro_{L_{i}} \R_{L_{i-1}} \cdots \R_{L_{1}}) \\
	& =  \frac{1}{2} \left(\Id + \Big(\frac{\R_{L_{i}} + \Id }{2}\Big) \R_{L_{i-1}} \cdots \R_{L_{1}} \right) \\
	& =\frac{1}{2} \Id + \frac{1}{4} \R_{L_{i}}\R_{L_{i-1}} \cdots \R_{L_{1}} + \frac{1}{4} \R_{L_{i-1}} \cdots \R_{L_{1}} \\
	& \in \aff \{ \Id, \R_{L_{1}}, \R_{L_{2}}\R_{L_{1}}, \ldots, \R_{L_{t}}\R_{L_{t-1}}\cdots \R_{L_{2}}\R_{L_{1}}\},
	\end{align*}
	which yield that
	\begin{align*}
	T_{\mathcal{S}}=\frac{1}{t} \sum^{t}_{i=1} T_{i} \in \aff \{ \Id, \R_{L_{1}}, \R_{L_{2}}\R_{L_{1}}, \ldots, \R_{L_{t}}\R_{L_{t-1}}\cdots \R_{L_{2}}\R_{L_{1}}\} \subseteq \aff (\mathcal{S}_{L}).
	\end{align*}
	Using {\rm \cite[Lemma~2.1(i)]{BCS2018}}, we know the $T_{\mathcal{S}}$ is linear and $\frac{1}{2}$-averaged, and by \cite[Lemma~2.1(ii)]{BCS2018}, $\Fix T_{\mathcal{S}} = \cap^{t}_{i=1} L_{i}$. Hence, by \cref{theorem:CCSLinearConvTSFirmNone}\cref{theorem:CCSLinearConvTSFirmNone:LineaConv} and \cref{lemma:capUi}, we obtain that for every $y \in \mathcal{H}$, $ (\CC{\mathcal{S}_{L}}^{k}y)_{k \in \mathbb{N}}$ converges to $ \Pro_{ \cap^{t}_{i=1} L_{i} }y$ with a linear rate $\norm{T_{\mathcal{S}} \Pro_{(\cap^{t}_{i=1} L_{i})^{\perp} }}$ $\in$ $[0,1[\,$.
	Therefore, the desired result follows from \cref{prop:LineConvCCS:CCSL}.
\end{proof}

\begin{remark}
	In fact, \cite[Lemma~2.1(ii)]{BCS2018} is $\Fix T_{\mathcal{S}} = \cap^{t}_{i=1} L_{i}$.
	In the proof of \cite[Lemma~2.1(ii)]{BCS2018}, the authors claimed that \enquote{it is easy to see that $\Fix T_{i} =L_{i}$}. We provide more details here.  For every $i \in \{1,\ldots, m\}$, by \cite[Proposition~4.49]{BC2017}, we know that $\Fix T_{i} =\Fix \Pro_{L_{i}} \cap \Fix \R_{L_{i-1}}\cdots\R_{L_{1}} \subseteq L_{i}$. As \cite[Lemma~2.1(ii)]{BCS2018} proved that $\Fix T_{\mathcal{S}} \subseteq \cap^{m}_{i=1} \Fix T_{i}$, we get that $\Fix T_{\mathcal{S}}   \subseteq \cap^{m}_{i=1} L_{i}$. On the other hand, by definition of $T_{\mathcal{S}} $, we have  $ \cap^{m}_{i=1} L_{i} \subseteq \Fix T_{\mathcal{S}}$. Altogether, $\Fix T_{\mathcal{S}}  = \cap^{m}_{i=1} L_{i}$, which implies that \cite[Lemma~2.1(ii)]{BCS2018} is true.
\end{remark}	

The idea of the proofs in the following two lemmas is obtained from {\rm \cite[Lemma~2.1]{BCS2018}}.
\begin{lemma} \label{lem:AmP:TS}
	Assume that $\mathcal{H}=\mathbb{R}^{n}$ and that $\{\Id, \R_{U_{1}},\ldots, \R_{U_{t-1}},\R_{U_{t}}\} \subseteq \mathcal{S}$. Let $L_{1}, \ldots, L_{t}$ be the closed linear subspaces defined in \cref{defn:L}. Define the operator $T_{\mathcal{S}}: \mathbb{R}^{n} \to \mathbb{R}^{n}$ as $T_{\mathcal{S}} :=\frac{1}{t} \sum^{t}_{i=1}\Pro_{L_{i}}$.
	Then the following  hold:
	\begin{enumerate}
		\item \label{lem:AP:TS:AffHull} $T_{\mathcal{S}} \in \aff (\mathcal{S}_{L})$.
		\item \label{lem:AP:TS:LineFirmNone} $T_{\mathcal{S}}$ is linear and firmly nonexpansive.
		\item \label{lem:AP:TS:fix} $\Fix T_{\mathcal{S}}=\cap^{t}_{i=1} L_{i} = \cap_{F \in \mathcal{S}_{L}} \Fix F$.
	\end{enumerate}
\end{lemma}

\begin{proof}
	\cref{lem:AP:TS:AffHull}: Now $(\forall i \in \{1, \ldots,t\})$, $\Pro_{L_{i}}=\frac{\Id + \R_{L_{i}}} {2}$, so
	\begin{align*}
	T_{\mathcal{S}}=\frac{1}{t} \sum^{t}_{i=1}\Pro_{L_{i}} = \frac{1}{t} \sum^{t}_{i=1} \frac{\Id + \R_{L_{i}}} {2} \in \aff \{\Id, \R_{L_{1}},\ldots, \R_{L_{t-1}},\R_{L_{t}}\} \subseteq \aff (\mathcal{S}_{L}).
	\end{align*}
	
	\cref{lem:AP:TS:LineFirmNone}: Let $i \in \{1, \ldots, t\}$. 
%Using \cref{fac:ProjecMonoNonexpan}\cref{fac:ProjecMonoNonexpan:ii} and \cref{fact:AverFirmNone}\cref{fact:AverFirmNone:firm}, we know 
	Because $\Pro_{L_{i}}$ is firmly nonexpansive, it is $\frac{1}{2}$-averaged. Using \cref{fact:AlphAvera}, we know $T_{\mathcal{S}}$ $\frac{1}{2}$-averaged, that is, it is firmly nonexpansive. In addition, because $(\forall i \in \{1, \ldots, t\})$ $L_{i}$ is linear subspace implies that $\Pro_{L_{i}}$ is linear, we know that $T_{\mathcal{S}}$ is linear.
	
	\cref{lem:AIdP:TS:fix}: The projection is firmly
	nonexpansive, so it is quasinonexpansive. Hence, the result follows from
%	\cref{fac:ProjecMonoNonexpan}\cref{fac:ProjecMonoNonexpan:ii}, \cref{rem:NonexpImplication}, 
%\cref{fact:FixSumInters} and 
 \cite[Proposition~4.47]{BC2017} and 
\cref{theorem:CCSLinearConvTSFirmNone}\cref{theorem:CCSLinearConvTSFirmNone:FixT}.
\end{proof}

\begin{lemma} \label{lem:AIdmP:TS}
	Assume that $\mathcal{H}=\mathbb{R}^{n}$ and that $\{\Id, \R_{U_{1}},\ldots, \R_{U_{t-1}},\R_{U_{t}}\} \subseteq \mathcal{S}$.  Let $L_{1}, \ldots, L_{t}$ be the closed linear subspaces defined in \cref{defn:L}.  Define the operator $T_{\mathcal{S}}: \mathbb{R}^{n} \to \mathbb{R}^{n}$ by $T_{\mathcal{S}} :=\frac{1}{t} \sum^{t}_{i=1}T_{i}$, where $( \forall i \in \{1,2,\ldots,t\}) $  $T_{i} := \frac{1}{2} (\Id +\Pro_{L_{i}})$.
	Then
	\begin{enumerate}
		\item \label{lem:AIdP:TS:AffHull} $T_{\mathcal{S}} \in \aff (\mathcal{S}_{L})$.
		\item \label{lem:AIdP:TS:LineFirmNone} $T_{\mathcal{S}}$ is linear and firmly nonexpansive.
		\item \label{lem:AIdP:TS:fix} $\Fix T_{\mathcal{S}}=\cap^{t}_{i=1} L_{i} = \cap_{F \in \mathcal{S}_{L}} \Fix F$.
	\end{enumerate}
\end{lemma}

\begin{proof}
	\cref{lem:AIdP:TS:AffHull}: 
	Now for every $i \in \{1, \ldots,t\}$, $T_{i}=\frac{1}{2} (\Id
	+\Pro_{L_{i}})=\frac{1}{2} \big(\Id +\frac{\Id + \R_{L_{i}}}{2}\big)
	=\frac{3}{4}\Id +\frac{1}{4} \R_{L_{i}}$. Hence,
	\begin{align*}
	T_{\mathcal{S}}=\frac{1}{t} \sum^{t}_{i=1}T_{i} = \frac{1}{t} \sum^{t}_{i=1} \Big(\frac{3}{4}\Id +\frac{1}{4} \R_{L_{i}}\Big)
	\in \aff \{\Id, \R_{L_{1}}, \R_{L_{2}}, \ldots, \R_{L_{t}}\} \subseteq \aff ( \mathcal{S}_{L} ).
	\end{align*}
	
	The proofs for \cref{lem:AIdP:TS:LineFirmNone} and \cref{lem:AIdP:TS:fix} are similar to the corresponding parts of the proof in \cref{lem:AmP:TS}.
\end{proof}

\begin{proposition}  \label{prop:CwAverProLinRate}
	Assume that $\mathcal{H}= \mathbb{R}^{n}$ and $\{\Id, \R_{U_{1}}, \ldots, \R_{U_{t-1}},\R_{U_{t}}\} \subseteq \mathcal{S}$. Then for every $x \in \mathcal{H}$, $ (\CC{\mathcal{S}}^{k}x)_{k \in \mathbb{N}}$ converges to $ \Pro_{ \cap^{t}_{i=1} U_{i} }x$ with a linear rate $\norm{(\frac{1}{t} \sum^{t}_{i=1}\Pro_{L_{i}})\Pro_{(\cap^{t}_{i=1} L_{i})^{\perp}}}$.
\end{proposition}

\begin{proof}
	Combining \cref{lem:AmP:TS} and \cref{theorem:CCSLinearConvTSFirmNone}\cref{theorem:CCSLinearConvTSFirmNone:LineaConv}, we know that for every $y \in \mathcal{H}$, $ (\CC{\mathcal{S}_{L}}^{k}y)_{k \in \mathbb{N}}$ converges to $ \Pro_{ \cap^{t}_{i=1} L_{i} }y$ with a linear rate $\norm{(\frac{1}{t} \sum^{t}_{i=1}\Pro_{L_{i}})\Pro_{(\cap^{t}_{i=1} L_{i})^{\perp}}}$.
	
	Hence, the required result comes from \cref{prop:LineConvCCS:CCSL}.
\end{proof}

\begin{proposition} \label{prop:W1reflection:conver}
	Assume that $\mathcal{H}=\mathbb{R}^{n}$ and $\{\Id, \R_{U_{1}}, \R_{U_{2}}, \ldots, \R_{U_{t}}\} \subseteq \mathcal{S}$. Denote $T_{\mathcal{S}} := \frac{1}{t} \sum^{t}_{i=1}T_{i} x$ where $(\forall i \in \{1,2, \ldots,t\})$ $T_{i} := \frac{1}{2} (\Id +\Pro_{L_{i}})$. Let $x \in \mathbb{R}^{n}$. Then $(\CC{\mathcal{S}}^{k}x)_{k \in \mathbb{N}}$ linearly converges to $\Pro_{\cap^{t}_{i=1} U_{i} }x$ with a linear rate $\norm{T_{\mathcal{S}}\Pro_{(\cap^{t}_{i=1} L_{i})^{\perp}}}$.
\end{proposition}

\begin{proof}
	Using the similar method used in the proof of \cref{prop:CwAverProLinRate}, and using \cref{lem:AIdmP:TS} and \cref{theorem:CCSLinearConvTSFirmNone}\cref{theorem:CCSLinearConvTSFirmNone:LineaConv}, we obtain the required result.
\end{proof}

Clearly, we can take $\mathcal{S}=\{\Id, \R_{U_{1}}, \R_{U_{2}}, \ldots, \R_{U_{t}}\}$ in \cref{prop:CwAverProLinRate,prop:W1reflection:conver}.
In addition, \cref{prop:CwAverProLinRate,prop:W1reflection:conver} tell us that for different $T_{\mathcal{S}} \in \aff (\mathcal{S}_{L})$, we may obtain different linear convergence rates of $(\CC{\mathcal{S}}^{k}x)_{k \in \mathbb{N}}$.

\subsection{Accelerating the Douglas--Rachford method} \label{subsec:AcceDRM}
In this subsection, we consider the case when $t=2$.
\begin{lemma} \label{lem:PU1U2FixTEqut}
	Let $L_{1}, L_{2}$ be the closed linear subspaces defined in \cref{defn:L}.  Let $z \in L_{1}+L_{2}$. Denote $T:=T_{L_{2},L_{1}}$ defined in \cref{defn:DRO}.  Assume $L_{1} \cap L_{2} \subseteq \cap_{F \in \mathcal{S}_{L}} \Fix F$. Then
	\begin{align*}
	(\forall k \in \mathbb{N}) \quad \Pro_{L_{1}\cap L_{2}}(z) = \Pro_{L_{1}\cap L_{2}}(\CC{\mathcal{S}_{L}}^{k}z) = \Pro_{\Fix T}(\CC{\mathcal{S}_{L}}^{k}z).
	\end{align*}
\end{lemma}

\begin{proof}
	Using \cref{lemma:com:affspan:U1Um} \cref{lemma:com:affspan:U1Um:lem:SpnSx:SpanU1Um}, we get
	$
	(\CC{\mathcal{S}_{L}}^{k}z )_{k \in \mathbb{N}} \subseteq \spn(L_{1} \cup L_{2})= L_{1} +L_{2}.
	$
	Combining \cref{lemma:capUi}, \cref{prop:CCS:Equations}\cref{prop:CCS:Equations:CCSxk} (by taking $W= L_{1} \cap L_{2}$) with \cref{lem:PUVFixSpnUV}, we obtain that $(\forall k \in \mathbb{N})$ $ \Pro_{\Fix T}z = \Pro_{L_{1}\cap L_{2}}z= \Pro_{L_{1}\cap L_{2}}(\CC{\mathcal{S}_{L}}^{k}z) = \Pro_{\Fix T}(\CC{\mathcal{S}_{L}}^{k}z)$.
\end{proof}

\begin{corollary} \label{cor:PU1U2FixTEqut:CCSk}
	Let $L_{1}, L_{2}$ be the closed linear subspaces defined in \cref{defn:L}. Assume $L_{1} \cap L_{2} \subseteq \cap_{F \in \mathcal{S}_{L}} \Fix F$.   Let $x \in \mathcal{H}$. Let $K$ be a closed linear subspace of $\mathcal{H}$ such that
	\begin{align*}
	L_{1} \cap L_{2} \subseteq K \subseteq L_{1}+L_{2}.
	\end{align*}
	Denote $T:=T_{L_{2},L_{1}}$ defined in \cref{defn:DRO}. Then
	\begin{align*}
	(\forall k \in \mathbb{N}) \quad \Pro_{L_{1}\cap L_{2}}x = \Pro_{\Fix T}\Pro_{K}x =\Pro_{L_{1}\cap L_{2}}\Pro_{K}x =\Pro_{L_{1}\cap L_{2}}(\CC{\mathcal{S}_{L}}^{k}\Pro_{K}x ) = \Pro_{\Fix T}(\CC{\mathcal{S}_{L}}^{k}\Pro_{K}x ).
	\end{align*}
\end{corollary}

\begin{proof}
	Because $\Pro_{K}x \in K \subseteq L_{1}+L_{2}$. Then \cref{lem:PFixTPWPUCapV} implies that
	\begin{align} \label{eq:cor:PU1U2FixTEqut:CCSk:P}
	\Pro_{L_{1}\cap L_{2}}x  =\Pro_{L_{1}\cap L_{2}}\Pro_{K}x = \Pro_{\Fix T}\Pro_{K}x.
	\end{align}
	Applying  \cref{lem:PU1U2FixTEqut} with $z = \Pro_{K}x$, we get the desired result.
\end{proof}

Using \cref{cor:PU1U2FixTEqut:CCSk}, \cref{prop:CCS:Equations}
\cref{prop:CCS:Equations:CCSTSP}, \cref{fac:cFLess1}, \cref{fac:DRCFn} and an
idea similar to the proof of \cite[Theorem~1]{BCS2017}, we obtain the
following more general result, which is motivated by \cite[Theorem~1]{BCS2017}. 
In fact, \cite[Theorem~1]{BCS2017} reduces to
\cref{prop:CW:CCS:linear}\cref{prop:CW:CCS:linearconv} when
$\mathcal{H}=\mathbb{R}^{n}$ and $\mathcal{S} = \{\Id, \R_{U_{1}},
\R_{U_{2}}\R_{U_{1}} \}$.

% \textcolor{purple}{Define $c_{F}$ as it is in the following
% \cref{prop:CW:CCS:General:linearconv}. In the following result, we provide
% circumcentered reflection methods with convergence rates no worse than
% $c^{d}_{F}$ where $d \in \mathbb{N} \smallsetminus \{0\}$.}

\begin{proposition}  \label{prop:CW:CCS:General:linearconv}
	Let $L_{1}, L_{2}$ be the closed linear subspaces defined in \cref{defn:L}.    Assume $L_{1} \cap L_{2} \subseteq \cap_{F \in \mathcal{S}_{L}} \Fix F$.   Let $K$ be a closed affine subspace of $\mathcal{H}$ such that for $K_{L}=\pa K$,
	\begin{align*}
	L_{1}\cap L_{2} \subseteq K_{L} \subseteq L_{1}+L_{2}.
	\end{align*}
	Denote $T:=T_{U_{2},U_{1}}$ and $T_{L}:=T_{L_{2},L_{1}}$ defined in \cref{defn:DRO}. Denote the $c(L_{1}, L_{2})$ defined in \cref{defn:FredrichAngleClassical} by $c_{F}$.
	Assume there exists $d \in \mathbb{N} \smallsetminus \{0\}$ such that $T^{d} \in \aff \mathcal{S}$. 	Let $x \in \mathcal{H}$.  Then
	\begin{align*}
	(\forall k \in \mathbb{N}) \quad \norm{\CC{\mathcal{S}}^{k}\Pro_{K}x - \Pro_{U_{1}\cap U_{2}}x} \leq (c_{F})^{d k} \norm{\Pro_{K}x  - \Pro_{U_{1}\cap U_{2}}x}.
	\end{align*}
\end{proposition}

\begin{proof}
	By definition, $T^{d} \in \aff \mathcal{S}$ means that $T_{L}^{d} \in \aff \mathcal{S}_{L}$.
	Using \cref{cor:PU1U2FixTEqut:CCSk}, we get
	\begin{align} \label{eq:prop:CW:CCS:linearcon:equalites}
	(\forall n \in \mathbb{N}) \quad \Pro_{L_{1}\cap L_{2}}x = \Pro_{\Fix T_{L}} \Pro_{K_{L}}x =\Pro_{L_{1}\cap L_{2}}\Pro_{K_{L}}x=\Pro_{L_{1}\cap L_{2}}(\CC{\mathcal{S}_{L}}^{n}\Pro_{K_{L}}x) = \Pro_{\Fix T_{L}}(\CC{\mathcal{S}_{L}}^{n}\Pro_{K_{L}}x).
	\end{align}
	Since  $T^{d}_{L} \in \aff \mathcal{S}_{L}$, \cref{prop:CCS:Equations}\cref{prop:CCS:Equations:CCSTSP} implies that
	\begin{align}  \label{eq:prop:CW:CCS:linearcon:Inequality}
	(\forall y \in \mathcal{H}) \quad \norm{\CC{\mathcal{S}_{L}}(y)- \Pro_{L_{1} \cap L_{2}} y} \leq \norm{T^{d}_{L}(y)- \Pro_{L_{1} \cap L_{2}} y}.
	\end{align}
	Using \cref{fac:DRCFn}, we get
	\begin{align} \label{eq:prop:CW:CCS:linearcon:TCCS}
	(\forall y \in \mathcal{H}) \quad  \norm{T^{d}_{L}y - \Pro_{\Fix T_{L}}y} \leq c^{d}_{F} \norm{y - \Pro_{\Fix T_{L}}y}.
	\end{align}
%	Hence,
%	\begin{align*}
%	\norm{\CC{\mathcal{S}_{L}}\Pro_{K_{L}}x - \Pro_{L_{1}\cap L_{2}}x} & \stackrel{\cref{eq:prop:CW:CCS:linearcon:equalites}}{=} \norm{\CC{\mathcal{S}_{L}}\Pro_{K_{L}}x - \Pro_{L_{1}\cap L_{2}}\Pro_{K_{L}}x} \\
%	& \stackrel{\cref{eq:prop:CW:CCS:linearcon:Inequality}}{\leq} \norm{T^{d}_{L}\Pro_{K_{L}}x-  \Pro_{L_{1}\cap L_{2}}\Pro_{K_{L}}x}\\
%	& \stackrel{\cref{eq:prop:CW:CCS:linearcon:equalites}}{=} \norm{T^{d}_{L}\Pro_{K_{L}}x - \Pro_{\Fix T_{L}}\Pro_{K_{L}}x} \\
%	& \stackrel{\cref{eq:prop:CW:CCS:linearcon:TCCS}}{\leq} c^{d}_{F} \norm{\Pro_{K_{L}}x - \Pro_{\Fix T_{L}}\Pro_{K_{L}}x}\\
%	& \stackrel{\cref{eq:prop:CW:CCS:linearcon:equalites}}{=} c^{d}_{F} \norm{\Pro_{K_{L}}x - \Pro_{L_{1}\cap L_{2}}x}.
%	\end{align*}
If $k=0$, then the result is trivial. 
Thus, we assume that for some $k \geq 0$,  we have
	\begin{align} \label{eq:prop:CW:CCS:linearconv:K}
	\norm{\CC{\mathcal{S}_{L}}^{k}\Pro_{K_{L}}x - \Pro_{L_{1}\cap L_{2}}x}   \leq (c_{F})^{d k} \norm{\Pro_{K_{L}}x - \Pro_{L_{1}\cap L_{2}}x}.
	\end{align}
	Then
	\begin{align*}
	\norm{\CC{\mathcal{S}_{L}}^{k+1}\Pro_{K_{L}}x - \Pro_{L_{1}\cap L_{2}}x} & \stackrel{\cref{eq:prop:CW:CCS:linearcon:equalites}}{=}\norm{\CC{\mathcal{S}_{L}}(\CC{\mathcal{S}_{L}}^{k}\Pro_{K_{L}}x) - \Pro_{L_{1}\cap L_{2}}(\CC{\mathcal{S}_{L}}^{k}\Pro_{K_{L}}x)}\\
	& \stackrel{\cref{eq:prop:CW:CCS:linearcon:Inequality}}{\leq} \norm{T_{L}^{d} (\CC{\mathcal{S}_{L}}^{k}\Pro_{K_{L}}x) - \Pro_{L_{1}\cap L_{2}}(\CC{\mathcal{S}_{L}}^{k}\Pro_{K_{L}}x)} \\
	& \stackrel{\cref{eq:prop:CW:CCS:linearcon:equalites}}{=} \norm{T^{d}_{L} (\CC{\mathcal{S}_{L}}^{k}\Pro_{K_{L}}x) - \Pro_{\Fix T_{L}}(\CC{\mathcal{S}_{L}}^{k}\Pro_{K_{L}}x)} \\
	& \stackrel{\cref{eq:prop:CW:CCS:linearcon:TCCS}}{\leq} c^{d}_{F} \norm{ \CC{\mathcal{S}_{L}}^{k}\Pro_{K_{L}}x - \Pro_{\Fix T_{L}} (\CC{\mathcal{S}_{L}}^{k}\Pro_{K_{L}}x)} \\
	& \stackrel{\cref{eq:prop:CW:CCS:linearcon:equalites}}{=} c^{d}_{F} \norm{\CC{\mathcal{S}_{L}}^{k}\Pro_{K_{L}}x - \Pro_{L_{1}\cap L_{2}}x}\\
	& \stackrel{\cref{eq:prop:CW:CCS:linearconv:K}}{\leq} c^{d}_{F} (c_{F})^{d k} \norm{\Pro_{K_{L}}x - \Pro_{L_{1}\cap L_{_2}}x}\\
	&~ = (c_{F})^{d(k+1)} \norm{\Pro_{K_{L}}x - \Pro_{L_{1}\cap L_{2}}x}.
	\end{align*}
	Hence, we have inductively proved
	\begin{align} \label{eq:prop:CW:CCS:Results:L}
	(\forall k \in \mathbb{N}) \quad (\forall y\in \mathcal{H}) \quad 	\norm{\CC{\mathcal{S}_{L}}^{k}\Pro_{K_{L}}y - \Pro_{L_{1}\cap L_{2}}y}   \leq (c_{F})^{d k} \norm{\Pro_{K_{L}}y- \Pro_{L_{1}\cap L_{2}}y}.
	\end{align}
	Let $u \in U_{1} \cap U_{2}$. By \cref{lemma:RU:RL}\cref{lemma:RU:RL:CCSk}, we know that  $(\forall k \in \mathbb{N})$ $(\forall y \in \mathcal{H})$ $\CC{\mathcal{S}}^{k}y = u+ \CC{\mathcal{S}_{L}}^{k}(y-u)$ and  by \cref{fac:SetChangeProje}, we have $\Pro_{ \cap^{2}_{i=1} U_{i} }y = \Pro_{u +\cap^{2}_{i=1} L_{i} }y  = u+ \Pro_{ \cap^{2}_{i=1} L_{i} }(y-u)$. Hence we obtain that for every $k \in \mathbb{N}$ and for every $x \in \mathcal{H}$,
	\begin{align*}
	\norm{\CC{\mathcal{S}}^{k}(\Pro_{K}x) - \Pro_{U_{1}\cap U_{2}}x}  & = \norm{u + \CC{\mathcal{S}_{L}}^{k}(\Pro_{K}(x) -u) - u -\Pro_{L_{1}\cap L_{2}}(x-u)}  \\
	& =  \norm{\CC{\mathcal{S}_{L}}^{k}(\Pro_{K_{L}}(x-u) )  -\Pro_{L_{1}\cap L_{2}}(x-u)}  \\
	& \stackrel{\cref{eq:prop:CW:CCS:Results:L}}{\leq}  (c_{F})^{d k} \norm{\Pro_{K_{L}}(x-u)- \Pro_{L_{1}\cap L_{2}}(x-u)}\\
	&= (c_{F})^{d k} \norm{u+\Pro_{K_{L}}(x-u)-\left( u+ \Pro_{L_{1}\cap L_{2}}(x-u) \right)}\\
	&=  (c_{F})^{d k} \norm{\Pro_{K}x- \Pro_{U_{1}\cap U_{2}}x }.
	\end{align*}
Therefore, the proof is complete.
\end{proof}

Let us now provide an application of \cref{prop:CW:CCS:General:linearconv}. 

\begin{proposition} \label{prop:CW:CCS:linear}
	Assume that $U_{1}, U_{2}$ are two closed affine subspaces with $\pa U_{1}+\pa U_{2}$ being closed. Let $x \in \mathcal{H}$. Let $c_{F}$ be the cosine of the Friedrichs angle between $\pa U_{1}$ and $\pa U_{2}$. Then the following  hold:
	\begin{enumerate}
		\item \label{prop:CW:CCS:linearconv}
  Assume that $\{\Id, \R_{U_{2}}\R_{U_{1}}\} \subseteq \mathcal{S}$. Then each of the
  three sequences $(\CC{\mathcal{S}}^{k}(\Pro_{U_{1}}x))_{k \in \mathbb{N}}$,
  $(\CC{\mathcal{S}}^{k}(\Pro_{U_{2}}x))_{k \in \mathbb{N}}$ and
  $(\CC{\mathcal{S}}^{k}(\Pro_{U_{1} + U_{2}}x))_{k \in \mathbb{N}}$ converges
  linearly to $\Pro_{U_{1} \cap U_{2}}x$. Moreover, their rates of
  convergence are no larger than $c_{F} \in [0,1[\,$.
  \item \label{prop:TT:lineconv} Assume that $\{\Id, \R_{U_{2}}\R_{U_{1}},
  \R_{U_{2}}\R_{U_{1}}\R_{U_{2}}\R_{U_{1}}\} \subseteq \mathcal{S}$. Then the
  sequences $(\CC{\mathcal{S}}^{k}(\Pro_{U_{1}}x))_{k \in \mathbb{N}}$,
  $(\CC{\mathcal{S}}^{k}(\Pro_{U_{2}}x))_{k \in \mathbb{N}}$ and
  $(\CC{\mathcal{S}}^{k}(\Pro_{U_{1} + U_{2}}x))_{k \in \mathbb{N}}$ converge
  linearly to $\Pro_{U_{1} \cap U_{2}}x$. Moreover, their rates of
  convergence are no larger than $c^{2}_{F}$.
	\end{enumerate}
\end{proposition}

\begin{proof}
	Clearly, under the conditions of each statement, $\pa U_{1}  \cap \pa U_{2} \subseteq \cap_{F \in \mathcal{S}_{L}} \Fix F$. In addition, we are able to substitute $K_{L}$ in \cref{prop:CW:CCS:General:linearconv} by any one of $\pa U_{1}$, $\pa U_{2}$ or $ \pa U_{1}+\pa U_{2}$.
	
	\cref{prop:CW:CCS:linearconv}:  Since $\{\Id, \R_{U_{2}}\R_{U_{1}}\} \subseteq \mathcal{S}$,
	\begin{align*}
	T_{U_{2},U_{1}} :=\frac{\Id +\R_{U_{2}}\R_{U_{1}}}{2} \in \aff \{\Id, \R_{U_{2}}\R_{U_{1}}\} \subseteq \aff \mathcal{S}.
	\end{align*}
	Substitute $d =1$ in \cref{prop:CW:CCS:General:linearconv} to obtain
	\begin{align*}
	(\forall k \in \mathbb{N}) \quad \norm{\CC{\mathcal{S}}^{k}\Pro_{K_{L}}x- \Pro_{U_{1}\cap U_{2}}x} \leq c^{k}_{F} \norm{\Pro_{K_{L}}x - \Pro_{U_{1}\cap U_{2}}x}.
	\end{align*}
	Because $\pa U_{1}+\pa U_{2}$ is closed, by \cref{fac:cFLess1}, we know that $c_{F} \in [0,1[\,$.
	
	\cref{prop:TT:lineconv}: Since $\{\Id, \R_{U_{2}}\R_{U_{1}}, \R_{U_{2}}\R_{U_{1}}\R_{U_{2}}\R_{U_{1}}\} \subseteq \mathcal{S}$, by \cite[Proposition~4.13(i)]{BOyW2018Proper}, we know that
	\begin{align*}
	T^{2}_{U_{2},U_{1}} = \left( \frac{\Id +\R_{U_{2}}\R_{U_{1}}}{2}\right)^{2} \in \aff \mathcal{S}.
	\end{align*}
	The remainder of the proof is similar with the proof in \cref{prop:CW:CCS:linearconv} above. The only difference is that this time we substitute $d =2$ but not $d =1$.
\end{proof}

The following example shows that the special address for the initial points in \cref{prop:CW:CCS:linear} is necessary.
\begin{example} \label{exam:DROLimNotInIntersec}
	Assume that $U_{1},U_{2}$ are two closed linear subspaces in $\mathcal{H}$ such that $U_{1}+U_{2}$ is closed. Assume $\mathcal{S} = \{\Id, \R_{U_{2}}\R_{U_{1}}\}$. Let $x \in \Big(\mathcal{H} \smallsetminus (U_{1}+U_{2})\Big)$.  Clearly, $U_{1} \cap U_{2} \subseteq  \cap_{T \in  \mathcal{S}} \Fix T$. But
	\begin{align*}
	\lim_{k \rightarrow \infty}  \CC{\mathcal{S}}^{k}x = \Pro_{\Fix \CC{\mathcal{S}}}x \not \in U_{1} \cap U_{2}.
	\end{align*}
\end{example}

\begin{proof}
	By definition of $\mathcal{S}$ and by \cref{fact:form:m2:Oper}, $\CC{\mathcal{S}} = T_{U_{2},U_{1}}$,
	where the $T_{U_{2},U_{1}}$ is the Douglas--Rachford operator defined in \cref{defn:DRO}.
	By assumptions, \cref{fac:cFLess1} and \cref{fac:DRCFn} imply that $(\CC{\mathcal{S}}^{k}x)_{k \in \mathbb{N}}$ converges linearly to $\Pro_{\Fix \CC{\mathcal{S}}}x$. So
	\begin{align} \label{eq:CCSkPFixCCSx}
	\lim_{k \rightarrow \infty} \CC{\mathcal{S}}^{k}x = \Pro_{\Fix \CC{\mathcal{S}}}x.
	\end{align}
	Since $x \not \in U_{1} + U_{2} =\overline{U_{1} + U_{2}}$, \cref{lem:PUVFixSpnUV} yields that
	\begin{align} \label{eq:prop:DROLimNotInIntersec}
	\Pro_{\Fix \CC{\mathcal{S}}}x \neq \Pro_{U_{1} \cap U_{2}}x.
	\end{align}
	
	Assume to the contrary $\Pro_{\Fix \CC{\mathcal{S}}}x \in U_{1} \cap U_{2}$. By \cref{thm:CWP}\cref{prop:CWP:norm:conv}  and \cref{eq:CCSkPFixCCSx}, we get $\Pro_{\Fix \CC{\mathcal{S}}}x = \Pro_{U_{1} \cap U_{2}}x$, which contradicts \cref{eq:prop:DROLimNotInIntersec}.
	
	Therefore, $\lim_{k \rightarrow \infty} \CC{\mathcal{S}}^{k}x = \Pro_{\Fix \CC{\mathcal{S}}}x  \not \in U_{1} \cap U_{2}$.
\end{proof}

\subsection{Best approximation for the intersection of finitely many affine subspaces} \label{subsec:BestApproFiniteAffine}
In this subsection, our main goal is to apply
\cref{prop:CW:CCS:linear}\cref{prop:CW:CCS:linearconv} to find the best
approximation onto the intersection of finitely many affine subspaces. 
% In
% order to attain our goal, we also obtain the interesting byproduct
% \cref{prop:CcapD:Character}. Although we have some other ways to find the
% best approximation $\Pro_{ \cap^{m}_{i=1}C_{i}}x$, our proof of
% \cref{prop:CcapD:Character} is short, easy and self-contained.
Unless stated otherwise, let $\I := \{1,\ldots,N\}$ with $N \geq 1$
and let $\mathcal{H}^{N}$ be the real Hilbert space obtained by endowing the Cartesian product $\bigtimes_{i \in \I } \mathcal{H}$ with the usual vector space structure and with the inner product $(\mathbf{x},\mathbf{y}) \mapsto  \sum^{N}_{i =1}  \innp{x_{i},y_{i}}$, where $\mathbf{x} = (x_{i})_{i \in \I}$ and $\mathbf{y}=(y_{i})_{i \in \I}$ (for details, see \cite[Proposition~29.16]{BC2017}).

Let $(\forall i \in \I)$ $C_{i}$ be a nonempty closed convex subset of $\mathcal{H}$. Define two subsets of $\mathcal{H}^{N}$:
\begin{align*}
\mathbf{C} := \bigtimes_{i \in \I } C_{i} \quad \text{and} \quad \mathbf{D} := \left\{(x)_{i \in \I} \in \mathcal{H}^{N} ~|~ x \in \mathcal{H} \right\},
\end{align*}
which are both closed and convex (in fact, $\mathbf{D}$ is a linear subspace).

\begin{fact}  {\rm \cite[Propositions~29.3 and 29.16]{BC2017}} \label{fac:PDC1CN}
	Let $\mathbf{x} :=(x_{i})_{i \in \I}$.  Then
	\begin{enumerate}
		\item $\Pro_{\mathbf{C}}\mathbf{x}  =  \left( \Pro_{C_{i}} x_{i} \right)_{i \in \I }  $.
		\item $\Pro_{\mathbf{D}}\mathbf{x} = \left(   \frac{1}{N}  \sum_{i \in \I }  x_{i} \right)_{i \in \I } $.
	\end{enumerate}
\end{fact}
The following two results are clear from the definition of the 
sets $\mathbf{C}$ and $\mathbf{D}$.
\begin{lemma} \label{lemma:CcapD:charac}
	Let $x \in \mathcal{H}$. Then $
	(x,\ldots,x) \in  \mathbf{C} \cap  \mathbf{D}  \Leftrightarrow x \in  \cap_{i \in \I}C_{i}$.
\end{lemma}
\begin{proposition} \label{prop:CcapD:Character}
	Let $x \in \mathcal{H}$. Then
$
	\Pro_{ \mathbf{C} \cap  \mathbf{D} } (x, \ldots,x)  = (\Pro_{ \cap^{N}_{i=1}C_{i}}x, \ldots, \Pro_{ \cap^{N}_{i=1}C_{i}}x).
$
\end{proposition}	
% \begin{proof}
% 	By definition of projector and by \cref{lemma:CcapD:charac}, we know that
% 	\begin{align*}
% 	(y,\ldots,y) = \Pro_{\mathbf{C} \cap  \mathbf{D} } (x, \ldots,x)
% 	& \Leftrightarrow   (y,\ldots,y) = \argmin \left\{ \Norm{(y,\ldots,y) - (x,\ldots,x)}~|~ (y,\ldots,y) \in \mathbf{C} \cap  \mathbf{D}   \right\}\\
% 	& \Leftrightarrow  y=\argmin \left\{ \Norm{(y,\ldots,y) - (x,\ldots,x)}~|~ y  \in\cap^{N}_{i=1}C_{i} \right\}\\
% 	& \Leftrightarrow  y =  \argmin \left\{ \Norm{(y-x,\ldots,y-x) }~|~ y  \in \cap^{N}_{i=1}C_{i} \right\}.
% 	\end{align*}
% 	On the other hand, by the definition of inner product endowed in the Hilbert space $\mathcal{H}^{N}$,
% 	\begin{align*}
% 	\argmin \left\{ \Norm{(y-x,\ldots,y-x) }~|~ y  \in \cap^{N}_{i=1}C_{i} \right\}
% 	& =  \argmin \left\{ \left( \sum^{N}_{i=1} \norm{y-x}^{2} \right)^{\frac{1}{2}} ~\Big|~ y  \in\cap^{N}_{i=1}C_{i} \right\} \\
% 	& =  \argmin \left\{   \norm{y-x}~|~ y  \in \cap^{N}_{i=1}C_{i} \right\} \\
% 	& = \Pro_{ \cap^{N}_{i=1}C_{i}}x.
% 	\end{align*}
% 	Therefore, the proof is complete.
% \end{proof}	

\begin{fact} \label{fact:vonNeum} {\rm \cite[Corollary~5.30]{BC2017}}
	Let $t$ be a strictly positive integer, set $\J :=\{1,\ldots,t\}$, let $(U_{j})_{j \in \J}$ be a family of closed affine subspaces of $\mathcal{H}$ such that $\cap^{t}_{j=1} U_{j} \neq \varnothing$. Let $x_{0} \in \mathcal{H}$. Set $(\forall n \in \mathbb{N})$ $x_{n+1} := \Pro_{U_{t}}\cdots \Pro_{U_{1}}x_{n}$. Then $x_{n} \rightarrow \Pro_{\cap^{t}_{j=1} U_{j} }x_{0}$.
\end{fact}
Using \cref{fact:vonNeum} and \cref{prop:CcapD:Character}, we obtain the following interesting by-product, which can be treated as a new method to solve the best approximation problem associated with  $ \cap^{N}_{i=1}C_{i}$.
\begin{proposition} \label{prop:Proj:byproduct}
	Assume $(\forall i \in \I)$ $C_{i}$ is a closed affine subspace of $\mathcal{H}$ with $ \cap^{N}_{i=1}C_{i} \neq \varnothing$. 	Let $x \in \mathcal{H}$. Then the following  hold:
	\begin{enumerate}
		\item  \label{prop:Proj:byproduct:vonNeu}$\Pro_{ \mathbf{C} \cap  \mathbf{D} } (x, \ldots,x) = \lim_{k \to \infty} (\Pro_{ \mathbf{D} } \Pro_{\mathbf{C} })^{k} (x, \ldots,x) $.

		\item \label{prop:Proj:byproduct:converg}	

		Denote by $Q:=\frac{1}{N}(\Pro_{C_{1}} + \ldots + \Pro_{C_{N}} )$, then $$Q^{k}x \to \Pro_{ \cap^{m}_{i=1}C_{i}}x.$$
	\end{enumerate}	
\end{proposition}	

\begin{proof}
	Since $(\forall i \in \I)$ $C_{i}$ is closed affine subspace of $\mathcal{H}$ with $ \cap^{N}_{i=1}C_{i} \neq \varnothing$, thus $\mathbf{C}$ is closed affine subspace of $\mathcal{H}^{N}$ and $\mathbf{C} \cap \mathbf{D} \neq \varnothing$. By definition of $\mathbf{D}$, it is a linear subspace of $\mathcal{H}^{N}$.
	
	\cref{prop:Proj:byproduct:vonNeu}: The result is from \cref{fact:vonNeum} by taking $t=2$ and considering the two closed affine subspaces $\mathbf{C}$ and $\mathbf{D}$ in $\mathcal{H}^{N}$.

	\cref{prop:Proj:byproduct:converg}: Combine \cref{fac:PDC1CN}, \cref{prop:CcapD:Character} with the above \cref{prop:Proj:byproduct:vonNeu}  to obtain the desired results.
\end{proof}

\begin{fact} {\rm \cite[Lemma~5.18]{BB1996}} \label{fac:BoldCplusDClosed}
	Assume each set $C_{i}$ is a closed linear subspace. Then $C^{\perp}_{1} + \cdots + C^{\perp}_{N}$ is closed if and only if $\mathbf{D} + \mathbf{C}$ is closed.
\end{fact}

The next proposition shows that we can use the circumcenter method induced by reflectors to solve the best approximation problem associated with finitely many closed affine subspaces.
Recall that for each affine subspace $U$, we denote the linear subspace paralleling $U$ as $\pa U$, i.e., $\pa U := U-U$.
\begin{proposition} \label{prop:feasiProbCCSLineConv}
Assume $U_{1}, \ldots ,U_{t}$ are closed affine subspaces in
$\mathcal{H}$, with $\cap^{t}_{i=1} U_{i} \neq \varnothing$ and $ (\pa
U_{1})^{\perp}+\cdots+(\pa U_{t})^{\perp}$ being closed. Set $\J:=\{1,
\ldots, t\}$, $\mathbf{C}:= \bigtimes_{j
\in \J} U_{i}$, and $\mathbf{D}:= \{(x, \ldots, x) \in \mathcal{H}^{t}~ |~ x
\in \mathcal{H}\}$. Assume $\{\Id, \R_{\mathbf{C}}\R_{\mathbf{D}}\}
\subseteq \mathcal{S}$ or $\{\Id, \R_{\mathbf{D}}\R_{\mathbf{C}}\}
\subseteq \mathcal{S} $. Let $x \in \mathcal{H}$ and set $\mathbf{x} :=
(x, \ldots, x) \in \mathcal{H}^{t} \cap \mathbf{D}$. Then $(
\CC{\mathcal{S}}^{k} \mathbf{x} )_{k \in \mathbb{N}}$ converges to $
\Pro_{\mathbf{C} \cap \mathbf{D}} \mathbf{x} =
(\Pro_{\cap^{t}_{i=1}U_{i}}x, \ldots, \Pro_{\cap^{t}_{i=1}U_{i}}x)$
linearly.
\end{proposition}

\begin{proof}
	Denote $\mathbf{C_{L}} := \bigtimes_{j \in \J} \pa U_{j}$. Clearly, $\mathbf{C_{L}} = \pa \mathbf{C} $.  Now $\pa  U_{1}, \ldots , \pa U_{t}$ are closed linear subspaces implies that $\mathbf{C_{L}}$ is closed linear subspace.
	It is clear that $\mathbf{D} =\pa \mathbf{D}$ is a closed linear subspace. Because $  (\pa U_{1})^{\perp}+\cdots+(\pa U_{t})^{\perp}$ is closed, by \cref{fac:BoldCplusDClosed}, we get $\mathbf{C_{L}}+\mathbf{D}$ is closed. Then using  \cref{prop:CW:CCS:linear}\cref{prop:CW:CCS:linearconv}, we know there exists a constant $c_{F} \in [0,1[$ such that
	\begin{align*}
	\label{eq:prop:feasiProbCCSLineConv:L}
	(\forall k \in \mathbb{N})\quad (\forall \mathbf{y} \in \mathbf{D}) \quad \norm{\CC{\mathcal{S}_{L}}^{k}\mathbf{y} -\Pro_{\mathbf{C_{L}}\cap \mathbf{D}} \mathbf{y}}=\norm{\CC{\mathcal{S}_{L}}^{k}\Pro_{\mathbf{D}}\mathbf{y} -\Pro_{\mathbf{C_{L}}\cap \mathbf{D}} \mathbf{y} } \leq c^{k}_{F}   \norm{\Pro_{\mathbf{D}}\mathbf{y} -\Pro_{\mathbf{C_{L}}\cap \mathbf{D}} \mathbf{y} },
	\end{align*}
	which imply that $( \CC{\mathcal{S}_{L}}^{k} (\mathbf{x}-\mathbf{u}) )_{k \in \mathbb{N}}$ linearly converges to $\Pro_{\mathbf{C_{L}}\cap \mathbf{D}} (\mathbf{x}-\mathbf{u})$ for any $u \in \cap^{t}_{i=1} U_{i}$ and $\mathbf{u}= (u, \ldots, u)$.
	Hence, by \cref{prop:LineConvCCS:CCSL}, we conclude that $( \CC{\mathcal{S}}^{k} \mathbf{x} )_{k \in \mathbb{N}}$ linearly converges to $\Pro_{\mathbf{C}\cap \mathbf{D}} \mathbf{x}$.
	Since by \cref{prop:CcapD:Character}, $ \Pro_{\mathbf{C} \cap \mathbf{D}} \mathbf{x} =(\Pro_{\cap^{t}_{i=1}U_{i}}x, \ldots, \Pro_{\cap^{t}_{i=1}U_{i}}x)$, thus
	$( \CC{\mathcal{S}}^{k} \mathbf{x} )_{k \in \mathbb{N}}$ linearly converges to $(\Pro_{\cap^{t}_{i=1}U_{i}}x, \ldots, \Pro_{\cap^{t}_{i=1}U_{i}}x)$.
\end{proof}

\section{Numerical experiments} \label{sec:NumericalExperiment}
In order to explore the convergence rate of the circumcenter methods, in this
section we use the performance profile introduced by Dolan and Mor\'e
\cite{DM2002} to compare circumcenter methods induced by reflectors developed
in \cref{sec:CircumMethodReflectors} with the Douglas--Rachford method (DRM) and
the method of alternating projections (MAP) for solving the best approximation
problems associated with linear subspaces. (Recall that by
\cref{prop:LineConvCCS:CCSL}, for any convergence results on circumcenter
methods induced by reflectors associated with linear subspaces, we will
obtain the corresponding equivalent convergence result on that associated
with affine subspaces.)

In the whole section, given a pair of closed and linear subspaces, $U_{1}, U_{2}$,  and  a initial point $x_{0}$, the problem we are going to solve is to
\begin{empheq}[box=\mybluebox]{equation*}
\text{find the best approximation } \overline{x} :=\Pro_{U_{1} \cap U_{2}}x_{0}.
\end{empheq}
Denote the cosine of the Friedrichs angle between $U_{1}$ and $U_{2}$ by  $c_{F}$.
It is well known that the sharp rate of the linear convergence of  DRM and MAP for finding $\Pro_{U_{1}\cap U_{2}}x_{0}$  are $c_{F}$ and $c^{2}_{F}$  respectively (see, \cite[Theorem~4.3]{BCNPW2014} and \cite[Theorem~9.8]{D2012} for details).   Hence, if $c_{F}$ is \enquote{small}, then we expect  DRM and MAP converge to $\Pro_{U_{1} \cap U_{2}}x_{0}$ \enquote{fast}, but if $c_{F} \approx 1$, the two classical solvers should converge to $\Pro_{U_{1} \cap U_{2}}x_{0}$ \enquote{slowly}.
The  $c_{F}$ associated with the problems in each experiment below is randomly chosen from some certain range.

\subsection{Numerical preliminaries} \label{eq:NumericalProlimi}
Dolan and Mor\'e define a benchmark in terms of a set $\mathbf{P}$ of benchmark problems, a set $\mathbf{S}$ of optimization solvers, and a convergence measure matrix $\mathbf{T}$. Once these components of a benchmark are defined, performance profile can be used to compare the performance of the solvers.

We assume $\mathcal{H}= \mathbb{R}^{1000}$. In every one of our experiment, we randomly generate $10$ pairs of linear subspaces, $U_{1}, U_{2}$ with Friedrichs angles in certain range. We create pairs of linear subspaces with particular Friedrichs angle by \cite{GK1987}. For each pair of subspaces, we choose randomly $10$ initial points, $x_{0}$. This results in a total of $100$ problems, that constitute our set $\mathbf{P}$ of benchmark problems.
Set
\begin{align*}
& \mathcal{S}_{1} := \{ \Id, R_{U_{1}}, R_{U_{2}}\},\quad \mathcal{S}_{2} := \{ \Id, R_{U_{1}}, R_{U_{2}}R_{U_{1}}\},\\
& \mathcal{S}_{3} := \{ \Id, R_{U_{1}}, R_{U_{2}}, R_{U_{2}}R_{U_{1}}\}, \quad
\mathcal{S}_{4} := \{ \Id, R_{U_{1}}, R_{U_{2}}, R_{U_{2}}R_{U_{1}}, R_{U_{1}}R_{U_{2}}, R_{U_{1}}R_{U_{2}}R_{U_{1}}\}.
\end{align*}
Notice that
\begin{empheq}{equation*}
\text{$\CC{S_{2}}$ is the C--DRM operator $C_{T}$ in \cite{BCS2017}}
\end{empheq}
and hence, it is also the CRM operator $C$ in \cite{BCS2018} when $m=2$.

Our test algorithms and sequences to monitor are as follows.
\noindent
\begin{table}[H]
	\centering
	\begin{tabu} to 0.9\textwidth { m{9cm}  m{5.5cm} }
		\hline
		Algorithm & Sequence to monitor\\
		\hline
		Douglas--Rachford method  & $P_{U_{1}} (\frac{1}{2}(\Id + R_{U_{2}}R_{U_{1}}))^{k}(x_{0})$ \\
		Method of alternating projections & $(P_{U_{2}}P_{U_{1}})^{k}(x_{0})$ \\
		Circumcenter method induced by $\mathcal{S}_{1}$ & $(\CC{\mathcal{S}_{1}})^{k}(x_{0})$ \\
		Circumcenter method induced by $\mathcal{S}_{2}$ & $(\CC{\mathcal{S}_{2}})^{k}(x_{0})$ \\
		Circumcenter method induced by $\mathcal{S}_{3}$ & $(\CC{\mathcal{S}_{3}})^{k}(x_{0})$ \\
		Circumcenter method induced by $\mathcal{S}_{4}$ & $(\CC{\mathcal{S}_{4}})^{k}(x_{0})$ \\
		\hline
	\end{tabu}
	\caption{Forming the set of solvers $\mathbf{S}$} \label{tab:Test algorithms}
\end{table}
Hence,  our set $\mathbf{S}$ of optimization solvers is subset of the set consists of the six algorithms above.

For every $i \in \{1,2,3,4\}$, 
we calculate the operator $\CC{\mathcal{S}_{i}}$ by applying
\cref{prop:CW:Welldefined:Formula}, and for notational simplicity,
\begin{empheq}[box=\mybluebox]{equation*}
\text{we denote the circumcenter method induced by}~\mathcal{S}_{i}~\text{by}~\CC{\mathcal{S}_{i}}.
\end{empheq}
We use $10^{-6}$ as the tolerance employed in our stopping criteria and we
terminate the algorithm when the number of iterations reaches $10^{6}$ (in
which case the problem is declared unsolved). For each problem $p$ with the
exact solution being $\overline{x}=\Pro_{U_{1} \cap U_{2}}x_{0}$, and for
each solver $s$, the performance measure considered in the
whole section is either
\begin{empheq}[box=\mybluebox]{equation} \label{eq:tps:NumberIterate}
t_{p,s} =\text{the smallest}~k~\text{such that}~ \norm{a^{(k)}_{p,s}-\overline{x}} \leq 10^{-6}~\text{with}~k \leq 10^{6},
\end{empheq}
or 
\begin{empheq}[box=\mybluebox]{equation} \label{eq:tps:Runtime}
t_{p,s} =\text{the run time used until the smallest }~k~\text{such that}~ \norm{a^{(k)}_{p,s}-\overline{x}} \leq 10^{-6}~\text{with}~k \leq 10^{6},
\end{empheq} 
where $a^{(k)}_{p,s}$ is the $k^{\text{th}}$ iteration of solver $s$ to solve
problem $p$.
We would not have access to $\overline{x}=\Pro_{U_{1} \cap
U_{2}}x_{0}$ in applications, but we use it here to see the true performance of
the algorithms.
After collecting the related performance matrices, $\mathbf{T} = (t_{p,s})
_{100 \times \card(\mathbf{S})}$, we use the \texttt{perf.m} file in Dolan
and Mor\'e \cite{Cops} to generate the plots of performance profiles. All of
our calculations are implemented in \texttt{Matlab}.
%\textcolor{purple}{Note that in our theoretical results on circumcentered reflector methods shown in \cref{sec:CircumMethodReflectors}, we consider only  the linear convergence of circumcentered reflector methods which is related on the number of iterations and that our goal in numerical experiments is to explore the convergence rate of circumcentered reflector methods. Hence, we spend more energy on the performance profiles with measure on the number of iterations.}

\subsection{Performance evaluation} \label{subsec:PerformEvalu}
%\textcolor{purple}{In this section, we show the performance profiles comparing the related algorithms.} 
%\subsubsection*{Number of iterations as performance measure}
In this subsection, we present the performance profiles from  four experiments. 
(We ran many other experiments and the results were similar to the ones shown here.) 
The cosine of the Friedrichs angels of the four experiments are from  
$[0.01,0.05[ , [0.05,0.5[, [0.5, 0.9[$ and $[0.9,0.95[$ respectively.  In each one of the  four experiments, we  randomly generate 10 pairs of linear subspaces with the cosine of Friedrichs angles, $c_{F}$, in the corresponding range, and 
as we mentioned in the last subsection,  for each pair of subspaces, we choose randomly $10$ initial points, $x_{0}$, which gives us 100 problems in each experiment. The outputs  of every one of our four experiments are the pictures of performance profiles with performance measure shown in \cref{eq:tps:NumberIterate} (the left-hand side pictures in \Cref{fig:ProfProfSixAlgorithms_v12,fig:ProfProfSixAlgorithms_v34}) and with performance measure shown in \cref{eq:tps:Runtime} (the right-hand side ones in \Cref{fig:ProfProfSixAlgorithms_v12,fig:ProfProfSixAlgorithms_v34}) 

According to \cref{fig:ProfProfSixAlgorithms_v12}, we conclude that when $c_{F} \in [0.01,0.5[$, $\CC{\mathcal{S}_{4}}$ needs the smallest number of iterations to satisfy the inequality shown in \cref{eq:tps:NumberIterate}, that MAP is the fastest to attain the inequality shown in \cref{eq:tps:Runtime}, and that  $\CC{\mathcal{S}_{3}}$  takes the second place in terms of both required number of iterations and run time. 
Note that the circumcentered reflection methods need to solve the linear system 
 (see \cref{prop:CW:Welldefined:Formula}). 
Hence,
it is reasonable that MAP is the the fastest although MAP needs more number
of iterations than circumcentered reflection methods.

\begin{figure}[H]
	\centering
	\noindent
	\hspace*{-0.8cm}
	\begin{tabular}{ccc}
		\subfloat[]{\includegraphics[scale=0.4]{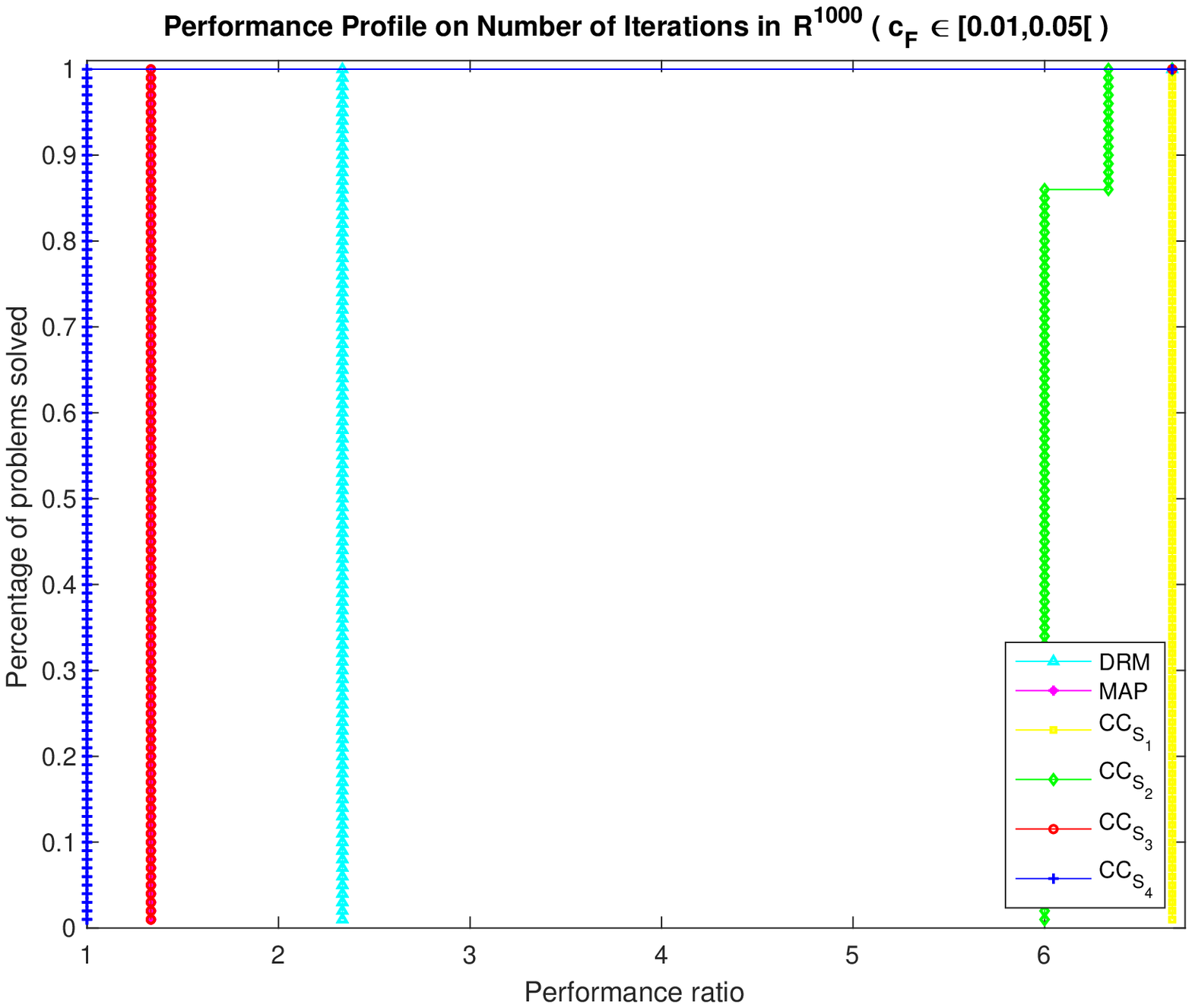}}
		\hspace*{0.5cm}
		\subfloat[]{\includegraphics[scale=0.4]{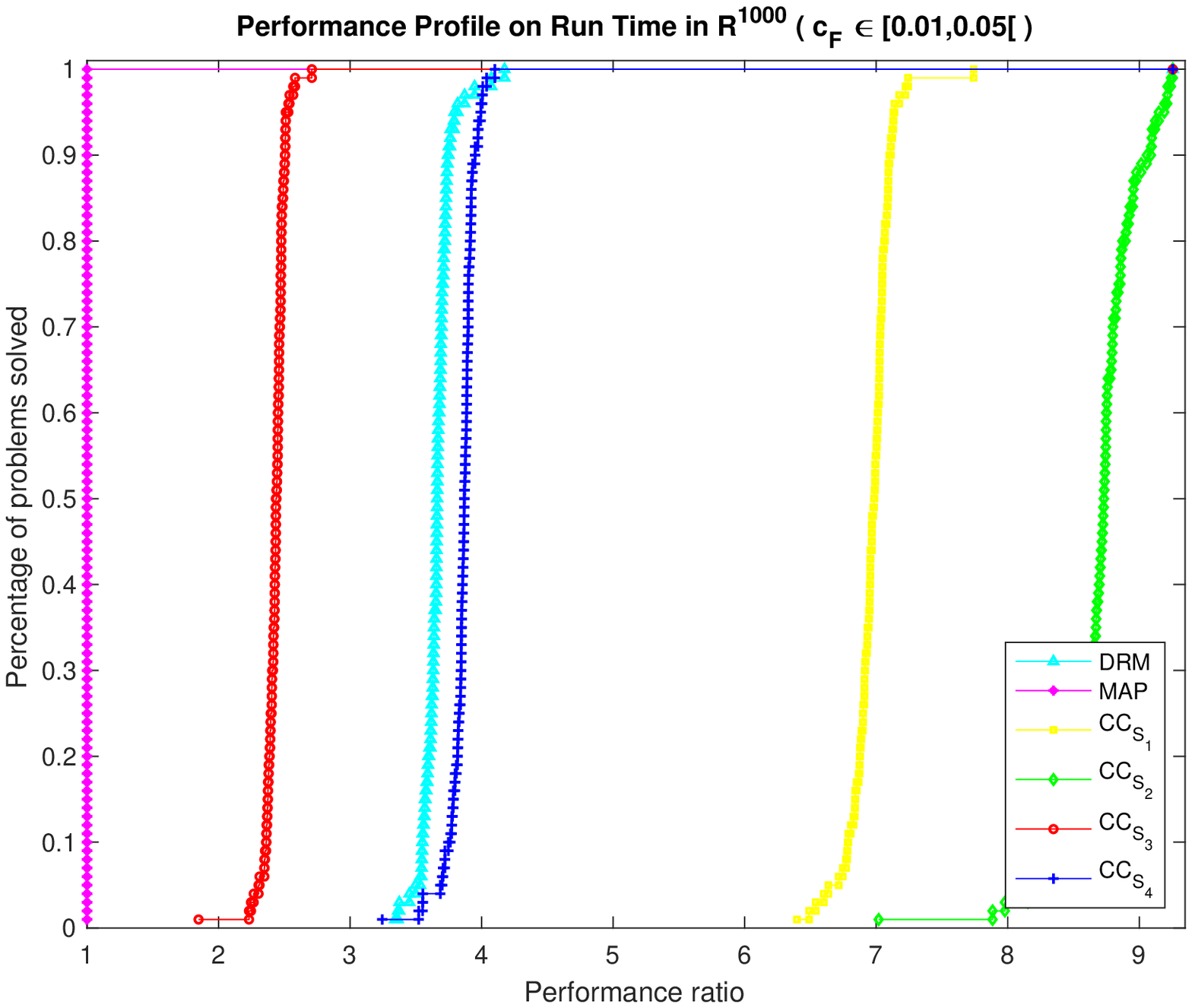}}\\
		\subfloat[]{\includegraphics[scale=0.4]{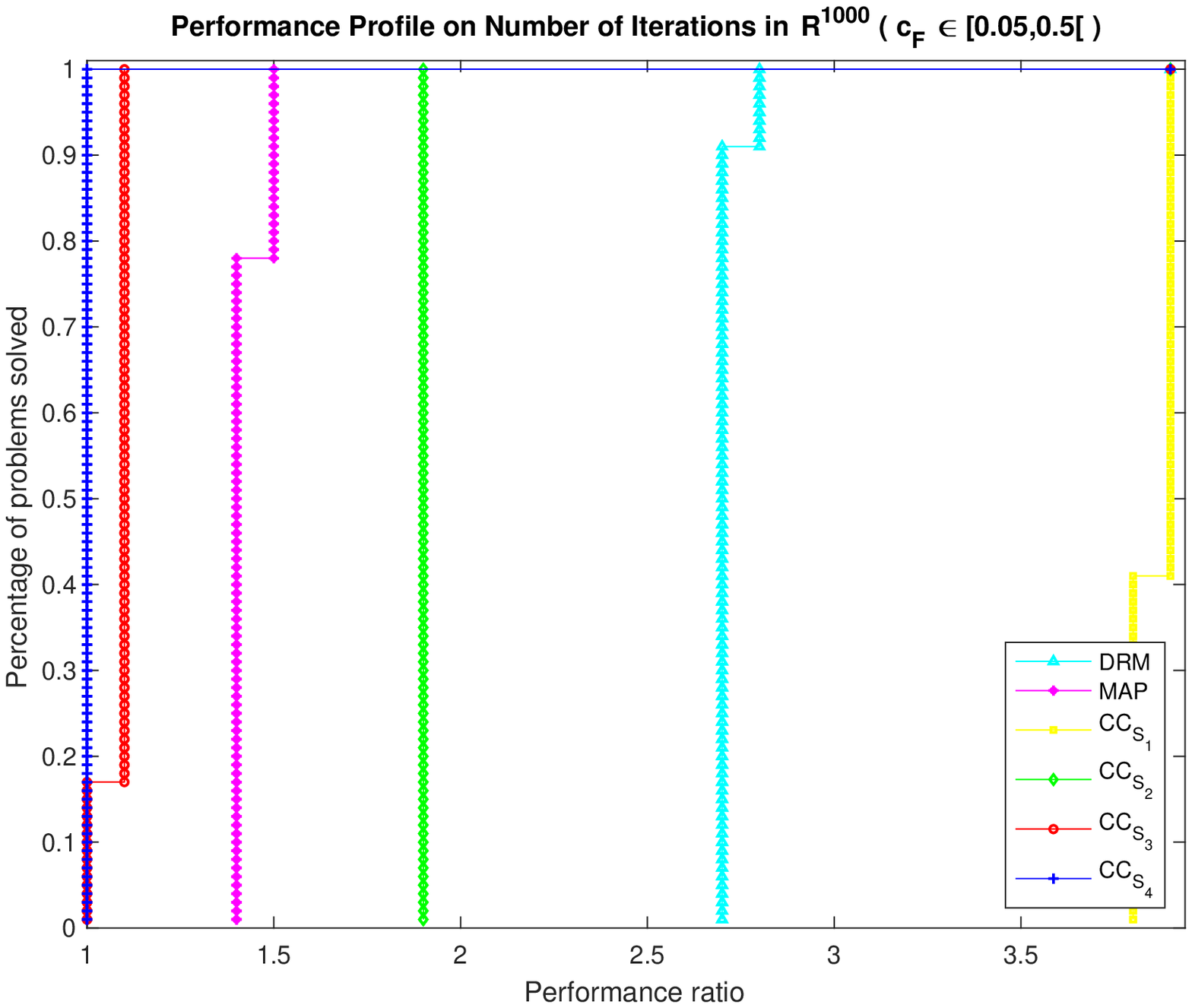}}
		\hspace*{0.5cm}
		\subfloat[]{\includegraphics[scale=0.4]{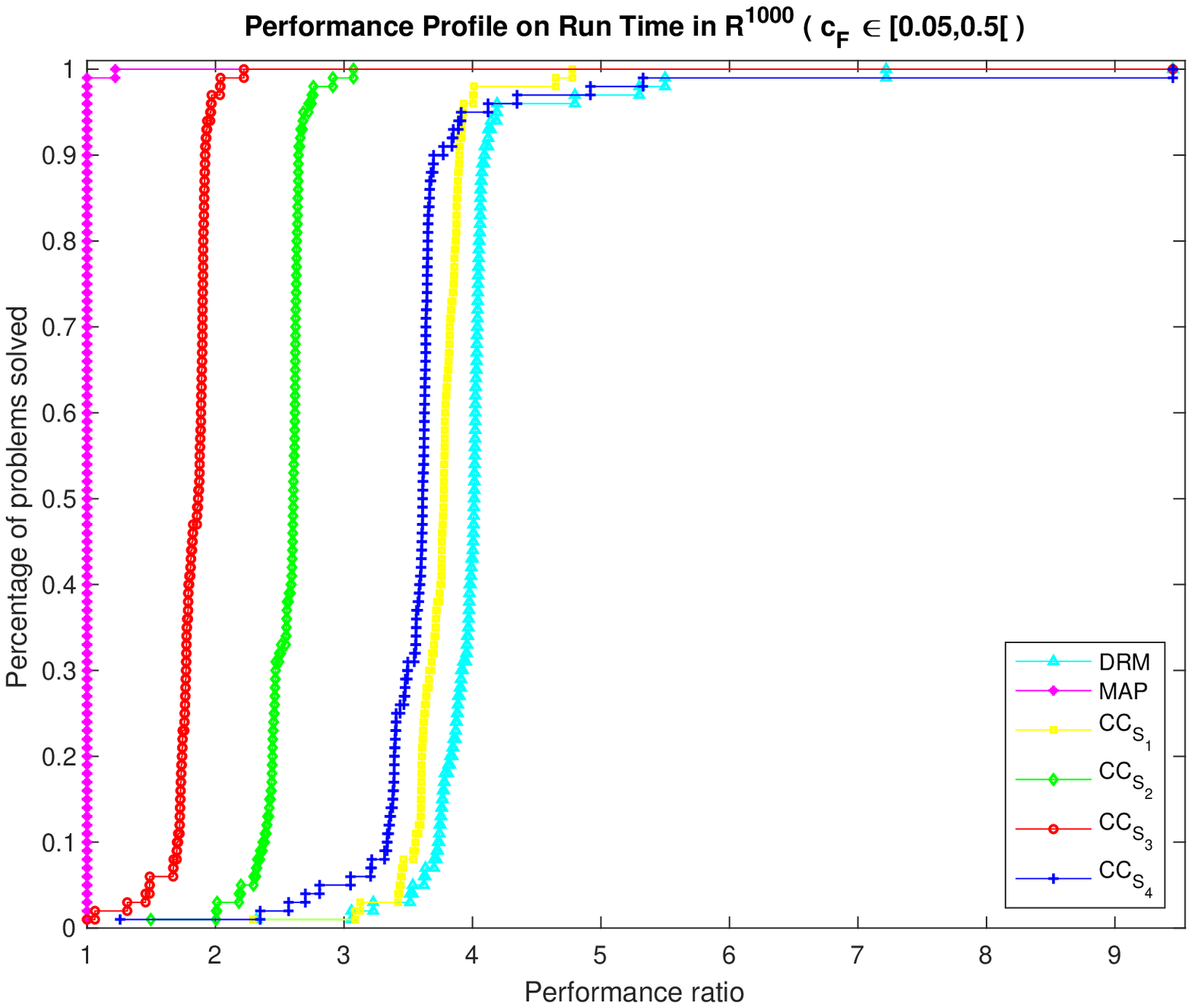}}
	\end{tabular}
	\caption{Performance profiles on six solvers for $c_{F} \in [0.01,0.5[$} \label{fig:ProfProfSixAlgorithms_v12}
	\hspace*{-0.5cm}
\end{figure}

%\begin{figure}[H]
%	\centering
%	\noindent
%	\hspace*{-0.8cm}
%	\begin{tabular}{ccc}
%		
%	\end{tabular}
%\caption{Performance profiles on six solvers for $c_{F} \in [0.05,0.5[$} \label{fig:ProfProfSixAlgorithms_v2}
%\hspace*{-0.5cm}
%\end{figure}
From \cref{fig:ProfProfSixAlgorithms_v34}(a)(b),  we know that when $c_{F} \in [0.5,0.9[$, 
the number of iterations required by $\CC{\mathcal{S}_{2}}$ and
$\CC{\mathcal{S}_{3}}$ are similar (the lines from $\CC{\mathcal{S}_{2}}$ and
$\CC{\mathcal{S}_{3}}$ almost overlap) and dominate the other 4
algorithms, and $\CC{\mathcal{S}_{2}}$ is the fastest followed closely by MAP
and $\CC{\mathcal{S}_{3}}$.
By \cref{fig:ProfProfSixAlgorithms_v34}(c)(d), we find that when $c_{F} \in [0.9,0.95[$ in which case MAP and DRM  are very slow for solving  the best approximation problem, $\CC{\mathcal{S}_{3}}$ needs the least number of iterations and is the fastest in every one of the 100 problems.   

Note that in $\mathbb{R}^{1000}$, the calculation of projections takes the
majority time in the whole time to solve the problems. As we mentioned
before, we apply the \cref{prop:CW:Welldefined:Formula} to calculate our
circumcenter mappings: $\CC{\mathcal{S}_{1}}$, $\CC{\mathcal{S}_{2}}$,
$\CC{\mathcal{S}_{3}}$ and $\CC{\mathcal{S}_{4}}$. Because the largest number
of the operators in our $\mathcal{S}$ is $ 6$ (attained for 
$\mathcal{S}_{4}$), the size of the Gram matrix in
\cref{prop:CW:Welldefined:Formula} is less than or equal $5 \times 5$. As it
is shown in \cref{fig:ProfProfSixAlgorithms_v34}(a)(c), the methods 
$\CC{\mathcal{S}_{2}}$, $\CC{\mathcal{S}_{3}}$ and $\CC{\mathcal{S}_{4}}$
need fewer iterations to solve the problems than MAP and DRM. It is
well-known that MAP and DRM are very slow when $c_{F}$ is close to 1.
% Comparing with the majority computational cost, the calculation of
% projections, the cost of solving the linear system with the size of the
% related matrix is always less than or equal $5 \times 5$ are remedied, hence,
It is not surprising that \cref{fig:ProfProfSixAlgorithms_v34}(b) shows that
$\CC{\mathcal{S}_{2}}$ is the fastest when for $c_{F} \in [0.5,0.9[$ and
\cref{fig:ProfProfSixAlgorithms_v34}(d) illustrates that
$\CC{\mathcal{S}_{3}}$ is the fastest for $c_{F} \in [0.9,0.95[$.

\begin{figure}[H]
	\centering
	\noindent
	\hspace*{-0.8cm}
	\begin{tabular}{ccc}
			\subfloat[]{\includegraphics[scale=0.4]{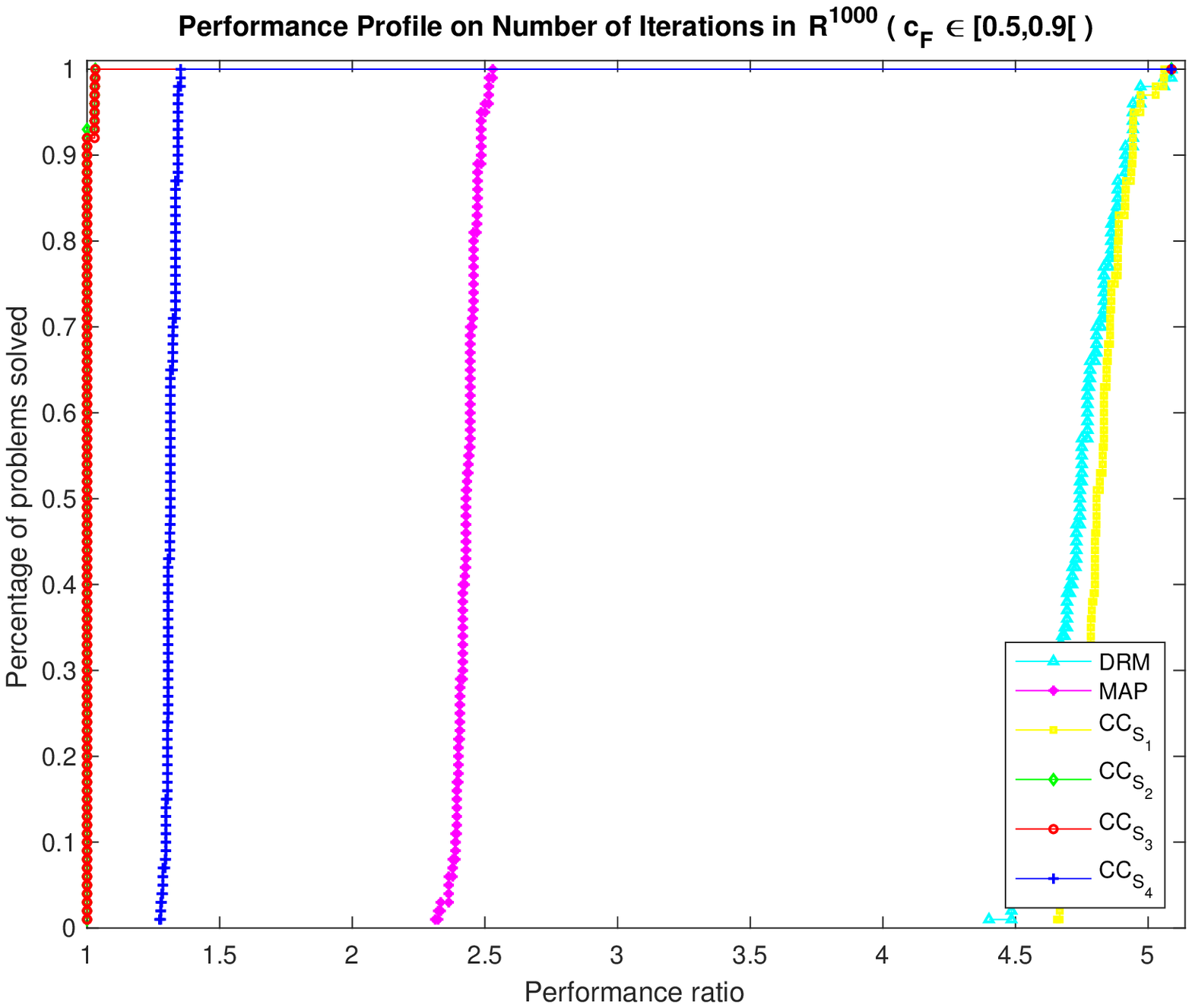}}
		\hspace*{0.5cm}
		\subfloat[]{\includegraphics[scale=0.4]{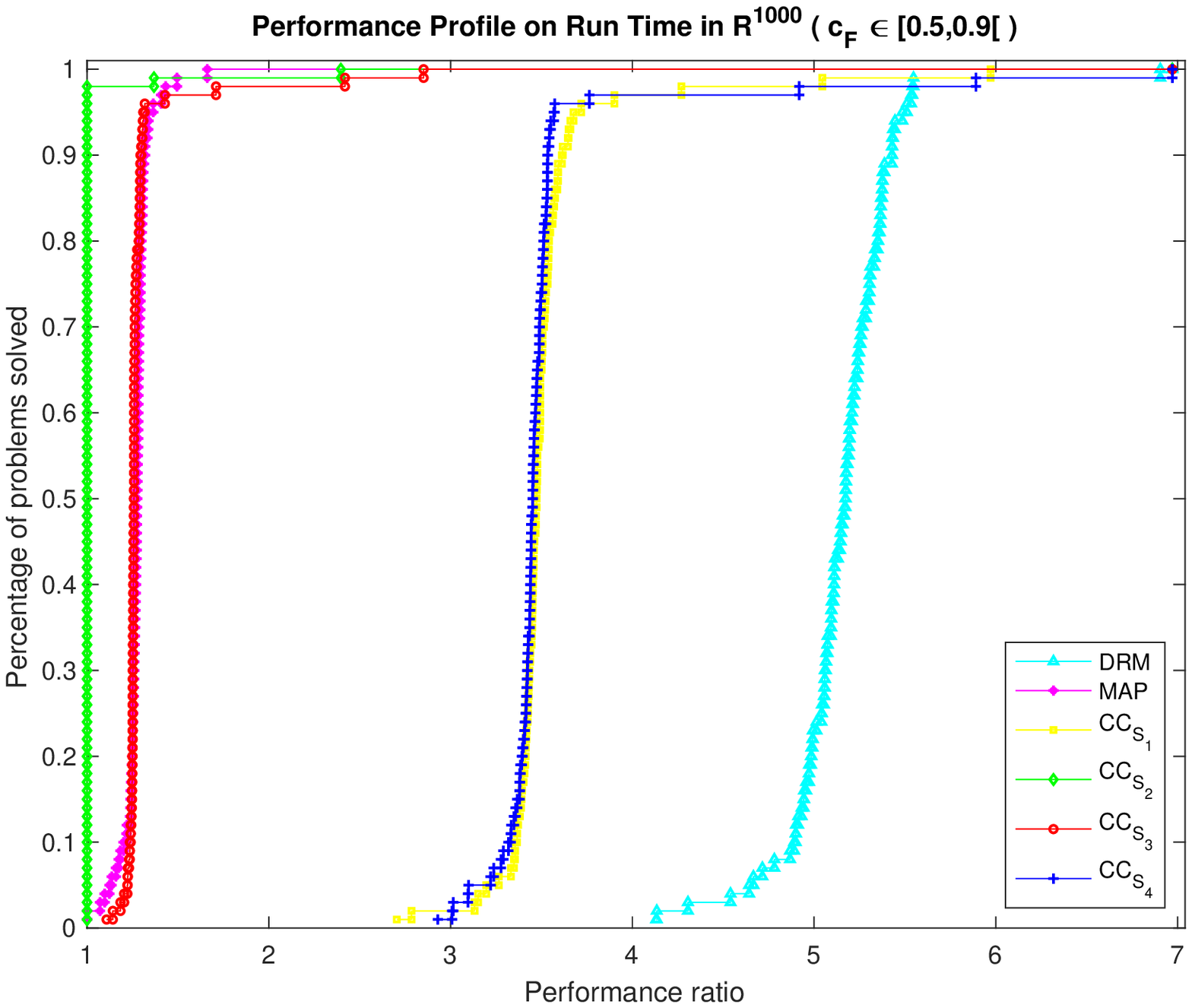}}\\
			\subfloat[]{\includegraphics[scale=0.4]{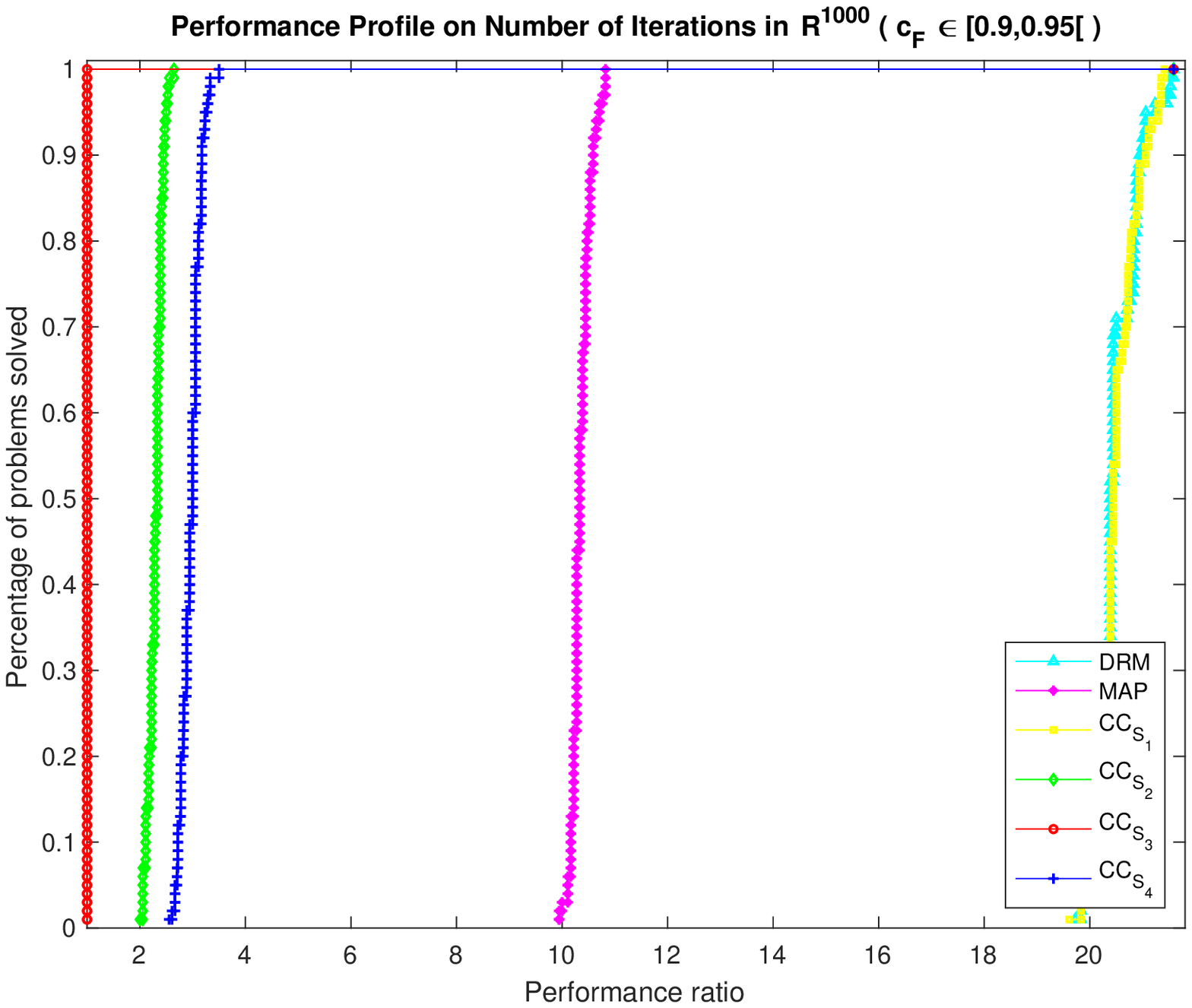}}
		\hspace*{0.5cm}
		\subfloat[]{\includegraphics[scale=0.4]{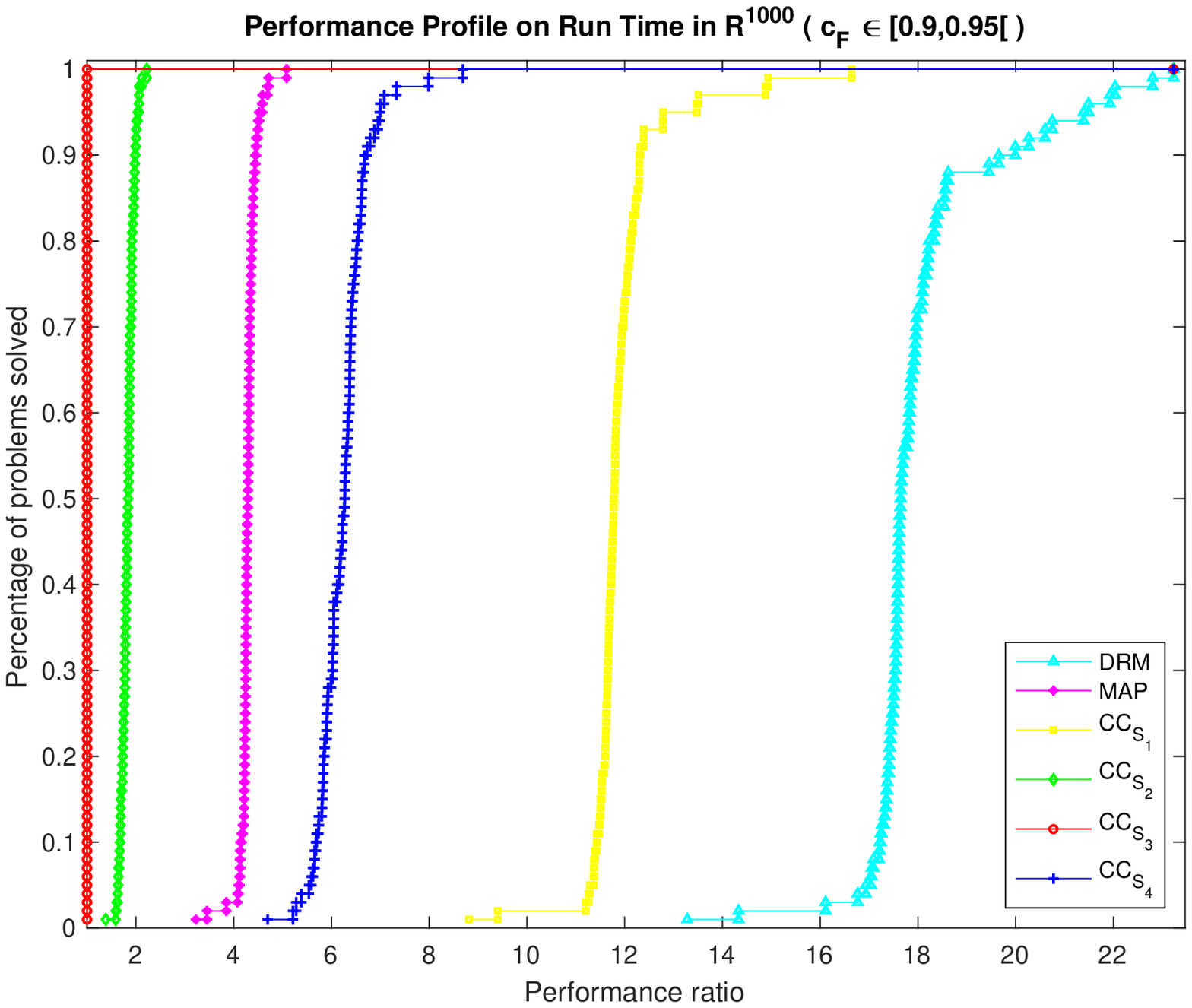}}
	\end{tabular}
	\caption{Performance profiles on six solvers for $c_{F} \in [0.5,0.95[$} \label{fig:ProfProfSixAlgorithms_v34}
	\hspace*{-0.5cm}
\end{figure}

The main conclusions that can be drawn from our experiments are the following.

When $c_{F} \in [0.01,0.5[$ is small, $\CC{\mathcal{S}_{4}}$ is the winner in
terms of number of iterations and MAP is the best solver with consideration
of the required run time. $\CC{\mathcal{S}_{3}}$ takes the second place in
performance profiles with both of the performance measures
\cref{eq:tps:NumberIterate} and \cref{eq:tps:Runtime} for $c_{F} \in
[0.01,0.5[$.

When $c_{F} \in [0,5,0.9[$, Behling, Bello Cruz and Santos' method 
$\CC{\mathcal{S}_{2}}$ is the optimal solver and the performance of
$\CC{\mathcal{S}_{3}}$ is outstanding for both
the required number of iterations and run time. 

When $c_{F} \in [0,9,0.95[$,
$\CC{\mathcal{S}_{3}}$ is the best option with regard to both required
number of iterations and run time.

Altogether, if the user does not have an idea about the range of $c_{F}$, 
then we recommend $\CC{\mathcal{S}_{3}}$.

\section{Concluding remarks}

Generalizing some of our work in \cite{BOyW2018Proper} and using the idea in
\cite{BCS2017}, we showed the properness of the circumcenter mapping induced
by isometries, which allowed us to study the circumcentered isometry methods.
Sufficient conditions for the (weak, strong, linear) convergence of the
circumcentered isometry methods were presented. In addition, we provided
certain classes of linear convergent circumcentered reflection methods and
established some of their applications. Numerical experiments suggested that
three (including the C--DRM introduced in \cite{BCS2017}) out of our
four chosen circumcentered reflection methods dominated the DRM and
MAP in terms of number of iterations for every pair of linear subspaces with the cosine of Friedrichs angle $c_{F} \in [0.01,0.95[$. Although MAP is fastest to solve the related problems when $c_{F} \in [0.01,0.5[$ and C--DRM is the fastest when $c_{F} \in [0.5,0.9[$, 
one of our new circumcentered reflection methods is a competitive 
choice when we have no prior knowledge on the Friedrichs angle $c_{F}$. 

We showed the weak convergence of certain class of circumcentered isometry
methods in \cref{theore:CM:WeakConver}. Naturally, we may ask whether 
strong convergence holds. If $\mathcal{S}$ consists of isometries and
$\cap_{T \in \mathcal{S}} \Fix T \neq \varnothing$, then
\cref{thm:CCS:proper:NormPres:T}\cref{thm:CCS:proper:NormPres:T:prop} shows
the properness of $\CC{\mathcal{S}}$. Assuming additionally that
$(\CC{\mathcal{S}}^{k}x)_{k \in \mathbb{N}}$ has a norm cluster in $\cap_{T
\in \mathcal{S}} \Fix T$, \cref{thm:CWP}\cref{prop:CWP:clust} says that
$(\CC{\mathcal{S}}^{k}x)_{k \in \mathbb{N}}$ converges to $\Pro_{\cap_{T \in
\mathcal{S}} \Fix T}x$. Another question is: Can one find more general
condition on $\mathcal{S}$ such that $\CC{\mathcal{S}}$ is proper and
$(\CC{\mathcal{S}}^{k}x)_{k \in \mathbb{N}}$ has a norm cluster in $\cap_{T
\in \mathcal{S}} \Fix T$ for some $x \in \mathcal{H}$? 
These are interesting questions to explore in future work. 

\section*{Acknowledgements}
The authors thank two anonymous referees and the editors for their
constructive comments and professional handling of the manuscript. HHB and XW
were partially supported by NSERC Discovery Grants.

\addcontentsline{toc}{section}{References}

\small 

\bibliographystyle{abbrv}

\end{document}